\documentclass[reqno]{amsart}

\usepackage[utf8]{inputenc}
\usepackage[T1]{fontenc}
\usepackage[english]{babel}
\usepackage{lmodern}
\UseRawInputEncoding

\usepackage{eurosym}

\usepackage{enumitem}
\usepackage{graphicx}
\usepackage{float}
\usepackage{amsmath}
\allowdisplaybreaks[4]
\usepackage{mathtools}
\usepackage{amssymb}
\usepackage{amsthm}
\usepackage{bbm}

\newcommand{\N}{\mathbb{N}}

\newcommand{\R}{\mathbb{R}}

\theoremstyle{plain}
\newtheorem{theorem}{Theorem}
\newtheorem{proposition}[theorem]{Proposition}

\newtheorem{lemma}[theorem]{Lemma}

\theoremstyle{definition}

\usepackage{color}
\usepackage[square,sort,comma,numbers]{natbib}
\usepackage[a4paper,bindingoffset=0.2in,%
            left=0.5in,right=0.5in,top=1in,bottom=1in,%
            footskip=.25in ]{geometry}

\usepackage{setspace}

\usepackage{fancyhdr}

\usepackage{ifthen}

\fancyheadoffset{\textwidth}
\usepackage{blindtext}
\usepackage[pdftex,pdfpagelabels=true]{hyperref}

\hypersetup{
    colorlinks=true,
    linkcolor=blue,
    filecolor=magenta,
    urlcolor=cyan,
}

\usepackage[title]{appendix}

\begin{document}

\title[]{Rigorous Derivation of the Degenerate Parabolic-Elliptic Keller-Segel System from A Moderately Interacting Stochastic Particle System.\\
Part I Partial Differential Equation }
\date{\today}
\author{Li Chen}
\address{School of Business Informatics and Mathematics, Universität Mannheim, 68131, Mannheim, Germany}
\email{li.chen@uni-mannheim.de}

\author{Veniamin Gvozdik}
\address{School of Business Informatics and Mathematics, Universität Mannheim, 68131, Mannheim, Germany}
\email{veniamin.gvozdik@uni-mannheim.de}

\author{Yue Li}
\address{Department of Mathematics, Nanjing University,
 Nanjing, 210093, P.R. China}
\email{liyue2011008@163.com}

\begin{abstract}
The aim of this paper is to provide the analysis result for the partial differential equations arising from the rigorous derivation of the degenerate parabolic-elliptic Keller-Segel system from a moderately interacting
stochastic particle system.
The rigorous derivation is divided into two articles. In this paper, we establish the solution theory of the degenerate parabolic-elliptic  Keller-Segel problem and its non-local version,
which will be used in the second paper for the discussion of the mean-field limit. A parabolic regularized system is introduced to bridge the stochastic particle model and the degenerate Keller-Segel system.
We derive the existence of the solution to this regularized system by constructing approximate solutions, giving uniform estimates
and taking the limits, where a crucial step is to obtain the $L^\infty$ Bernstein type estimate for the gradient of the approximate solution.
Based on this, we obtain the well-posedness of the corresponding
non-local equation through perturbation method. Finally, the weak solution of the degenerate Keller-Segel system is obtained by using a nonlinear version of Aubin-Lions lemma.
\end{abstract}

\maketitle

\pagenumbering{arabic}

\section{Introduction}

We consider a heuristic mathematical model from biology which describes chemotaxis, namely the collective motion of cells which are attracted or repelled by a chemical substance. From the biological view of point, there are two types of chemotaxis. If organisms move from a lower concentration towards a higher concentration of the chemical substance, it's called a positive chemotaxis.   The opposite phenomenon is called a negative chemotaxis. Substances that induce chemotaxis are called chemoattractants or chemorepelents. One of the easiest ways to model chemotaxis is  to assume that the chemotactic flux $F$ is given by
$F= \chi u \nabla c$, where $u$ is the organisms' density, $c$ is the chemical concentration and $\chi$ is a so-called chemotactic sensitivity (for more details we refer to \citep{arumugam2021keller}). If we choose  $\chi>0$ then we obtain positive chemotaxis and vice versa. Since the chemical substance  diffuses as well, one typically considers $c$ as a solution to some elliptic or parabolic equations which are coupled with $u$.

In this series of two papers, we present the rigorous derivation of the following degenerate parabolic-elliptic Keller-Segel system  in the sub-critical regime
\begin{align}
	\label{Keller_Segel_classical}
	\begin{cases}
		\partial_t u = - \nabla \cdot (\chi u \nabla c) + \nabla \cdot (u \nabla p(u)),    \\
		-\Delta c= u(t,x), \\
		u(0,x)=u_0(x), \ \ \ \ x\in \R^d,\;\; t>0,
	\end{cases}
\end{align}
where $u$ and $c$ stand for the density of cells and chemical substance respectively and $p(u)=\frac{m}{m-1}u^{m-1}$ for a given $m\geq 1$. For simplicity we take the chemotactic sensitivity $\chi=1$.

The problem \eqref{Keller_Segel_classical} will be derived from a system of stochastic differential equations,
which describes the movements of $N$ particles in $\R^d$ ($d\geq 2$).
Let $(\Omega, \mathcal{F}, (\mathcal{F}_{t\geq 0}) , \mathbb{P})$ be a complete filtered probability space.
We consider $d$-dimensional $\mathcal{F}_t$-Brownian motions $\{ (B_t^i)_{t\geq 0} \}_{i=1}^N$ which are assumed to be independent of each other. Involving a parabolic regularization with diffusion coefficient $\sigma>0$, the dynamics of $N$ particles $(X_t^{N,i, \varepsilon, \sigma})_{1\leq i\leq N}$ is governed by
\begin{align}
\label{generalized_regularized_particle_model}
\begin{cases}
dX_t^{N,i, \varepsilon, \sigma} = \frac{1}{N} \sum_{j =1}^N \nabla \Phi^{\varepsilon_k}(X_t^{N,i, \varepsilon, \sigma} -X_t^{N,j, \varepsilon, \sigma}) dt - \nabla p_{\lambda} \Big( \frac{1}{N} \sum_{j=1}^N  V^{\varepsilon_p}(X_t^{N,i, \varepsilon, \sigma} -X_t^{N,j, \varepsilon, \sigma} ) \Big) dt + \sqrt{2\sigma }  \,dB_t^i, \\
X_0^{N,i, \varepsilon, \sigma} =\zeta^i,
\end{cases}
\end{align}
where $\{\zeta_i\}_{i=1}^N$ are given i.i.d. random variables, independent of $\{ (B_t^i)_{t\geq 0} \}_{i=1}^N$ and have a common density function $u_0^\sigma$.
Here $(u_0^\sigma)_\sigma\subset C_0^\infty(\R^d)$ is a sequence which approximates $u_0\in L^1(\mathbb{R}^d)\cap L^\infty(\mathbb{R}^d)$ in the following sense
\begin{eqnarray}\label{app_u0_sigma}
	u_0^\sigma\rightarrow u_0 \mbox{ in } L^l((0,T)\times\R^d), \forall l\in [1,\infty) \mbox{ as } \sigma\rightarrow 0, \mbox{ and } \|u_0^\sigma\|_{L^q((0,T)\times\R^d)}\leq \|u_0\|_{L^q((0,T)\times\R^d)}, \forall q\in [1,\infty].
\end{eqnarray}
In the stochastic system \eqref{generalized_regularized_particle_model} we assume that $ 0\leq V (|x|) \in C_0^{\infty} (\R^d)$ and satisfies $\int _{\R^d}V(x) dx =1$. We use standard mollification kernel $V^{\varepsilon} (x) := \frac{1}{\varepsilon^d } V(x / \varepsilon)$ to approximate the fundamental solution $\Phi$ of the Laplace equation  $\Phi^{\varepsilon_k}:= \Phi * V^{\varepsilon_k}$, and use $V^{\varepsilon_p}$ as the kernel of the moderate interaction. Notice that we use different mollification radius $\varepsilon_k$ and $\varepsilon_p$ to approximate the aggregation kernel and the moderate interaction in diffusion. Furthermore, we use $p_{\lambda}$  as an approximation of $p$ (for $\lambda\rightarrow 0$) such that $p_{\lambda} \in C^3(\R_+)$ and
\begin{align}
	\label{p_lambda_definition}
	p_{\lambda}(r)  = \begin{cases}
		p(2/\lambda), \ &\text{if} \ r>2/ \lambda,  \\
		p(r), \ & \text{if} \ 2\lambda <r< 1/ \lambda,\\
		p(\lambda), \ & \text{if} \ 0<r<\lambda.
	\end{cases}
\end{align}

The so-called Keller-Segel model was introduced by \citep{keller1970initiation}, \citep{keller1971model} and \citep{patlak1953random}, which nowadays has many different modifications. The Keller-Segel systems are of diffusion aggregation type. It is shown in the literature, for example in \cite{blanchet2006two}, that the linear diffusion competes with the aggregation on the same level, while for higher dimensions the aggregation effect of the Keller-Segel system is stronger than the diffusion effect, therefore nonlinear diffusion has been introduced in order to balance the aggregation. One typical example is \eqref{Keller_Segel_classical}, where the nonliear diffusion is of the porous medium type.

Moreover, Keller-Segel system \eqref{Keller_Segel_classical} has the mass conservation property, namely
\begin{align*}
M_0:= \int_{\R^d} u_0(x) dx = \int_{\R^d} u(t,x) dx \;\;\; \text{for} \;\;\;  t\geq 0.
\end{align*}
It is well-known that solution's behavior  depends on  the quantity $M_0$. Choosing $d=2$ and $m=1$ in \eqref{Keller_Segel_classical} one obtains a  system of partial differential equations which  corresponds  to the original work of   \citep{keller1970initiation}. The authors of \citep{blanchet2006two} proved that for the sub-critical case, namely $0< M_0< 8\pi$,  the solution exists globally  and  for the super-critical case, that is $M_0> 8 \pi$, there is a blow up in finite time.
Later using the free-energy method  the  authors of \citep{blanchetinfinite} considered the critical case, namely $M_0=8 \pi$, and  proved that the solutions of  \eqref{Keller_Segel_classical} for $d=2$ and $m=1$ exist globally and blow up as a Dirac delta at the center of mass when time goes to infinity.
The behavior of the solution to \eqref{Keller_Segel_classical} for $d\geq 3$ and $m=2-\frac{2}{d} \in (1,2)$ was studied in \citep{blanchet2009critical}.  The authors of \citep{blanchet2009critical} obtained a critical mass $M_c$ and studied  behavior of the solution to \eqref{Keller_Segel_classical}  for different $M_0$. Authors of  \citep{luckhaus2006large}, \citep{sugiyama2006global}, \citep{sugiyama2007application}, \citep{sugiyama2007time} and \citep{sugiyama2006global_power} proved  that if $m> 2-\frac{2}{d}$ then the uniqueness and existence of the solution to  \eqref{Keller_Segel_classical}  hold for any initial data and if $1<m\leq  2-\frac{2}{d}$ then one observers the solutions blow up for some large initial data.
It was proposed in \citep{chen2012multidimensional} that the nonlinear diffusion term with $m_c= \frac{2d}{d+2}$,
which is smaller than the usual exponent $2-\frac{2}{d}$,
is also important for Keller-Segel system from many points of view. In this case, the system has a family of stationary solutions which are quite similar to the two dimensional case from \citep{keller1971model}, which is also the sharp profile in Hardy-Littlewood-Sobolev inequality. Moreover, the associated free entropy is conformal invariant.  For $ \frac{2d}{d+2} < m < \frac{2(d+1)}{d}$  a clear criterion for the  classification of  initial data is given in \citep{wang2016parabolic}. We refer to   \citep{wang2016parabolic} for detailed information about these two critical exponents, namely, the exponent $\frac{2(d-1)}{d}$  which comes from the scaling invariance of the mass, and the exponent $\frac{2d}{d+2}$  which comes from the conformal invariance of the entropy.
The aim of the series of these two papers is to derive rigorously \eqref{Keller_Segel_classical} from \eqref{generalized_regularized_particle_model}. Due to the moderate interaction and the singular interaction potential $\Phi$, the whole analysis is obtained in the subcritical regime of $m$ where the regular solution exists.

We introduce an intermediate particle problem, which is formally viewed as a mean-field limit $N\rightarrow \infty$ in the system \eqref{generalized_regularized_particle_model} for fixed
$\varepsilon_k,\varepsilon_p>0$, namely
\begin{align}
\label{generalized_intermediate_particle_model}
\begin{cases}
d\bar{X}_t^{i, \varepsilon, \sigma} = \nabla  \Phi^{\varepsilon_k}  *u^{\varepsilon, \sigma} ( t, \bar{X}_t^{i, \varepsilon, \sigma}) dt - \nabla p_{\lambda} \big(  V^{\varepsilon_p} * u^{\varepsilon, \sigma} (t, \bar{X}_t^{i, \varepsilon, \sigma}  ) \big) dt +\sqrt{2\sigma } dB_t^i ,\\
\bar{X}_0^{i, \varepsilon, \sigma} =\zeta^i,
\end{cases}
\end{align}
where $ u^{\varepsilon, \sigma}(t, x)$ is the probability density function of $\bar{X}_t^{i, \varepsilon, \sigma}$ and solves the following non-local partial differential equations
\begin{align}
\label{generalized_equation_u_epsilon_sigma}
\begin{cases}\partial_t u^{\varepsilon, \sigma} =  \sigma \Delta u^{\varepsilon, \sigma} -\nabla ( u^{\varepsilon, \sigma} \nabla c^{\varepsilon, \sigma}) + \nabla \cdot (u^{\varepsilon, \sigma} \nabla p_{\lambda} (V^{\varepsilon_p} *u^{\varepsilon, \sigma} ) ),     \\
-\Delta c^{\varepsilon, \sigma} = V^{\varepsilon_k}* u^{\varepsilon, \sigma} ( t,x ),  \\
u^{\varepsilon, \sigma}(0,x)=u^\sigma_0(x), \ \ \ \ x\in \R^d, t>0.
\end{cases}
\end{align}
With the help of the intermediate particle problem, it is proved in \cite{CGHL} that the solution of \eqref{generalized_regularized_particle_model} converges to the solution of its McKean-Vlasov problem,
\begin{align}
\label{generalized_particle_model}
\begin{cases} d\hat{X}_t^{i, \sigma} =   \nabla  \Phi  *u^{\sigma} ( t, \hat{X}_t^{i, \sigma})dt - \nabla p  (u^{ \sigma} (t, \hat{X}_t^{i,  \sigma}  ) )dt +\sqrt{2\sigma } dB_t^i,  \\
\hat{X}_0^{i, \sigma} =\zeta^i,
\end{cases}
\end{align}
where $ u^{\sigma}(t, x)$ is the probability density function of $\hat{X}_t^{i, \sigma}$ and solves the following partial differential equations
\begin{align}
	\label{generalized_equation_u_sigma}
	\begin{cases}\partial_t u^{ \sigma} = \sigma \Delta u^{\sigma} -\nabla \cdot  ( u^{ \sigma} \nabla c^\sigma)
		+ \nabla  \cdot  (u^{ \sigma} \nabla p (u^{ \sigma}) ),  \\
		-\Delta c^{ \sigma} = u^{ \sigma} ( t,x ),  \\
		u^{ \sigma}(0,x)=u^\sigma_0(x), \ \ \ \ x\in \R^d, t>0.
	\end{cases}
\end{align}

If we consider \eqref{generalized_equation_u_sigma} and take the limit $\sigma \rightarrow 0$ on the PDE level, we can obtain system of partial differential equations \eqref{Keller_Segel_classical}.
In this paper, we only focus on the analysis of partial differential equations \eqref{generalized_equation_u_epsilon_sigma}, \eqref{generalized_equation_u_sigma}, \eqref{Keller_Segel_classical}, and the convergence among them.
More details about the mean-field result for this problem can be found in \cite{CGHL}.
Our goal of this paper is to establish the existence, uniqueness and regularity of  $u^{\sigma}$ and $u^{\varepsilon, \sigma}$, and furthermore the existence of a weak solution $u$ for problem \eqref{Keller_Segel_classical}.

The well-posedness result for \eqref{generalized_equation_u_sigma} is given in the following theorem
\begin{theorem}
\label{theorem_u_sigma}
For any given $T>0$, let $u_0^\sigma$ be given by \eqref{app_u0_sigma}, then the system \eqref{generalized_equation_u_sigma} possesses a unique global solution $u^{\sigma}$ such that  $u^{\sigma} \in L^{\infty} ( 0,T; W^{1, \infty} (\R^d)  ) \cap L^{\infty}(0,T; L^1(\R^d ))$ and $u^{\sigma} \in W_q^{3,1}((0,T)\times \R^d)$ for any $1<q<\infty$.
\end{theorem}
In this paper we use the notation
\begin{align*}
W_q^{3,1}((0,T)\times \R^d):=\big\{ f\in L^q(0,T;W^{3,q}(\R^d) ) \cap W^{1,q}(0,T;L^q(\R^d) )  \ \big| \  \| f \|_{W_q^{3,1}((0,T)\times \R^d)}< \infty \big\},
\end{align*}
where, with the notation of multi-index $\alpha=(\alpha_1,\ldots,\alpha_d)$ ($|\alpha|=\sum_{1\leq i\leq d}\alpha_i$), the corresponding norm is given by
\begin{align*}
\| f \|_{W_q^{3,1}((0,T)\times \R^d)}:=& \sum_{|\alpha|\leq 3} \|D^\alpha f\|_{L^q((0,T)\times \R^d)} + \| \partial_t f \|_{L^q((0,T)\times \R^d)}.
\end{align*}

Since the system \eqref{generalized_equation_u_sigma} contains aggregation and diffusion terms, which are both non-linear, we need to perform a two-step
fixed point argument. Using classical tools from the theory of linear PDEs, such as  Schauder fixed-point theorem and Banach fixed-point theorem, we can construct a unique local approximate solution to the problem \eqref{generalized_equation_u_sigma}.
The key point of establishing global well-posedness of this approximate solution is to derive $L^\infty$ estimate of the gradient of approximate solution. We adapt the Bernstein type estimates from Lemma 13 in \citep{sugiyama2007time}, which is very technical and involved.
Substantial difference of our approach from the estimates given in \citep{sugiyama2007time} is that the whole procedure should not depend on the parabolic regularity from the nonlinear diffusion term $\nabla\cdot(u^\sigma\nabla p(u^\sigma))$, but should depend on the viscosity term $\sigma \Delta u^\sigma $.
More details can be found in Section \ref{final}. This additional Bernstein type estimate makes it possible to proceed the regularity argument for $u^\sigma$ from the theory of linear parabolic equations by using iteration scheme, for example Theorem 9.2.2 of \citep{wu2006elliptic} can be directly applied.

The second main result of this paper is the well-posedness of the non-local equations \eqref{generalized_equation_u_epsilon_sigma} for $u^{\varepsilon,\sigma}$ and the error estimate of $u^\sigma-u^{\varepsilon,\sigma}$.
\begin{theorem}
 \label{theorem_u_epsilon_sigma}
For any given $T>0$, let $u_0^\sigma$ be given by \eqref{app_u0_sigma}, $s> \frac{d}{2} +2 $, $m \in    \big( \N \cap \big( 0, \frac{d}{2} +3 \big]  \big) \cup \big(\frac{d}{2} +3, \infty\big)$, and $u^\sigma$ be obtained in Theorem \ref{theorem_u_sigma}. Then there exist $\varepsilon_0, \lambda_0>0$ such that for $\varepsilon_k$, $\varepsilon_p\leq \varepsilon_0$ and $\lambda\leq \lambda_0$, a unique global solution $u^{\varepsilon, \sigma} \in L^{\infty}(0,T; H^s(\R^d))$ of \eqref{generalized_equation_u_epsilon_sigma} exists and
\begin{align}\label{a1}
\|u^{\sigma} - u^{\varepsilon, \sigma}\|_{L^{\infty} (0,T; H^s (\R^d) ) }  \leq C(\varepsilon_k + \varepsilon_p),
\end{align}
where $C$ is a positive constant independent of $\varepsilon_k$, $\varepsilon_p$ and $\lambda$.
\end{theorem}
In general, it is hard to obtain well-posedness results for non-local problems without assuming that the initial data are small enough. Due to the fact that the non-local problem \eqref{generalized_equation_u_epsilon_sigma} is used as an approximation, it is enough to deal with it for small $\varepsilon_k$, $\varepsilon_p$, and $\lambda$. Therefore, it is possible to handle arbitrarily given initial data. The key idea is to study the time evolution of $u^{\sigma} - u^{\varepsilon, \sigma}$. Together with the regularity we obtained for $u^\sigma$, we derive the $H^s$ energy estimate for $u^{\sigma} - u^{\varepsilon, \sigma}$ and obtain not only the well-posedness of \eqref{generalized_equation_u_epsilon_sigma} but also the error estimate shown in Theorem \ref{theorem_u_epsilon_sigma} simultaneously.

Finally, we obtain the compactness of $(u^\sigma)_\sigma$ and show that the limit of a sub-sequence is the weak solution to the problem \eqref{Keller_Segel_classical}.
\begin{theorem}
\label{lsigma}
For any given $T>0$, let $0\leq u_0 \in  L^1(\R^d)\cap L^\infty(\R^d)$, $ m =2 \ \text{or} \ m\geq 3$, and $u^\sigma$ be obtained in Theorem \ref{theorem_u_sigma}.
Then there exist a subsequence $(u^\sigma)_\sigma$ (not relabeled) and a function $u$ with
$u\in L^\infty(0,T;L^1(\R^d)\cap L^\infty(\R^d))$ and $u^m\in L^2(0,T;H^1(\R^d))$, such that
\begin{align*}
u^\sigma\rightarrow u\;\;\;{\rm{in}}\;\;\;L^{2m}(0,T;L^{2m}(\R^d)).
\end{align*}
Furthermore, $u$ satisfies the problem \eqref{Keller_Segel_classical} in the sense of distribution.
\end{theorem}

Due to the degeneracy of $\nabla (u^\sigma)^m$, it is difficult to obtain uniform estimates in $\sigma$ for $\nabla u^\sigma$.
Therefore, instead of using the classical Aubin-Lions lemma, we apply a nonlinear version of it from \cite{CJJ} to derive the strong compactness of $(u^\sigma)_\sigma$.

Existence of a weak solution for degenerate parabolic-elliptic Keller-Segel system \eqref{Keller_Segel_classical} in the subcritical regime has been obtained in \citep{sugiyama2007time}.
In this sense, we reprove the same result by using an alternative approximation scheme. However, this is only a by-product of the current paper. The analysis for three problems \eqref{Keller_Segel_classical}, \eqref{generalized_equation_u_sigma}, and \eqref{generalized_equation_u_epsilon_sigma} are required for the rigorous derivation of \eqref{Keller_Segel_classical}. Especially the results for regular solutions of \eqref{generalized_equation_u_sigma} and the non-local problem \eqref{generalized_equation_u_epsilon_sigma} can not be covered by any known results.

For simplicity we present the proofs only for $d\geq 3$, one can adapt the estimates, where Sobolev's inequalities are used, for the case $d=2$ and obtain easily the same results. Additionally, in this paper we denote $C$ as a generic constant which might change line to line.

The rest of the paper is organized as follows. Section \ref{final} is devoted to the analysis of problem \eqref{generalized_equation_u_sigma}, namely the proof of Theorem \ref{theorem_u_sigma}.
In Section \ref{section_u_epsilon_sigma}, we use the perturbation method to complete the proof of Theorem \ref{theorem_u_epsilon_sigma}.
Based on this, we study the compactness in the weak formulation of problem \eqref{generalized_equation_u_sigma} with respect to $\sigma$ to finish the proof of Theorem \ref{lsigma}
in section \ref{section_u}.

\section{Well-posedness of \eqref{generalized_equation_u_sigma}}
\label{final}
The well-posedness of problem \eqref{generalized_equation_u_sigma} is derived in this section. A sequence of approximated solution of \eqref{generalized_equation_u_sigma} is constructed by a  two-step fixed point argument. Afterwards, uniform estimate through a modified Bernstein type estimate for the space gradient of the solution is preformed in order to do further $H^s$ energy estimates. These provide the extension of local solution to global solution. In the end, the solvability of problem \eqref{generalized_equation_u_sigma} is obtained by compactness argument and further regularity estimates.

\subsection{Approximate solutions to the system \eqref{generalized_equation_u_sigma}. }
Our goal of this subsection is to establish the local wellposedness of the following approximate problem of the system \eqref{generalized_equation_u_sigma}:
\begin{align}
\label{u_sigma_approximation}
\begin{cases}\partial_t u^{\sigma}_{\eta} =  \nabla \cdot \left( \sigma  \nabla u^{\sigma}_{\eta} -u^{\sigma}_{\eta} \nabla \Phi * u^{\sigma}_{\eta} +  \nabla (u^{\sigma}_{\eta}+\eta)^m \right), \\
u^{\sigma}_{\eta}(0,x)=u_0^\sigma(x),\;\;\;\;x\in \R^d,\; t>0,
\end{cases}
\end{align}
where $0<\eta \leq M:= 2\|u_0^\sigma\|_{H^s(\R^d)}+1$.
The main result of this subsection is the following proposition.
\begin{proposition}\label{prop_local_approximate}
Let the assumptions in Theorem \ref{theorem_u_sigma} hold, then there exists $0<T'<\infty$ such that the approximate problem \eqref{u_sigma_approximation} possesses a unique solution $u^\sigma_\eta\in X_{T'}:= \big\{0\leq u \in L^\infty(0,T'; H^s(\R^d))  : \|u \|_{L^{\infty}(0,T'; H^s(\R^d))} \leq M,\quad s \in \big(\frac{d}{2} +3, \infty \big) \cap \N \big\}$, which satisfies
\begin{align}\label{add10}
\sup_{t\in(0,T')} \| u^{\sigma}_{\eta} \|_{L^q(\R^d)} \leq C,\;\;\; q\in [1,\infty],
\end{align}
where $C$ is a positive constant independent of $\sigma$ and $\eta$.
\end{proposition}

\begin{proof}

The proof  will be given by a two-step fixed point argument.

We will first prove that for a fixed $\xi \in X_{T_1}$ ($T_1$ to be determined), there exists a unique local solution $w\in L^{\infty}(0,T_1; H^s(\R^d))  \cap  L^2(0,T_1; H^{s+1}(\R^d))$
to the following system
\begin{align}
	\label{u_delta_xi_equation}
	\begin{cases}\partial_t w =  \nabla \cdot \left( \sigma  \nabla  w -w \nabla \Phi * \xi +  \nabla (w+\eta)^m \right), \\
		w(0,x)=u_0^\sigma(x), \qquad x\in \R^d, t>0.
	\end{cases}
\end{align}

This is done by a fixed point argument in the following:

At the beginning, we introduce a space
\begin{align}
\label{space_Y_T}
Y_T:= &\big\{0\leq u \in L^4(0,T; H^s(\R^d))  \cap  L^{\infty}(0,T; H^{s-1}(\R^d)): \|u \|_{L^4(0,T; H^s(\R^d))  \cap L^{\infty}(0,T; H^{s-1}(\R^d))} \leq M\big\}. \end{align}

For fixed $\xi \in X_T$ and $\rho \in Y_T$, we consider the following linear equation
\begin{align}
\label{u_delta_xi_rho_equation}
\begin{cases}\partial_t w =  \nabla \cdot \left( \sigma  \nabla  w -w \nabla \Phi * \xi +  m (\rho+ \eta)^{m-1} \nabla w \right), \\
w(0,x)=u_0(x),\qquad x\in \R^d, t>0.
\end{cases}
\end{align}
By the theory of parabolic differential equations, it follows that there exists a weak solution $w$ to \eqref{u_delta_xi_rho_equation}.
In addition, it is easy to see that $w$ is non-negative.
Indeed, we multiply \eqref{u_delta_xi_rho_equation} by $w^-:=\min\{0,w \}$ and integrate it over $\R^d$ to obtain that
\begin{align*}
\frac{1}{2} \frac{d}{dt} \| w^- \|^2_{L^2(\R^d)}
&\leq -\frac{1}{2}  \int_{\R^d} \left(\sigma +m (\rho+ \eta)^{m-1} \right) |\nabla w^- |^2 dx + \frac{1}{2 \sigma} \|\nabla \Phi * \xi \|_{L^{\infty}(0,T; L^{\infty} (\R^d)) }^2 \int_{\R^d}  |w^-|^2 dx\\
&\leq \frac{1}{2 \sigma} \|\nabla \Phi * \xi \|_{L^{\infty}(0,T; L^{\infty} (\R^d) ) }^2 \int_{\R^d}  |w^-|^2 dx
\leq \frac{CM}{2\sigma}\int_{\R^d}|w^-|^2dx,
\end{align*}
where we have used the fact that $\| \nabla \Phi * \xi \|_{L^{\infty}(0,T; H^s(\R^d))} \leq C \|  \xi \|_{L^{\infty}(0,T; H^s(\R^d))}  \leq CM$.
Combining it with the Gr\"{o}nwall's inequality, we deduce that
\begin{align*}
\underset{t\in (0,T')}{\sup} \| w^- (t,\cdot)\|_{L^2(\R^d)} \leq 0.
\end{align*}

Multiplying  \eqref{u_delta_xi_rho_equation} by $w$ and integrating it over $\R^d$, we obtain that
\begin{align}
\label{estim_w_l_2}
\frac{1}{2} \frac{d}{dt} \int_{\R^d} w^2 dx + \sigma \int_{\R^d} |\nabla w|^2 dx + m \int_{\R^d} (\rho+ \eta)^{m-1} |\nabla w|^2 dx
\leq \frac{\sigma}{2}  \int_{\R^d} |\nabla w|^2 dx + C M^2 \int_{\R^d}  |w|^2 dx,
\end{align}
where $C$ is a positive constant which depends on $\sigma$, $\eta$, $m$, $d$ and $s$.
We take the $D$ derivative of \eqref{u_delta_xi_rho_equation}, multiply the result equality by $D w$ and integrate it over $\R^d$, we deduce that
\begin{align}
\label{estim_w_der_l_2}
&\frac{1}{2} \frac{d}{dt} \int_{\R^d} |  D  w |^2 dx + \sigma  \int_{\R^d} |\nabla D  w |^2 dx
 + m \int_{\R^d}  (\rho+ \eta)^{m-1}   | \nabla  D   w|^2 \, dx\nonumber\\
\leq& \frac{\sigma}{2} \int_{\R^d} |\nabla D  w|^2 dx + C M^{2(m-1)}   \int_{\R^d} |\nabla w|^2 dx
+CM^2\int_{\R^d}|w|^2dx.
\end{align}

We let $\alpha$ be a multi-index such that $2\leq |\alpha| \leq s$ and get
\begin{align*}
&\frac{1}{2} \frac{d}{dt} \int_{\R^d} | D^{\alpha}w |^2 dx + \sigma  \int_{\R^d} |\nabla D^{\alpha}w |^2 dx \\
=&   \int_{\R^d}   D^{\alpha} (w \nabla \Phi * \xi ) \cdot\nabla D^{\alpha} w \, dx
 - m \int_{\R^d}  D^{\alpha} ((\rho+ \eta)^{m-1} \nabla w)  \cdot \nabla D^{\alpha} w \, dx
=:I_1+I_2.
\end{align*}

Now we are going to estimate the terms on the right-hand side. For the first term $I_1$
\begin{align*}
I_1
&\leq \frac{\sigma}{4} \int_{\R^d} | \nabla D^{\alpha} w|^2 \, dx + \frac{1}{\sigma} \int_{\R^d} |  D^{\alpha} (w \nabla \Phi * \xi )|^2 dx\\
&\leq   \frac{\sigma}{4} \int_{\R^d} | \nabla D^{\alpha} w|^2 \, dx + C \left(\|w \|_{L^{\infty}(\R^d)} \| D^{\alpha} ( \nabla \Phi * \xi )\|_{L^2(\R^d)} + \| \nabla \Phi * \xi  \|_{L^{\infty}(\R^d)} \| D^{\alpha} w \|_{L^2(\R^d)} \right)^2\\
&\leq  \frac{\sigma}{4} \int_{\R^d} | \nabla D^{\alpha} w|^2 \, dx + CM^2 \| w\|_{H^s(\R^d)}^2.
\end{align*}
For the second term $I_2$
\begin{align*}
 I_2=& -m \int_{\R^d}   (\rho+ \eta)^{m-1} D^{\alpha} \nabla w  \cdot \nabla D^{\alpha} w \, dx
 -  m \int_{\R^d} \left( D^{\alpha} ((\rho+ \eta)^{m-1} \nabla w) -  (\rho+ \eta)^{m-1} D^{\alpha} \nabla w  \right) \cdot \nabla  D^{\alpha} w   \, dx\\
 \leq& - m \int_{\R^d}(\rho+ \eta)^{m-1} |\nabla D^{\alpha} w|^2 dx + \frac{\sigma}{4} \int_{\R^d} |\nabla D^{\alpha} w|^2  dx\\
  &+ C \left(\|\nabla  (\rho+ \eta)^{m-1} \|_{L^{\infty}(\R^d)} \| D^{\alpha} w\|_{L^2(\R^d)} + \| \nabla w \|_{L^{\infty}(\R^d)} \| D^{\alpha} (\rho+ \eta)^{m-1} \|_{L^2(\R^d)} \right)^2.
\end{align*}
From the estimates for $I_1$, and $I_2$, \eqref{estim_w_l_2} and \eqref{estim_w_der_l_2}, we deduce that
\begin{align*}
&\frac{1}{2} \frac{d}{dt} \| w\|_{H^s(\R^d)}^2 + \frac{\sigma}{2} \sum_{|\alpha|=2}^s \int_{\R^d} |\nabla D^{\alpha} w|^2 dx + m \sum_{|\alpha|=2}^s \int_{\R^d} (\rho+ \eta)^{m-1} |\nabla D^{\alpha} w|^2 dx \notag \\
\leq& C M^{2(m-1)}  \| w\|_{H^s(\R^d)}^2 + C \|\nabla (\rho+ \eta)^{m-1} \|_{L^{\infty}(\R^d)}^2 \|  w\|_{H^s(\R^d)}^2 + C \sum_{|\alpha|=2}^s \| D^{\alpha} (\rho+ \eta)^{m-1} \|_{L^2(\R^d)}^2 \| w\|_{H^s(\R^d)}^2.
\end{align*}
With the help of Gr\"{o}nwall's lemma, we obtain that
\begin{align}\label{add3}
\| w(t)\|_{H^s(\R^d)}^2 \leq \| u_0 \|_{H^s(\R^d)}^2 \exp \left( C \int_0^t  \big( M^{2(m-1)} + \|\nabla  (\rho+ \eta)^{m-1} \|_{L^{\infty}(\R^d)}^2 +  \sum_{|\alpha|=2}^s \| D^{\alpha} (\rho+ \eta)^{m-1} \|_{L^2(\R^d)}^2  \big)  d \tau \right).
\end{align}
Since $\rho \in Y_T$, we obtain that
\begin{align}\label{add2}
\int_0^t  \|\nabla  (\rho+ \eta)^{m-1} \|_{L^{\infty}(\R^d)}^2   d \tau
& \leq C \underset{\tau \in (0,t) }{ \text{ess sup} } \|  (\rho+ \eta)^{m-2} \|_{L^{\infty}(\R^d)}^2 \int_0^t  \|  \rho \|_{H^s(\R^d)}^2 d \tau \\
&\leq C \underset{\tau \in (0,t) }{ \text{ess sup} } \|  (\rho+ \eta)^{m-2}  \|_{L^{\infty}(\R^d)}^2  t^{\frac{1}{2}}  \|\rho \|_{L^4(0,T; H^s(\R^d))}^2 \notag\\
&\leq C M^{2(m-1)}t^{\frac{1}{2}}.\notag
\end{align}
In addition,
\begin{align*}
&\int_0^t \| D^{\alpha} (\rho+ \eta)^{m-1} \|_{L^2(\R^d)}^2   d \tau\\
 \leq& C \int_0^t \left( \| (\rho+ \eta)^{m-2} D^{\alpha} (\rho+ \eta) \|_{L^2(\R^d)}^2 + \| D^{\alpha}((\rho+ \eta)^{m-2} (\rho+ \eta)) - (\rho+ \eta)^{m-2}D^{\alpha} (\rho+ \eta)  \|_{L^2(\R^d)}^2 \right) d \tau \\
\leq& C \underset{\tau \in (0,t) }{ \text{ess sup} }  \|(\rho+ \eta)^{m-2} \|_{L^{\infty}(\R^d)}^2 \int_0^t \| D^{\alpha} \rho \|_{L^2(\R^d)}^2 d \tau \\
&+ C  \int_0^t  \left(\|\nabla (\rho+ \eta)^{m-2} \|_{L^{\infty}(\R^d)} \| D^{\alpha-1} \rho \|_{L^2(\R^d)}  + \|  (\rho+ \eta) \|_{L^{\infty}(\R^d)} \| D^{\alpha} (\rho+ \eta)^{m-2} \|_{L^2(\R^d)} \right)^2 d \tau\\
\leq& C M^{2(m-1)} t^{\frac{1}{2}} + C M^2  \int_0^t  \| D^{\alpha} (\rho+ \eta)^{m-2} \|_{L^2(\R^d)}^2  d\tau.
\end{align*}
By iteration technique, we obtain that
\begin{align}\label{add1}
\int_0^t \| D^{\alpha} (\rho+ \eta)^{m-1} \|_{L^2(\R^d)}^2   d \tau &\leq C M^{2(m-1)} t^{\frac{1}{2}} + C M^{2\lfloor m \rfloor} \int_0^t  \| D^{\alpha} (\rho+ \eta)^{m- \lfloor m \rfloor } \|_{L^2(\R^d)}^2  d\tau \\
& \leq C (M^{2(m-1)} + M^{2( \lfloor m \rfloor  + s)} ) t^{\frac{1}{2}}.\notag
\end{align}
Plugging \eqref{add2}-\eqref{add1} into \eqref{add3}, we have
\begin{align*}
\| w(t)\|_{ L^{\infty} ( 0, T_1; H^s(\R^d) )} \leq \| u_0 \|_{H^s(\R^d)}
\exp \left( C M^{2(m-1)} T_1+C M^{2(m-1)} T_1^\frac{1}{2} +C M^{2( \lfloor m \rfloor  + s)}  T_1^{\frac{1}{2}} \right),
\end{align*}
which is less or equal than $M$ if we choose $T_1>0$ small enough.
With this observation at hand, we can define the following operator:
\begin{align*}
\mathcal{T}:  Y_{T_1}  &\rightarrow  X_{T_1} \subset Y_{T_1}, \\
\rho &\mapsto w.
\end{align*}

We need to prove that $\mathcal{T}$ is a compact operator.
Let $(\rho_i)_{i\in \N}$ be a uniformly bounded sequence in $Y_{T_1}$ and $w_i := \mathcal{T}(\rho_i)$ for $i\in \N$.
From \eqref{add3}, we get
\begin{align}
&\|w_i \|_{L^{\infty}(0,T_1; H^s(\R^d))} + \|w_i \|_{L^2(0,T_1; H^{s+1}(\R^d))} \leq C,\label{add4}\\
&\|\partial_t w_i\|_{L^2(0,T_1; H^{s-1}(\R^d)) } \leq C.\label{add5}
\end{align}
Since we cannot apply Aubin-Lions lemma in the whole space $\R^d$, we  should prove  that $\int_{\R^d} w(\cdot, x)|x|dx < \infty$.
It is easy to see
 \begin{align*}
\frac{d}{dt} \int_{\R^d} w |x| dx &\leq  \sigma  \int_{\R^d} |\nabla w|dx + \int_{\R^d} |\nabla  \Phi * \xi w |dx + m \int_{\R^d} (\rho+ \eta)^{m-1} |\nabla w|dx\\
 &\leq \sigma \| w\|_{H^s(\R^d)} + \|\nabla  \Phi * \xi\|_{L^2(\R^d)} \|w\|_{L^2(\R^d)} + m \| \rho + \eta \|_{L^{2(m-1)}(\R^d)}^{m-1} \|\nabla w\|_{L^2(\R^d)}\\
  &\leq \sigma M +M^2 + C(2M)^m \leq CM^m.
 \end{align*}
With this observation at hand, and combining it with Gr\"onwall's inequality, we get
\begin{align*}
\underset{t \in (0,T_1) }{ \text{ess sup} }\int_{\R^d} w |x| dx \leq \left( C M^m + \int_{\R^d} u_0^\sigma |x| dx \right),
\end{align*}
which is bounded since $\int_{\R^d} u_0^\sigma |x| dx < \infty$.
This together with \eqref{add4}-\eqref{add5} make it possible to use Aubin-Lions lemma to infer
\begin{align*}
w_i \rightarrow \bar{w} \ \text{in} \ L^{\infty}(0,T_1; H^{s-1}(\R^d)) \ \text{,} \
w_i \rightarrow \bar{w} \ \text{in} \ L^2(0,T_1; H^s(\R^d)) \ \text{and} \ w_i \overset{*}{\rightharpoonup} \bar{w} \ \text{in} \ L^{\infty}(0,T_1; H^s(\R^d)).
\end{align*}
Thanks to interpolation inequality and weak lower semicontinuity of convex functions, we obtain that $\forall \gamma \in (0, 1)$
\begin{align*}
\|w_i -\bar{w} \|_{L^{\infty}(0,T_1; H^{s-\gamma}(\R^d))} &\leq \|w_i - \bar{w} \|_{L^{\infty}(0,T_1; H^{s-1}(\R^d))}^{1-\theta} \|w_i -\bar{w} \|_{L^{\infty}(0,T_1; H^{s}(\R^d))}^{\theta}\\
& \leq C \|w_i -\bar{w} \|_{L^{\infty}(0,T_1; H^{s-1}(\R^d))}^{1-\theta} \rightarrow 0 \ \text{as} \ i \rightarrow \infty,
\end{align*}
where $\theta$ satisfies $s- \gamma = (1-\theta)(s-1) + \theta s$.
This implies that
\begin{align*}
w_i \rightarrow \bar{w} \ \text{in} \ L^{\infty}(0,T_1; H^{s-\frac{1}{2}}(\R^d))\cap L^2(0,T_1; H^{s+\frac{1}{2}}(\R^d))  \ \ \text{as} \ i \rightarrow \infty,
\end{align*}
and combining interpolation inequality, it follows that
\begin{align*}
w_i \rightarrow \bar{w} \ \text{in}  \  L^4(0,T_1; H^s (\R^d)) \ \text{as} \ i \rightarrow \infty.
\end{align*}
This proves that $\mathcal{T} : Y_{T_1} \rightarrow Y_{T_1}$ is a compact operator.
Therefore, we conclude that there exists a local solution of \eqref{u_delta_xi_rho_equation} with $\rho$ replaced by $w$ which lies in $X_{T_1}$
with the help of Schauder's fixed point theorem.

It remains to prove that the fixed point of $\mathcal{T}$ is unique. We assume that $w_1$ and $w_2$ are two solutions of
\eqref{u_delta_xi_rho_equation} with $\rho$ replaced by $w$, i.e.
\begin{align*}
\begin{cases}\partial_t (w_1 -w_2) =  \nabla \cdot \big( \sigma  \nabla  (w_1 -w_2) -(w_1 -w_2) \nabla \Phi * \xi +  \nabla \left( (w_1+ \eta)^{m}  - (w_2+ \eta)^{m}  \right) \big), \\
(w_1-w_2)(0,x)=0,\qquad x\in \R^d,t>0.
\end{cases}
\end{align*}
We multiply the equation above by $w_1-w_2$ and integrate it over $\R^d$ to obtain that
\begin{align}\label{add8}
&\frac{1}{2} \frac{d}{dt} \int_{\R^d} |w_1-w_2|^2 dx+ \sigma \int_{\R^d} |\nabla (w_1-w_2)|^2 dx \\
=& \int_{\R^d}(w_1 -w_2) \nabla \Phi * \xi  \cdot \nabla (w_1-w_2)dx - \int_{\R^d}  \nabla \left( (w_1+ \eta)^{m}  - (w_2+ \eta)^{m}  \right)\cdot \nabla (w_1-w_2) dx
\notag\\
=&:J_1+J_2.\notag
\end{align}
For the term $J_2$ we get the following estimate,
\begin{align}\label{add6}
J_2=&- \int_{\R^d} \nabla \left( \int_0^1 m (z w_1+(1-z)w_2 + \eta)^{m-1} dz (w_1-w_2)\right) \cdot \nabla (w_1-w_2) dx \\
\leq&  - \int_{\R^d}  \int_0^1 m (z w_1+(1-z)w_2 + \eta )^{m-1} dz |\nabla (w_1-w_2)| ^2 dx \notag\\
&+C\left\| \int_0^1(z w_1+(1-z) w_2 + \eta )^{m-2} (z \nabla w_1+(1-z) \nabla w_2) dz \right\|_{L^{\infty}(\R^d)}  \cdot \int_{\R^d} |w_1-w_2| \ |\nabla(w_1-w_2)| dx  \\
\leq&  C M^{m-1} \int_{\R^d} |w_1-w_2| \ |\nabla(w_1-w_2)| dx\notag\\
\leq&  \frac{\sigma}{4} \int_{\R^d} |\nabla(w_1-w_2)|^2 dx + C M^{m-1} \int_{\R^d} |w_1-w_2|^2 dx.\notag
\end{align}
For the term $J_1$,
\begin{align}\label{add7}
J_1  \leq \frac{\sigma}{4} \int_{\R^d} |\nabla(w_1-w_2)|^2 dx + C M^2 \int_{\R^d} |w_1-w_2|^2 dx.
\end{align}
Plugging \eqref{add6}-\eqref{add7} into \eqref{add8}, we have
\begin{align}\label{add9}
\frac{1}{2} \frac{d}{dt} \int_{\R^d} |w_1-w_2|^2 dx   +\frac{\sigma}{2} \int_{\R^d}  |\nabla(w_1-w_2)|^2 dx\leq C(M^{m-1}+M^2) \int_{\R^d} |w_1-w_2|^2 dx.
\end{align}
Taking advantage of Gr\"onwall's inequality and of the initial data of $w_1-w_2$ being 0, we deduce that
\begin{align*}
\underset{t \in (0,T_1) }{ \text{ess sup} }\int_{\R^d} |w_1(t, x)-w_2(t, x)|^2 dx \leq 0.
\end{align*}
It implies that the system \eqref{u_delta_xi_rho_equation} with $\rho$ replaced by $w$ possesses a unique local  solution $w\in X_{T_1}$.
We have proved so far that for a fixed $\xi \in X_{T_1}$, there exists a unique local solution $w\in L^{\infty}(0,T_1; H^s(\R^d))  \cap  L^2(0,T_1; H^{s+1}(\R^d))$
to the system \eqref{u_delta_xi_equation}.

Next, we proceed with  the second-step fixed point argument.

Using the same technique as for the $H^s$ estimates of the system \eqref{u_delta_xi_rho_equation} which we obtained in  \eqref{add3}, we deduce that the following operator is well-defined
\begin{align*}
\mathcal{T}':  X_{T_1} &\rightarrow  X_{T_1},\\
\xi &\mapsto v.
\end{align*}
It remains to prove that $\mathcal{T}'$ is a contraction. Let  $v_1, v_2 \in X_{T_1}$ be solutions of \eqref{u_delta_xi_equation} for $\xi_1, \xi_2 \in X_{T_1}$ respectively. Using the same argument as the one used in the proof of \eqref{add9}, we can get
\begin{align*}
&\frac{1}{2}  \int_{\R^d} |v_1(t, \cdot)-v_2(t, \cdot)|^2 dx   +\frac{\sigma}{4} \int_0^t \int_{\R^d}  |\nabla(v_1-v_2)|^2 dx ds \\
\leq& C(M^{m-1}+M^2)  \int_0^t \int_{\R^d} |v_1-v_2|^2 dx ds  + C M^2 t  \ \underset{\tau \in (0,t) }{ \text{ess sup} }  \int_{\R^d} |\xi_1-\xi_2|^2 dx.
\end{align*}
With the help of Gr\"onwall's inequality, we have
\begin{align*}
\sup_{t\in (0,T_2)}  \int_{\R^d} |v_1(t, x)-v_2(t, x)|^2 dx \leq CM^2 \exp( CT_2(M^{m-1}+M^2) ) T_2 \| \xi_1 - \xi_2 \|_{L^{\infty}(0,T_2; L^2(\R^d))}^2
\end{align*}
If we choose $T_2>0$ small enough, we obtain that $\mathcal{T}'$ is a contraction. Therefore, \eqref{u_sigma_approximation}  possesses a unique local solution
$u_\eta^\sigma\in X_T$.
By Moser's iteration technique, it is easy to obtain \eqref{add10}.
As a consequence, we complete the proof of Proposition \ref{prop_local_approximate}.
\end{proof}

\subsection{\texorpdfstring{$L^{\infty}$ Estimate of $\nabla u^\sigma_\eta$}{L infty estimates gradient u sigma approx} }
\label{Section_estimate_grad_u_sigma_eta}
In order to derive global well-posedness of $u^\sigma_\eta$, we have to infer the $L^\infty$ estimate of $\nabla u^\sigma_\eta$.
Compared with \citep{sugiyama2007time}, we use $\sigma \Delta u $ instead of  $\Delta (u +\eta)^m$ to deal with some terms on the left-side hand of inequality in order
to derive uniform estimate independent of $\eta$.
Our goal of this subsection is to prove the following proposition:
\begin{proposition}\label{grau}
Let the assumptions in Theorem \ref{theorem_u_sigma} hold and $u^{\sigma}_{\eta}$ be the weak solution to \eqref{u_sigma_approximation}. Then it holds
\begin{align*}
\sup_{t\in(0,T')} \| \nabla u^{\sigma}_{\eta}(t, \cdot)  \|_{L^{\infty}(\R^d)} \leq C,
\end{align*}
where $C$ is a positive constant which depends on $d$, $\sigma$, $m$, $\|u_0 \|_{L^1(\R^d) \cap L^{\infty}(\R^d)}$ and $\| \nabla u_0^\sigma \|_{L^1(\R^d) \cap L^{\infty}(\R^d)}$ but is independent of $\eta$.
\end{proposition}

The rest of this subsection is devoted to the proof of Proposition \ref{grau}.
\subsubsection{Some basic computations.}
Let $u^{\sigma}_{\eta} = f(\bar{\mathfrak{u}})$.
By simple calculations, we have
\begin{align}
\label{equation_u_frac}
\partial_t \bar{\mathfrak{u}} =& \left( \sigma + m(u^{\sigma}_{\eta}+\eta)^{m-1} \right) \Delta \bar{\mathfrak{u}} + \left( \sigma + m(u^{\sigma}_{\eta}+\eta)^{m-1} \right)   \frac{f''(\bar{\mathfrak{u}})}{f'(\bar{\mathfrak{u}})} |\nabla \bar{\mathfrak{u}}|^2\\
& +\frac{m(m-1)(u^{\sigma}_{\eta}+\eta)^{m-2}  |\nabla u^{\sigma}_{\eta}|^2}{f'(\bar{\mathfrak{u}})} - \frac{\nabla \cdot ( u^{\sigma}_{\eta} \nabla \Phi * u^{\sigma}_{\eta} )}{f'(\bar{\mathfrak{u}})} . \notag
\end{align}

We take the derivative of \eqref{equation_u_frac} with respect to $x_k$ for $k\in \{1, \ldots , d \}$ to obtain that
\begin{align}
\label{equation_deriv_u_frac}
\partial_{x_k} \partial_t \bar{\mathfrak{u}} =&  m(m-1)(u^{\sigma}_{\eta}+\eta)^{m-2} f'(\bar{\mathfrak{u}}) \partial_{x_k} \bar{\mathfrak{u}}
 \Delta \bar{\mathfrak{u}} + \left( \sigma + m(u^{\sigma}_{\eta}+\eta)^{m-1} \right)   \partial_{x_k} \Delta \bar{\mathfrak{u}}\\
 &+\left( \sigma + m(u^{\sigma}_{\eta}+\eta)^{m-1} \right)    |\nabla \bar{\mathfrak{u}}|^2 \left( \frac{f'''(\bar{\mathfrak{u}})f'(\bar{\mathfrak{u}}) - (f''(\bar{\mathfrak{u}}))^2 }{(f'(\bar{\mathfrak{u}}))^2 }  \right) \partial_{x_k} \bar{\mathfrak{u}} \notag \\
 &+ 2\left( \sigma + m(u^{\sigma}_{\eta}+\eta)^{m-1} \right)   \frac{f''(\bar{\mathfrak{u}})}{f'(\bar{\mathfrak{u}})}  \nabla \bar{\mathfrak{u}} \cdot \nabla \partial_{x_k} \bar{\mathfrak{u}} \notag \\
 &+ m(m-1)(m-2)(u^{\sigma}_{\eta}+\eta)^{m-3} (f'(\bar{\mathfrak{u}}))^2   |\nabla \bar{\mathfrak{u}}|^2  \partial_{x_k} \bar{\mathfrak{u}} \notag \\
&+ 2m(m-1)(u^{\sigma}_{\eta}+\eta)^{m-2}  \nabla \bar{\mathfrak{u}} \cdot \nabla \partial_{x_k} u^{\sigma}_{\eta}  \notag \\
&- \frac{1}{f'(\bar{\mathfrak{u}})} \left( \nabla \partial_{x_k}   u^{\sigma}_{\eta} \cdot \nabla \Phi * u^{\sigma}_{\eta} + \nabla u^{\sigma}_{\eta} \cdot \nabla \partial_{x_k}  \Phi * u^{\sigma}_{\eta}  +  \partial_{x_k} u^{\sigma}_{\eta} \Delta  \Phi * u^{\sigma}_{\eta} + u^{\sigma}_{\eta} \partial_{x_k}   \Delta  \Phi * u^{\sigma}_{\eta}  \right) \notag \\
&+  \frac{f''(\bar{\mathfrak{u}})}{(f'(\bar{\mathfrak{u}}))^2} \partial_{x_k} \bar{\mathfrak{u}} \left( \nabla u^{\sigma}_{\eta} \cdot \nabla \Phi * u^{\sigma}_{\eta} + u^{\sigma}_{\eta} \Delta   \Phi * u^{\sigma}_{\eta} \right).    \notag
\end{align}

Multiplying \eqref{equation_deriv_u_frac} by $\partial_{x_k}  \bar{\mathfrak{u}}$ and denoting $|\partial_{x_k}  \bar{\mathfrak{u}}|^2$ by $\bar{\mathfrak{U}}_k$, we deduce that
\begin{align}
\label{equation_deriv_capital_u_frac}
\frac{1}{2} \partial_t  \bar{\mathfrak{U}}_k =&  m(m-1)(u^{\sigma}_{\eta}+\eta)^{m-2}  f'(\bar{\mathfrak{u}}) \Delta \bar{\mathfrak{u}} \ \bar{\mathfrak{U}}_k  + \left( \sigma + m(u^{\sigma}_{\eta}+\eta)^{m-1} \right)   \partial_{x_k} \Delta \bar{\mathfrak{u}} \partial_{x_k}  \bar{\mathfrak{u}} \\
 &+\left( \sigma + m(u^{\sigma}_{\eta}+\eta)^{m-1} \right)     \left( \frac{f'''(\bar{\mathfrak{u}})f'(\bar{\mathfrak{u}}) - (f''(\bar{\mathfrak{u}}))^2 }{(f'(\bar{\mathfrak{u}}))^2 }  \right)|\nabla \bar{\mathfrak{u}}|^2 \bar{\mathfrak{U}}_k \notag \\
 &+ 2\left( \sigma + m(u^{\sigma}_{\eta}+\eta)^{m-1} \right)   \frac{f''(\bar{\mathfrak{u}})}{f'(\bar{\mathfrak{u}})}  \nabla \bar{\mathfrak{u}} \cdot \nabla \partial_{x_k} \bar{\mathfrak{u}} \partial_{x_k}  \bar{\mathfrak{u}}  \notag \\
 &+ m(m-1)(m-2)(u^{\sigma}_{\eta}+\eta)^{m-3} (f'(\bar{\mathfrak{u}}))^2   |\nabla \bar{\mathfrak{u}}|^2 \bar{\mathfrak{U}}_k  \notag \\
&+ 2m(m-1)(u^{\sigma}_{\eta}+\eta)^{m-2}  \nabla \bar{\mathfrak{u}}\  \cdot \nabla \partial_{x_k} u^{\sigma}_{\eta} \ \partial_{x_k} \bar{\mathfrak{u}} \notag \\
&- \frac{\partial_{x_k}  \bar{\mathfrak{u}}}{f'(\bar{\mathfrak{u}})} \left( \nabla \partial_{x_k}   u^{\sigma}_{\eta} \cdot \nabla \Phi * u^{\sigma}_{\eta} + \nabla u^{\sigma}_{\eta} \cdot \nabla \partial_{x_k}  \Phi * u^{\sigma}_{\eta}  +  \partial_{x_k} u^{\sigma}_{\eta} \Delta  \Phi * u^{\sigma}_{\eta} + u^{\sigma}_{\eta} \partial_{x_k}   \Delta  \Phi * u^{\sigma}_{\eta}  \right) \notag \\
&+  \frac{f''(\bar{\mathfrak{u}})}{(f'(\bar{\mathfrak{u}}))^2}  \left( \nabla u^{\sigma}_{\eta} \cdot \nabla \Phi * u^{\sigma}_{\eta} + u^{\sigma}_{\eta} \Delta   \Phi * u^{\sigma}_{\eta} \right)   \bar{\mathfrak{U}}_k.   \notag
\end{align}
One can see that
\begin{align}
\label{equation_big_u_frac_easier}
\frac{1}{2}  \partial_t  \bar{\mathfrak{U}}_k =&  m(m-1)(u^{\sigma}_{\eta}+\eta)^{m-2}  f'(\bar{\mathfrak{u}})   \Delta \bar{\mathfrak{u}} \ \bar{\mathfrak{U}}_k  + \left( \sigma + m(u^{\sigma}_{\eta}+\eta)^{m-1} \right)  \left( \frac{1}{2}\Delta \bar{\mathfrak{U}}_k - |\partial_{x_k} \nabla \bar{\mathfrak{u}}|^2 \right) \\
 &+\left( \sigma + m(u^{\sigma}_{\eta}+\eta)^{m-1} \right)     \left( \frac{f'''(\bar{\mathfrak{u}})f'(\bar{\mathfrak{u}}) - (f''(\bar{\mathfrak{u}}))^2 }{(f'(\bar{\mathfrak{u}}))^2 }  \right) |\nabla \bar{\mathfrak{u}}|^2 \bar{\mathfrak{U}}_k \notag \\
 &+ 2m(m-1)(u^{\sigma}_{\eta}+\eta)^{m-2}   f''(\bar{\mathfrak{u}}) |\nabla \bar{\mathfrak{u}}|^2 \bar{\mathfrak{U}}_k \notag \\
  &+ m(m-1)(m-2)(u^{\sigma}_{\eta}+\eta)^{m-3} (f'(\bar{\mathfrak{u}}))^2   |\nabla \bar{\mathfrak{u}}|^2 \bar{\mathfrak{U}}_k \notag \\
 &+ \left( \sigma + m(u^{\sigma}_{\eta}+\eta)^{m-1} \right)   \frac{f''(\bar{\mathfrak{u}})}{f'(\bar{\mathfrak{u}})}  \nabla \bar{\mathfrak{u}} \cdot \nabla \bar{\mathfrak{U}}_k \notag \\
& + m(m-1)(u^{\sigma}_{\eta}+\eta)^{m-2} f'(\bar{\mathfrak{u}})  \nabla \bar{\mathfrak{u}} \cdot  \nabla \bar{\mathfrak{U}}_k  \notag \\
&- \frac{1}{2} \nabla \bar{\mathfrak{U}}_k \cdot \nabla \Phi * u^{\sigma}_{\eta}  - \nabla \bar{\mathfrak{u}} \cdot \nabla \partial_{x_k}  \Phi * u^{\sigma}_{\eta} \partial_{x_k}  \bar{\mathfrak{u}}  - \frac{u^{\sigma}_{\eta} \partial_{x_k} \Delta \Phi * u^{\sigma}_{\eta} \partial_{x_k}  \bar{\mathfrak{u}}   }{ f'(\bar{\mathfrak{u}})} \notag \\
& + \left( \frac{f''(\bar{\mathfrak{u}})}{(f'(\bar{\mathfrak{u}}))^2}  u^{\sigma}_{\eta} -1  \right) \Delta \Phi * u^{\sigma}_{\eta} \bar{\mathfrak{U}}_k.  \notag
\end{align}

We denote $\sum_{k=1}^d \bar{\mathfrak{U}}_k = | \nabla \bar{\mathfrak{u}}|^2$ by $\mathfrak{D}$. Since $\sum_{k=1}^d \bar{\mathfrak{U}}_k^2 \leq \mathfrak{D}^2$,  we obtain that
\begin{align}
\label{equation_big_d_frac}
\frac{1}{2} \partial_t \mathfrak{D} \leq&  \left( m(m-1)(u^{\sigma}_{\eta}+\eta)^{m-2}  f'(\bar{\mathfrak{u}})  \right) \Delta \bar{\mathfrak{u}} \ \mathfrak{D}  + \left( \sigma + m(u^{\sigma}_{\eta}+\eta)^{m-1} \right)  \left( \frac{1}{2}\Delta \mathfrak{D} - \sum_{k=1}^d|\partial_{x_k} \nabla \bar{\mathfrak{u}}|^2 \right) \\
 &+\left( \sigma + m(u^{\sigma}_{\eta}+\eta)^{m-1} \right)     \left( \frac{f'''(\bar{\mathfrak{u}})f'(\bar{\mathfrak{u}}) - (f''(\bar{\mathfrak{u}}))^2 }{(f'(\bar{\mathfrak{u}}))^2 }  \right) \mathfrak{D}^2 \notag \\
 &+ m(m-1)(m-2)(u^{\sigma}_{\eta}+\eta)^{m-3} (f'(\bar{\mathfrak{u}}))^2   \mathfrak{D}^2  + 2m(m-1)(u^{\sigma}_{\eta}+\eta)^{m-2}   f''(\bar{\mathfrak{u}}) \mathfrak{D}^2  \notag \\
&+ \left( \sigma + m(u^{\sigma}_{\eta}+\eta)^{m-1} \right)   \frac{f''(\bar{\mathfrak{u}})}{f'(\bar{\mathfrak{u}})}  \nabla \bar{\mathfrak{u}} \cdot \nabla \mathfrak{D} + m(m-1)(u^{\sigma}_{\eta}+\eta)^{m-2} f'(\bar{\mathfrak{u}})  \nabla \bar{\mathfrak{u}} \cdot  \nabla \mathfrak{D}  \notag \\
&- \frac{1}{2} \nabla \mathfrak{D} \cdot \nabla \Phi * u^{\sigma}_{\eta}  - \sum_{k=1}^d \nabla \bar{\mathfrak{u}} \cdot \nabla \partial_{x_k}  \Phi * u^{\sigma}_{\eta} \partial_{x_k}  \bar{\mathfrak{u}}  - \sum_{k=1}^d  \frac{u^{\sigma}_{\eta} \partial_{x_k} \Delta \Phi * u^{\sigma}_{\eta} \partial_{x_k}  \bar{\mathfrak{u}}   }{ f'(\bar{\mathfrak{u}})} \notag \\
& + \left( \frac{f''(\bar{\mathfrak{u}})}{(f'(\bar{\mathfrak{u}}))^2}  u^{\sigma}_{\eta} -1  \right) \Delta \Phi * u^{\sigma}_{\eta} \mathfrak{D}.  \notag
\end{align}

\subsubsection{Decomposition of $\R^d$.}
Let $t\in [0,T']$ be fixed. Define $\tilde{\omega}>0$ such that
\begin{align}
\label{def_var_sigma}
\tilde{\omega} := \begin{cases}  \min \left\{ \frac{\sqrt{\sigma}}{16  e \sqrt{m(m-1)(m-2)} (\| u^{\sigma}_{\eta} \|_{L^{\infty}(0,T' ; L^{\infty}(\R^d) ) }+M)^{\frac{m-3}{2}}}, \frac{\sigma}{16em(m-1)(\| u^{\sigma}_{\eta}  \|_{L^{\infty}(0,T' ; L^{\infty}(\R^d) )}+M)^{m-2} } \right\} , \ &\text{for} \ m\geq 3,  \\
\frac{\sigma}{16 e m (m-1)}, \ &\text{for} \ m= 2,
\end{cases}
\end{align}
and
\begin{align*}
\tilde{\Omega}_k(t):= \{x\in \R^d \ | \ (k-1) \tilde{\omega} \leq u^{\sigma}_{\eta}(t,x)<k \tilde{\omega} \} \ \text{for} \ k\in \N.
\end{align*}
Since $\| u^{\sigma}_{\eta}\|_{L^{\infty}(0,T' ; L^{\infty}(\R^d) ) } < \infty$,  there exists $k_0\in \N$ such that
\begin{align*}
\tilde{\Omega}_k(t)\cap\tilde{\Omega}_j(t)= \emptyset \ \forall j,k \in \{1, \ldots , k_0\} \ \text{and} \ \R^d = \bigcup_{k=1}^{k_0}\tilde{\Omega}_k(t).
\end{align*}
For any fixed $t\in[0,T']$,  we define operator $f_k(\bar{\mathfrak{u}})$ as
\begin{align*}
u^{\sigma}_{\eta}(t,x) = f_k(\bar{\mathfrak{u}}(t,x)) = (k-3)\tilde{\omega} +4e\tilde{\omega}\int_0^{\bar{\mathfrak{u}}(t,x)} e^{-s^2}ds \ \text{in}   \bigcup_{i=k-1}^{k+1}\tilde{\Omega}_i(t).
\end{align*}
Moreover, let $\phi_k(x)\in C^2(\R^d)$ satisfy $0\leq \phi_k(x) \leq 1$ and
\begin{align*}
\phi_k(x):= \begin{cases} 1, \ \text{for} \ x \in \tilde{\Omega}_k(t), \\
0, \ \text{for} \ x \in \R^d \setminus \bigcup_{j=k-1}^{k+1} \tilde{\Omega}_j(t).
\end{cases}
\end{align*}
Then there exists a constant $C_{\phi_k}$, which depends on $\text{dist}(\tilde{\Omega}_k, \partial \tilde{\Omega}_{k-1})$ and
$\text{dist}(\tilde{\Omega}_k, \partial \tilde{\Omega}_{k+1})$, such that
\begin{align*}
|\nabla \phi_k (x)| \leq C_{\phi_k} (\phi_k)^{\frac{3}{4}} \quad {\rm{and}}\quad |\Delta\phi_k(x) | \leq C_{\phi_k}.
\end{align*}
For details about the existence of this cut-off function we refer to \citep{sugiyama2007time} and \citep{ishida2013gradient} .

\subsubsection{Coefficient of $\mathfrak{D}^2$ in \eqref{equation_big_d_frac}.}
Now we are going to prove that the coefficient of $\mathfrak{D}^2$ in \eqref{equation_big_d_frac} is negative.
Following the same way in \citep{sugiyama2007time}, we obtain that
\begin{align*}
\frac{1}{4e} \leq \bar{\mathfrak{u}}(t,x) = f_k^{-1}(u^{\sigma}_{\eta}) \leq 1 \ \text{in} \bigcup_{i=k-1}^{k+1}\tilde{\Omega}_i(t).
\end{align*}
From the definition of $f_k(\bar{\mathfrak{u}})$, we derive that
\begin{align*}
\frac{f_k'''(\bar{\mathfrak{u}})f_k'(\bar{\mathfrak{u}}) - (f_k''(\bar{\mathfrak{u}}))^2 }{(f_k'(\bar{\mathfrak{u}}))^2 }   = \frac{8e\tilde{\omega} e^{- \bar{\mathfrak{u}}^2 } (2\bar{\mathfrak{u}}^2 -1 ) 4e\tilde{\omega} e^{- \bar{\mathfrak{u}}^2    } - (-8e\tilde{\omega}\bar{\mathfrak{u}} e^{- \bar{\mathfrak{u}}^2 })^2 }{(4e\tilde{\omega} e^{- \bar{\mathfrak{u}}^2 })^2} =-2.
\end{align*}
By \eqref{def_var_sigma}, we have
\begin{align*}
 m(m-1)(m-2)(u^{\sigma}_{\eta}+\eta)^{m-3}(f_k'(\bar{\mathfrak{u}}))^2 &\leq 16 e^2\tilde{\omega}^2 e^{-2\bar{\mathfrak{u}}^2}m(m-1)(m-2)(u^{\sigma}_{\eta}+\eta)^{m-3}\\
 &\leq 16 e^2\tilde{\omega}^2m(m-1)(m-2)(2\| u^{\sigma}_{\eta} \|_{L^{\infty}(0,T'; L^{\infty} (\R^d ) )} )^{m-3}
 \leq \frac{1}{8}\sigma.
\end{align*}
With this observation at hand, we deduce that
\begin{align}\label{coe}
 &\left( \sigma + m(u^{\sigma}_{\eta}+\eta)^{m-1} \right)     \left( \frac{f_k'''(\bar{\mathfrak{u}})f_k'(\bar{\mathfrak{u}}) - (f_k''(\bar{\mathfrak{u}}))^2 }{(f_k'(\bar{\mathfrak{u}}))^2 }  \right) \mathfrak{D}^2  \\
 &+ m(m-1)(m-2)(u^{\sigma}_{\eta}+\eta)^{m-3} (f_k'(\bar{\mathfrak{u}}))^2   \mathfrak{D}^2  + 2m(m-1)(u^{\sigma}_{\eta}+\eta)^{m-2}   f_k''(\bar{\mathfrak{u}}) \mathfrak{D}^2\notag \\
 \leq& -2\left( \sigma + m(u^{\sigma}_{\eta}+\eta)^{m-1} \right)      \mathfrak{D}^2 \notag + m(m-1)(m-2)(u^{\sigma}_{\eta}+\eta)^{m-3} (f_k'(\bar{\mathfrak{u}}))^2   \mathfrak{D}^2 \notag  \\
 \leq& -2 \sigma \mathfrak{D}^2 -2m (u^{\sigma}_{\eta}+\eta)^{m-1} \mathfrak{D}^2 + \frac{1}{8} \sigma \mathfrak{D}^2
 <0 \ \text{in} \ \bigcup_{i=k-1}^{k+1}\tilde{\Omega}_i(t) .\notag
\end{align}

\subsubsection{Bochner type inequality for $V_{\phi}:= |\nabla \bar{\mathfrak{u}}(t,x)|^2 \phi $.}
In order to simplify the notation in this step, we use $\phi$ and $f$ instead of $\phi_k$ and $f_k$ respectively.
Multiplying \eqref{equation_big_d_frac} by $\phi$ and using \eqref{coe} together with
\begin{align*}
&f'(\bar{\mathfrak{u}}) \leq 4 e \tilde{\omega}, \ \ \ \ \ f'' (\bar{\mathfrak{u}}) < 0, \ \ \ \ \ f''' (\bar{\mathfrak{u}})= 8e \tilde{\omega} e^{-\bar{\mathfrak{u}}^2}(2\bar{\mathfrak{u}}^2-1), \\
&\frac{f'''(\bar{\mathfrak{u}})f'(\bar{\mathfrak{u}}) - (f''(\bar{\mathfrak{u}}))^2 }{(f'(\bar{\mathfrak{u}}))^2 }  =-2, \ \ \ \ \ -2\leq \frac{f''(\bar{\mathfrak{u}})}{f'(\bar{\mathfrak{u}})}=-2\bar{\mathfrak{u}}\leq - \frac{1}{2e},
\end{align*}
we obtain that
\begin{align}
\label{equation_big_V_phi}
\frac{1}{2} \partial_t  V_{\phi} \leq&  m(m-1)(u^{\sigma}_{\eta}+\eta)^{m-2}  4 e \tilde{\omega}  \sum_{i=1}^d |\partial_{ii} \bar{\mathfrak{u}}| \ V_{\phi}
+ \left( \sigma + m(u^{\sigma}_{\eta}+\eta)^{m-1} \right)  \left( \frac{1}{2}\Delta \mathfrak{D} \phi  - \sum_{k=1}^d|\partial_{x_k} \nabla \bar{\mathfrak{u}}|^2 \phi \right) \\
 &-2\left( \sigma + m(u^{\sigma}_{\eta}+\eta)^{m-1} \right)      \mathfrak{D}^2 \phi  + \sigma  \frac{f''(\bar{\mathfrak{u}})}{f'(\bar{\mathfrak{u}})}  \nabla \bar{\mathfrak{u}} \cdot \nabla \mathfrak{D}  \ \phi  + \frac{1}{8} \sigma \mathfrak{D}^2 \phi  \notag \\
 &+ m \left(  (u^{\sigma}_{\eta}+\eta)  \frac{f''(\bar{\mathfrak{u}})}{f'(\bar{\mathfrak{u}})} + (m-1)f'(\bar{\mathfrak{u}}) \right) (u^{\sigma}_{\eta}+\eta)^{m-2}  \nabla \bar{\mathfrak{u}} \cdot \nabla \mathfrak{D} \ \phi     \notag \\
&- \frac{1}{2} \nabla \mathfrak{D} \cdot \nabla \Phi * u^{\sigma}_{\eta} \phi    - \sum_{k=1}^d \nabla \bar{\mathfrak{u}} \cdot \nabla \partial_{x_k}  \Phi * u^{\sigma}_{\eta} \partial_{x_k}  \bar{\mathfrak{u}} \phi    - \sum_{k=1}^d  \frac{u^{\sigma}_{\eta} \partial_{x_k} \Delta \Phi * u^{\sigma}_{\eta} \partial_{x_k}  \bar{\mathfrak{u}}   }{ f'(\bar{\mathfrak{u}})} \phi  \notag \\
& + \left( \frac{f''(\bar{\mathfrak{u}})}{(f'(\bar{\mathfrak{u}}))^2}  u^{\sigma}_{\eta} -1  \right) \Delta \Phi * u^{\sigma}_{\eta} \mathfrak{D} \phi.  \notag
\end{align}
\raggedbottom
Multiplying \eqref{equation_big_V_phi} by $V_{\phi}^{r-1}$ and integrating it over $\tilde{\Omega}:= \bigcup_{j=k-1}^{k+1}\tilde{\Omega}_j(t)$,
we deduce that
\begin{align}
\label{inequality_big_V_phi_L_r_norm_simpl}
\frac{1}{r} \frac{d}{dt} \| V_{\phi} \|_{L^r(\tilde{\Omega})}^r \leq& -m(r-1) \int_{\tilde{\Omega}} (u^{\sigma}_{\eta}+\eta)^{m-1} | \nabla  V_{\phi} |^2 (V_{\phi})^{r-2}  dx
-\sigma (r-1) \int_{\tilde{\Omega}}  | \nabla  V_{\phi} |^2 (V_{\phi})^{r-2}  dx   \\
 &-  2m \int_{\tilde{\Omega}} (u^{\sigma}_{\eta}+\eta)^{m-1}   \sum_{k=1}^d|\partial_{x_k} \nabla \bar{\mathfrak{u}}|^2 \phi V_{\phi}^{r-1} dx  -  2 \sigma \int_{\tilde{\Omega}}   \sum_{k=1}^d|\partial_{x_k} \nabla \bar{\mathfrak{u}}|^2 \phi V_{\phi}^{r-1} dx  \notag \\
   &-4 m \int_{\tilde{\Omega}}  (u^{\sigma}_{\eta}+\eta)^{m-1}      \mathfrak{D} V_{\phi}^r dx -4  \sigma \int_{\tilde{\Omega}}   \mathfrak{D} V_{\phi}^r dx  +  \frac{1}{4} \sigma \int_{\tilde{\Omega}}   \mathfrak{D} V_{\phi}^r dx  \notag \\
   &+m(m-1) \int_{\tilde{\Omega}} (u^{\sigma}_{\eta}+\eta)^{m-2} \nabla u^{\sigma}_{\eta} \cdot \nabla V_{\phi} (V_{\phi})^{r-1}  dx  \notag \\
   &+ 2m(r-1) \int_{\tilde{\Omega}}  (u^{\sigma}_{\eta}+\eta)^{m-1}  \mathfrak{D}  (V_{\phi})^{r-2}  \nabla \phi   \cdot \nabla V_{\phi}   dx \notag \\
   &+ 2\sigma(r-1) \int_{\tilde{\Omega}}   \mathfrak{D}  (V_{\phi})^{r-2}  \nabla \phi   \cdot \nabla V_{\phi}   dx
+ m \int_{\tilde{\Omega}}  (u^{\sigma}_{\eta}+\eta)^{m-1}  \mathfrak{D}    \Delta \phi   (V_{\phi})^{r-1}  dx \notag \\
&+ \sigma  \int_{\tilde{\Omega}}  \mathfrak{D}    \Delta \phi   (V_{\phi})^{r-1}  dx
+ \left( 8 e \tilde{\omega} m(m-1) \right)  \int_{\tilde{\Omega}}(u^{\sigma}_{\eta}+\eta)^{m-2}    \sum_{i=1}^d |\partial_{ii} \bar{\mathfrak{u}}|  \ V_{\phi}^r dx  \notag \\
&+ 2m \int_{\tilde{\Omega}} (u^{\sigma}_{\eta}+\eta)^{m-1}   \left| \frac{f''(\bar{\mathfrak{u}})}{f'(\bar{\mathfrak{u}})} \right|    |\nabla \bar{\mathfrak{u}} | \  |\nabla V_{\phi}| \  |V_{\phi}|^{r-1}  dx \notag \\
&+ 2\sigma \int_{\tilde{\Omega}}   \left| \frac{f''(\bar{\mathfrak{u}})}{f'(\bar{\mathfrak{u}})} \right|    |\nabla \bar{\mathfrak{u}} | \  |\nabla V_{\phi}| |V_{\phi}|^{r-1}  dx \notag \\
& +2m \int_{\tilde{\Omega}} (u^{\sigma}_{\eta}+\eta)^{m-1}   \left| \frac{f''(\bar{\mathfrak{u}})}{f'(\bar{\mathfrak{u}})} \right|    |\nabla \bar{\mathfrak{u}} | \  |\nabla \phi| \mathfrak{D} (V_{\phi})^{r-1}  dx \notag \\
& +2\sigma \int_{\tilde{\Omega}}  \left| \frac{f''(\bar{\mathfrak{u}})}{f'(\bar{\mathfrak{u}})} \right|    |\nabla \bar{\mathfrak{u}} | \  |\nabla \phi| \mathfrak{D} (V_{\phi})^{r-1}  dx \notag \\
&- \int_{\tilde{\Omega}} \nabla V_{\phi} \cdot \nabla \Phi * u^{\sigma}_{\eta} (V_{\phi})^{r-1} dx +  \int_{\tilde{\Omega}} \mathfrak{D} \nabla \phi \cdot  \nabla  \Phi * u^{\sigma}_{\eta}   (V_{\phi})^{r-1} dx \notag   \\
&    - 2   \sum_{k=1}^d  \int_{\tilde{\Omega}} \nabla \bar{\mathfrak{u}} \cdot \nabla \partial_{x_k}  \Phi * u^{\sigma}_{\eta} \partial_{x_k}  \bar{\mathfrak{u}} \phi  (V_{\phi})^{r-1}   dx \notag \\
& + 2 \int_{\tilde{\Omega}} \left( \frac{f''(\bar{\mathfrak{u}})}{(f'(\bar{\mathfrak{u}}))^2}  u^{\sigma}_{\eta} -1  \right) \Delta \Phi * u^{\sigma}_{\eta} (V_{\phi})^r  dx  \notag  \\
&  -2 \sum_{k=1}^d \int_{\tilde{\Omega}}    \frac{u^{\sigma}_{\eta} \partial_{x_k} \Delta \Phi * u^{\sigma}_{\eta} \partial_{x_k}  \bar{\mathfrak{u}}   }{ f'(\bar{\mathfrak{u}})} \phi  (V_{\phi})^{r-1} dx  \notag \\
=:& -I_1 -I_1'-I_2 -I_2'-I_3 -I_3'+ \frac{1}{16} I_3'  +II_1+II_2+II_3+II_4+II_5 +III \notag\\\
&+ IV_1+IV_1'+ IV_2+IV_2'+J_1^1+J_1^2+J_2+J_3+J_4.\notag
\end{align}

Now we are going to estimate 15 terms.
By Young's inequality and definition of $\tilde w$ we obtain that
\begin{align}
\label{est_II_1}
II_1
&\leq \frac{\sigma (r-1)}{8}  \int_{\tilde{\Omega}} |  V_{\phi} |^{r-2} | \nabla V_{\phi} |^2 dx + \frac{2m^2(m-1)^2}{\sigma(r-1)} \left(  \| u^{\sigma}_{\eta}\|_{L^{\infty}(0,T'; L^{\infty} (\R^d ) )}+M \right)^{2(m-2)} \int_{\tilde{\Omega}} (f'(\bar{\mathfrak{u}}))^2 | \nabla  \bar{\mathfrak{u}} |^2 |V_{\phi} |^r dx \\
&\leq \frac{\sigma (r-1)}{8}  \int_{\tilde{\Omega}} |  V_{\phi} |^{r-2} | \nabla V_{\phi} |^2 dx + \frac{2m^2(m-1)^2}{\sigma(r-1)} \left(  \| u^{\sigma}_{\eta}\|_{L^{\infty}(0,T'; L^{\infty} (\R^d ) )}+M \right)^{2(m-2)} \int_{\tilde{\Omega}} 16e^2\tilde{\omega}^2 e^{-2 \bar{\mathfrak{u}}^2 } \mathfrak{D} |V_{\phi} |^r dx \notag\\
&\leq  \frac{1}{8} I_1' + \frac{1}{8} I_3'.\notag
\end{align}
Since $|\nabla \phi|^4 \leq C_{\phi} ^4 \phi^3$, we infer that
\begin{align}
\label{est_II_2}
II_2
&\leq \frac{1}{8}I_1 + 8m(r-1)  \int_{\tilde{\Omega}} (u^{\sigma}_{\eta}+\eta)^{m-1} \mathfrak{D}^2 (V_{\phi})^{r-2} |\nabla \phi|^2dx\\
&\leq \frac{1}{8}I_1 +  \frac{1}{2} m \int_{\tilde{\Omega}} (u^{\sigma}_{\eta}+\eta)^{m-1} \mathfrak{D} (V_{\phi})^r  dx + 32m (r-1)^2 \int_{\tilde{\Omega}} (u^{\sigma}_{\eta}+\eta)^{m-1} \mathfrak{D}^3  (V_{\phi})^{r-4} |\nabla \phi|^4  dx\notag\\
&\leq \frac{1}{8} I_1 + \frac{1}{8} I_3 + 32m(r-1)^2 C_{\phi}^4  ( \| u^{\sigma}_{\eta}\|_{L^{\infty}(0,T'; L^{\infty}(\R^d) ) }+M)^{m-1} |\tilde{\Omega}|^{\frac{1}{r}} \| V_{\phi} \|_{L^r(\tilde{\Omega})}^{r-1}.\notag
\end{align}
Similarly,
\begin{align}
\label{est_II_3}
II_3
&\leq  \frac{1}{8} I_1' +  \frac{1}{8} I_3' +  32 \sigma (r-1)^2 C_{\phi}^4 |\tilde{\Omega}|^{\frac{1}{r}} \| V_{\phi} \|_{L^r(\tilde{\Omega})}^{r-1}.
\end{align}
The term $II_4$ can be divided into three parts
\begin{align*}
II_4
&= -m(m-1) \int_{\tilde{\Omega}} (u^{\sigma}_{\eta}+\eta)^{m-2} f'(\bar{\mathfrak{u}}) \nabla \bar{\mathfrak{u}} \cdot \mathfrak{D} \nabla \phi  ( V_{\phi})^{r-1}  dx   -m  \int_{\tilde{\Omega}} (u^{\sigma}_{\eta}+\eta)^{m-1} \nabla \mathfrak{D} \cdot \nabla \phi  ( V_{\phi})^{r-1}  dx \\
&\ \ \  - m(r-1)  \int_{\tilde{\Omega}} (u^{\sigma}_{\eta}+\eta)^{m-1} \mathfrak{D} \nabla \phi \cdot  ( V_{\phi})^{r-2} \nabla V_{\phi} dx\\
&=:II_{41}+II_{42}+II_{43}.
\end{align*}
Let us estimate these three terms of $II_4$.
Since $f'(\bar{\mathfrak{u}}) \leq 4e \tilde{\omega} e^{-\bar{\mathfrak{u}}^2}$, $|\nabla \bar{\mathfrak{u}}| \leq \mathfrak{D}^{\frac{1}{2}}$ and $|\nabla \phi| \leq C_{\phi} \phi^{\frac{3}{4}}$, we obtain that
\begin{align*}
 II_{41}
 \leq&  m(m-1) \int_{\tilde{\Omega}} 4 e \tilde{\omega} (u^{\sigma}_{\eta}+\eta)^{m-2} \mathfrak{D}^{\frac{3}{2}} C_{\phi} \phi^{\frac{3}{4}} (V_{\phi})^{r-1} dx\\
 \leq& \frac{1}{16} I_3' +  64^2 e^4 \tilde{\omega}^4 m^4(m-1)^4 \frac{1}{\sigma^3} ( \| u^{\sigma}_{\eta}\|_{L^{\infty}(0,T'; L^{\infty}(\R^d) )}+M)^{4(m-2)} C_{\phi}^4 \int_{\tilde{\Omega}} (V_{\phi})^{r-1} dx \\
 \leq& \frac{1}{16} I_3' +  64^2 e^4  m^4(m-1)^4 \frac{1}{\sigma^3} ( \| u^{\sigma}_{\eta}\|_{L^{\infty}(0,T'; L^{\infty}(\R^d) )}+M)^{4(m-2)} C_{\phi}^4 \| V_{\phi} \|_{L^r(\tilde{\Omega})}^{r-1} |\tilde{\Omega}|^{\frac{1}{r}}.
\end{align*}
For the second term,
\begin{align*}
II_{42}
\leq& 2m C_{\phi} \int_{\tilde{\Omega}} (u^{\sigma}_{\eta}+\eta)^{m-1} |\nabla \bar{\mathfrak{u}}| \ |D^2 \bar{\mathfrak{u}} | \phi^{\frac{3}{4}} |V_{\phi}|^{r-1}  dx \\
\leq& \frac{1}{8} \cdot 2m \int_{\tilde{\Omega}} (u^{\sigma}_{\eta}+\eta)^{m-1} \sum_{k=1}^d |\nabla \partial_{x_k} \bar{\mathfrak{u}}| \phi |V_{\phi}|^{r-1} dx +4m C_{\phi}^2 \int_{\tilde{\Omega}} (u^{\sigma}_{\eta}+\eta)^{m-1} \mathfrak{D} \phi^{\frac{1}{2}} |V_{\phi}|^{r-1} dx\\
\leq& \frac{1}{8} I_2 + \frac{1}{8} I_3 + 8m C_{\phi}^4 ( \| u^{\sigma}_{\eta}\|_{L^{\infty}(0,T'; L^{\infty} (\R^d) ) }+M)^{m-1} \int_{\tilde{\Omega}} |V_{\phi}|^{r-1}dx\\
\leq& \frac{1}{8} I_2 + \frac{1}{8} I_3 + 8m C_{\phi}^4 ( \| u^{\sigma}_{\eta}\|_{L^{\infty}(0,T'; L^{\infty} (\R^d) ) }+M)^{m-1}  \| V_{\phi} \|_{L^r(\tilde{\Omega})}^{r-1} |\tilde{\Omega}|^{\frac{1}{r}}.
\end{align*}
For the third term we deduce that
\begin{align*}
II_{43}
\leq& \frac{1}{8}  m(r-1)  \int_{\tilde{\Omega}} (u^{\sigma}_{\eta}+\eta)^{m-1} |V_{\phi}|^{r-2} |\nabla V_{\phi} |^2  dx \\
&+ 2m(r-1) C_{\phi}^2 \int_{\tilde{\Omega}} (u^{\sigma}_{\eta}+\eta)^{m-1} \mathfrak{D}^2 \phi^{\frac{3}{2}} (V_{\phi})^{r-2} dx \\
\leq& \frac{1}{8} I_1 + \frac{1}{8} I_3 + 2m(r-1)^2 C_{\phi}^4 ( \| u^{\sigma}_{\eta}\|_{L^{\infty}(0,T'; L^{\infty} (\R^d) ) }+M)^{m-1}  \| V_{\phi} \|_{L^r(\tilde{\Omega})}^{r-1} |\tilde{\Omega}|^{\frac{1}{r}}.
\end{align*}
Therefore,
\begin{align}
\label{est_II_4}
II_4 \leq& \frac{1}{8} I_1 + \frac{1}{8} I_2 + \frac{1}{4} I_3 + \frac{1}{16} I_3'\\
&+ 64^2 e^4  m^4(m-1)^4 \frac{1}{\sigma^3} ( \| u^{\sigma}_{\eta}\|_{L^{\infty}(0,T'; L^{\infty} (\R^d) ) }+M)^{4(m-2)} C_{\phi}^4  \| V_{\phi} \|_{L^r(\tilde{\Omega})}^{r-1} |\tilde{\Omega}|^{\frac{1}{r}} \notag  \\
&  +   8m C_{\phi}^4 ( \| u^{\sigma}_{\eta}\|_{L^{\infty}(0,T'; L^{\infty} (\R^d) ) }+M)^{m-1} \| V_{\phi} \|_{L^r(\tilde{\Omega})}^{r-1} |\tilde{\Omega}|^{\frac{1}{r}}  \notag  \\
&+ 2m(r-1)^2 C_{\phi}^4 ( \| u^{\sigma}_{\eta}\|_{L^{\infty}(0,T'; L^{\infty} (\R^d) ) }+M)^{m-1} \| V_{\phi} \|_{L^r(\tilde{\Omega})}^{r-1} |\tilde{\Omega}|^{\frac{1}{r}}. \notag
\end{align}
Similarly,
\begin{align}
\label{est_II_5}
II_5 &\leq  \frac{1}{8} I_1' + \frac{1}{8} I_2' +\frac{1}{16} I_3' + \left(32 \sigma C_{\phi}^4 + 32 \sigma (r-1)^2 C_{\phi}^4 \right) \| V_{\phi} \|_{L^r(\tilde{\Omega})}^{r-1} |\tilde{\Omega}|^{\frac{1}{r}}.
\end{align}
For $III$,
\begin{align}
\label{est_III}
III
&\leq \frac{1}{8} 2 \sigma \int_{\tilde{\Omega}} \sum_{i=1}^d |\partial_{ii} \bar{\mathfrak{u}}|^2 \phi |V_{\phi}|^{r-1}  dx  + \frac{64}{\sigma} \tilde{\omega}^2 e^2m^2(m-1)^2 ( \| u^{\sigma}_{\eta}\|_{L^{\infty}(0,T'; L^{\infty} (\R^d ) )}+M)^{2(m-2)} \int_{\tilde{\Omega}} \mathfrak{D} (V_{\phi})^r dx\\
&\leq \frac{1}{8} I_2' + \frac{1}{8} I_3'.\notag
\end{align}
We focus on $IV_1$.
\begin{align*}
IV_1 &\leq 4m \int_{\tilde{\Omega}} (u^{\sigma}_{\eta}+\eta)^{m-1}    |\nabla \bar{\mathfrak{u}} | \  |\nabla V_{\phi}| \  |V_{\phi}|^{r-1}  dx\\
&\leq \frac{m(r-1)}{8} \int_{\tilde{\Omega}} (u^{\sigma}_{\eta}+\eta)^{m-1}  |V_{\phi}|^{r-2} |\nabla V_{\phi}|^2 dx + \frac{32m}{r-1} \int_{\tilde{\Omega}} (u^{\sigma}_{\eta}+\eta)^{m-1} \mathfrak{D} |V_{\phi}|^r  dx.
\end{align*}
Choosing $r\geq 65$, we obtain that $\frac{32m}{r-1} \leq \frac{1}{2} m = \frac{1}{8} \cdot 4m$. Therefore,
\begin{align}
\label{est_IV_1}
IV_1 \leq \frac{1}{8} I_1 + \frac{1}{8} I_3.
\end{align}
Similarly,
\begin{align}
\label{est_IV_1_prime}
IV_1'&\leq  4\sigma \int_{\tilde{\Omega}}   |\nabla \bar{\mathfrak{u}} | \  |\nabla V_{\phi}| |V_{\phi}|^{r-1}  dx\\
&\leq \frac{\sigma(r-1)}{8} \int_{\tilde{\Omega}}  |V_{\phi}|^{r-2} |\nabla V_{\phi}|^2 dx
+ \frac{32 \sigma}{r-1} \int_{\tilde{\Omega}} \mathfrak{D} |V_{\phi}|^r  dx\notag\\
&\leq \frac{1}{8} I_1' + \frac{1}{8} I_3' .\notag
\end{align}
From $\left| \frac{f''(\bar{\mathfrak{u}})}{f'(\bar{\mathfrak{u}})} \right| \leq 2$, we deduce  that
\begin{align}
\label{est_IV_2}
IV_2 &\leq 4m C_{\phi}  \int_{\tilde{\Omega}} (u^{\sigma}_{\eta}+\eta)^{m-1} \mathfrak{D}^{\frac{1}{2}} \mathfrak{D} \phi^{\frac{3}{4}} |V_{\phi}|^{r-1} dx\\
&\leq \frac{1}{32} I_3 + 32 m C_{\phi}^2 \int_{\tilde{\Omega}}  (u^{\sigma}_{\eta}+\eta)^{m-1} \mathfrak{D}^2 \phi^{\frac{3}{2}} |V_{\phi}|^{r-2} dx\notag\\
&\leq   \frac{1}{16} I_3 +   64^2 m  C_{\phi}^4  \int_{\tilde{\Omega}}   (u^{\sigma}_{\eta}+\eta)^{m-1} \mathfrak{D}^3 \phi^3  |V_{\phi}|^{r-4} dx\notag\\
&\leq \frac{1}{16} I_3 +   64^2 m  C_{\phi}^4 ( \| u^{\sigma}_{\eta}\|_{L^{\infty}(0,T'; L^{\infty} (\R^d ))}+M)^{m-1} \| V_{\phi} \|_{L^r(\tilde{\Omega})}^{r-1} |\tilde{\Omega}|^{\frac{1}{r}}.\notag
\end{align}
Similarly,
\begin{align}
\label{est_IV_2_prime}
IV_2'  \leq \frac{1}{16} I_3' +   64^2 \sigma  C_{\phi}^4  \| V_{\phi} \|_{L^r(\tilde{\Omega})}^{r-1} |\tilde{\Omega}|^{\frac{1}{r}} .
\end{align}
Recall that for $q>d$
\begin{align*}
\| \nabla \Phi * u^{\sigma}_{\eta}\|_{L^{\infty}(0,T'; L^{\infty} (\R^d) ) }
\leq C \| \nabla \Phi * u^{\sigma}_{\eta}\|_{L^{\infty}(0,T'; W^{1,q}(\R^d))}
\leq C  \|   u^{\sigma}_{\eta}\|_{L^{\infty}(0,T'; L^q (\R^d))} \leq C,
\end{align*}
where $C$ appeared in this subsection is a positive constant dependent on $ \sigma$, $C_{\phi}$, $m$ and
$\|  u_0\|_{ L^{1} ( \R^d)\cap L^{\infty} ( \R^d) } $.
So we obtain that
\begin{align}
\label{est_J_1_1}
J_1^1 &\leq \frac{1}{8} I_1' + \frac{2 \| \nabla \Phi * u^{\sigma}_{\eta}\|_{L^{\infty}(0,T'; L^{\infty} (\R^d ) )}^2 }{\sigma(r-1)}
\int_{\tilde{\Omega}} |V_{\phi}|^r dx\\
&\leq  \frac{1}{8} I_1'
+ \frac{ \| \nabla \Phi * u^{\sigma}_{\eta}\|_{L^{\infty}(0,T'; L^{\infty} (\R^d ) )}^2 }{ 32 \sigma } \| V_{\phi} \|_{L^r(\tilde{\Omega})}^r.\notag
\end{align}
By the same argument, it follows that
\begin{align}
\label{est_J_1_2}
J_1^2
&\leq \frac{1}{8} I_3' + \frac{1}{2 \sigma} C_{\phi}^2 \int_{\tilde{\Omega}} \mathfrak{D}  \phi^{\frac{3}{2}} |\nabla \Phi * u^{\sigma}_{\eta} |^2 \ | V_{\phi} |^{r-2}  dx \\
&\leq \frac{1}{8} I_3' + \frac{1}{2 \sigma} C_{\phi}^2  \| \nabla \Phi * u^{\sigma}_{\eta}\|_{L^{\infty}(0,T'; L^{\infty} (\R^d ) )}^2 \int_{\tilde{\Omega}}  | V_{\phi} |^{r-1} dx \notag\\
&\leq \frac{1}{8} I_3' + \frac{1}{2 \sigma} C_{\phi}^2  \| \nabla \Phi * u^{\sigma}_{\eta}\|_{L^{\infty}(0,T'; L^{\infty} (\R^d ) )}^2  \| V_{\phi} \|_{L^r(\tilde{\Omega})}^{r-1} |\tilde{\Omega}|^{\frac{1}{r}}. \notag
\end{align}
Next, we analyze $J_2$.
\begin{align*}
J_2 &= 2   \sum_{k=1}^d  \int_{\tilde{\Omega}} \nabla \partial_{x_k} \bar{\mathfrak{u}} \cdot \nabla \Phi * u^{\sigma}_{\eta} \partial_{x_k}  \bar{\mathfrak{u}} \phi  (V_{\phi})^{r-1}   dx  + 2   \sum_{k=1}^d  \int_{\tilde{\Omega}} \nabla  \bar{\mathfrak{u}} \cdot \nabla \Phi * u^{\sigma}_{\eta} \partial_{x_k}  \partial_{x_k} \bar{\mathfrak{u}} \phi  (V_{\phi})^{r-1}   dx\\
&\ \ \ + 2   \sum_{k=1}^d  \int_{\tilde{\Omega}} \nabla  \bar{\mathfrak{u}} \cdot \nabla \Phi * u^{\sigma}_{\eta}   \partial_{x_k} \bar{\mathfrak{u}}  \partial_{x_k} \phi  (V_{\phi})^{r-1}   dx  + 2(r-1) \sum_{k=1}^d  \int_{\tilde{\Omega}} \nabla  \bar{\mathfrak{u}} \cdot \nabla \Phi * u^{\sigma}_{\eta}   \partial_{x_k} \bar{\mathfrak{u}}   \phi  (V_{\phi})^{r-2} \partial_{x_k} V_{\phi}    dx\\
&=:J_{21}+J_{22}+J_{23}+J_{24}.
\end{align*}
Now we are going to estimate these four terms. For the first term,
\begin{align*}
 J_{21}
\leq& \frac{1}{8}\cdot 2 \sigma \int_{\tilde{\Omega}} \sum_{k=1}^d  |  \nabla \partial_{x_k} \bar{\mathfrak{u}} |^2 \phi |V_{\phi}|^{r-1} dx  + \frac{4}{\sigma} \int_{\tilde{\Omega}} | \nabla \Phi * u^{\sigma}_{\eta} |^2 \mathfrak{D} \phi |V_{\phi}|^{r-1} dx \\
\leq& \frac{1}{8} I_2' + \frac{4  \| \nabla \Phi * u^{\sigma}_{\eta}\|_{L^{\infty}(0,T'; L^{\infty} (\R^d ) )}^2 }{\sigma} \| V_{\phi} \|_{L^r(\tilde{\Omega})}^r.
\end{align*}
For the second term,
\begin{align*}
 J_{22}
\leq& \frac{1}{8}\cdot 2 \sigma \int_{\tilde{\Omega}} \sum_{k=1}^d  |  \nabla \partial_{x_k} \bar{\mathfrak{u}} |^2 \phi |V_{\phi}|^{r-1} dx  + \frac{4}{\sigma} \int_{\tilde{\Omega}} | \nabla \Phi * u^{\sigma}_{\eta} |^2 \mathfrak{D} \phi |V_{\phi}|^{r-1} dx \\
\leq& \frac{1}{8} I_2' + \frac{4  \| \nabla \Phi * u^{\sigma}_{\eta}\|_{L^{\infty}(0,T'; L^{\infty} (\R^d ) )}^2 }{\sigma} \| V_{\phi} \|_{L^r(\tilde{\Omega})}^r.
\end{align*}
For the third term,
\begin{align*}
J_{23}
\leq& 2 C_{\phi}  \int_{\tilde{\Omega}} \mathfrak{D} |\nabla  \Phi * u^{\sigma}_{\eta} | \phi^{\frac{3}{4}} |V_{\phi}|^{r-1} dx \\
\leq& \frac{1}{16}\cdot 4 \sigma \int_{\tilde{\Omega}} \mathfrak{D} |V_{\phi}|^r dx  + \frac{4C_{\phi}^2}{\sigma} \int_{\tilde{\Omega}} \mathfrak{D} |\nabla \Phi * u^{\sigma}_{\eta} |^2  \phi^{\frac{3}{2}} |V_{\phi}|^{r-2} dx \\
\leq& \frac{1}{16} I_3' + \frac{4 C_{\phi}^2 \| \nabla \Phi * u^{\sigma}_{\eta}\|_{L^{\infty}(0,T'; L^{\infty} (\R^d ) )}^2 }{\sigma} |\tilde{\Omega}|^{\frac{1}{r}} \| V_{\phi} \|_{L^r(\tilde{\Omega})}^r.
\end{align*}
For the last term,
\begin{align*}
J_{24}
\leq& \frac{1}{16} \sigma(r-1) \int_{\tilde{\Omega}} |V_{\phi}|^{r-2} |\nabla V_{\phi}|^2   dx  +\frac{16 (r-1)}{\sigma}  \int_{\tilde{\Omega}} \phi^2 |\nabla \bar{\mathfrak{u}} |^4 |\nabla  \Phi * u^{\sigma}_{\eta} |^2   |V_{\phi}|^{r-2} dx \\
\leq& \frac{1}{16} I_1' + \frac{16 r \| \nabla \Phi * u^{\sigma}_{\eta}\|_{L^{\infty}(0,T'; L^{\infty} (\R^d ) )}^2 }{\sigma}  \| V_{\phi} \|_{L^r(\tilde{\Omega})}^r.
\end{align*}
Combining these estimates, we obtain that
\begin{align}
\label{est_J_2}
J_2 &\leq \frac{1}{16} I_1' + \frac{1}{4} I_2' + \frac{1}{16} I_3' + \frac{16  r \| \nabla \Phi * u^{\sigma}_{\eta}\|_{L^{\infty}((0,T')\times \R^d  )}^2 }{\sigma}  \| V_{\phi} \|_{L^r(\tilde{\Omega})}^r  + \frac{4 C_{\phi}^2 \| \nabla \Phi * u^{\sigma}_{\eta}\|_{L^{\infty}((0,T')\times\R^d  )}^2 }{\sigma} |\tilde{\Omega}|^{\frac{1}{r}} \| V_{\phi} \|_{L^r(\tilde{\Omega})}^{r-1}.
\end{align}
We handle the last two terms $J_3$ and $J_4$ in \eqref{inequality_big_V_phi_L_r_norm_simpl} together.
The term $J_4$ can be rewritten as
\begin{align*}
J_4
&=
2   \int_{\tilde{\Omega}}\left( 1 - \frac{f''(\bar{\mathfrak{u}})}{(f'(\bar{\mathfrak{u}}))^2} u^{\sigma}_{\eta} \right)  (V_{\phi})^r \Delta \Phi * u^{\sigma}_{\eta} dx +2 \int_{\tilde{\Omega}} \frac{ u^{\sigma}_{\eta} \Delta \bar{\mathfrak{u}} }{f'(\bar{\mathfrak{u}})} \phi (V_{\phi})^{r-1} \Delta \Phi * u^{\sigma}_{\eta} dx \\
&\;\;\;\;+2  \int_{\tilde{\Omega}} \frac{ u^{\sigma}_{\eta} \nabla \bar{\mathfrak{u}} }{f'(\bar{\mathfrak{u}})} \cdot  \nabla \phi (V_{\phi})^{r-1} \Delta \Phi * u^{\sigma}_{\eta} dx +2 (r-1)  \int_{\tilde{\Omega}}  \frac{ u^{\sigma}_{\eta} \nabla \bar{\mathfrak{u}} }{f'(\bar{\mathfrak{u}})} \cdot  \phi (V_{\phi})^{r-2} \nabla V_{\phi} \Delta \Phi * u^{\sigma}_{\eta} dx \\
&=:- 2  \int_{\tilde{\Omega}}\left( \frac{f''(\bar{\mathfrak{u}})}{(f'(\bar{\mathfrak{u}}))^2} u^{\sigma}_{\eta}  -1  \right)  (V_{\phi})^r \Delta \Phi * u^{\sigma}_{\eta} dx +J_{41} + J_{42} + J_{43}.
\end{align*}
From the definition of $J_3$, we deduce that
\begin{align*}
J_3 + J_4 &= 2 \int_{\tilde{\Omega}} \left( \frac{f''(\bar{\mathfrak{u}})}{(f'(\bar{\mathfrak{u}}))^2}  u^{\sigma}_{\eta} -1  \right) \Delta \Phi * u^{\sigma}_{\eta} (V_{\phi})^r  dx \\
&\;\;\;\;- 2  \int_{\tilde{\Omega}}\left( \frac{f''(\bar{\mathfrak{u}})}{(f'(\bar{\mathfrak{u}}))^2} u^{\sigma}_{\eta}  -1  \right)  (V_{\phi})^r \Delta \Phi * u^{\sigma}_{\eta} dx +J_{41} + J_{42} + J_{43}\\
&= J_{41} + J_{42} + J_{43}.
\end{align*}
Let us estimate these three terms. For $J_{41}$,
\begin{align*}
J_{41}
&\leq \frac{1}{8} \cdot 2 \sigma  \int_{\tilde{\Omega}}  | \Delta \bar{\mathfrak{u}} |^2 \phi |V_{\phi}|^{r-1} dx  + \frac{e^{2 (\bar{\mathfrak{u}}^2-1)}}{4 \sigma \tilde{\omega}^2 } \int_{\tilde{\Omega}} ( u^{\sigma}_{\eta} )^2 \phi |V_{\phi}|^{r-1} |\Delta \Phi * u^{\sigma}_{\eta}|^2  dx \\
&= \frac{1}{8} I_2' +  \frac{1}{4 \sigma \tilde{\omega}^2 }  \| u^{\sigma}_{\eta}  \Delta \Phi * u^{\sigma}_{\eta}   \phi^{\frac{1}{2}} \|_{L^{2r}(\tilde{\Omega})}^2 \| V_{\phi} \|_{L^r(\tilde{\Omega})}^{r-1}.
\end{align*}
For $J_{42}$,
\begin{align*}
J_{42}
&\leq \int_{\tilde{\Omega}} |V_{\phi}|^r dx + \int_{\tilde{\Omega}} \frac{ (u^{\sigma}_{\eta})^2 |\nabla \bar{\mathfrak{u}}|^2 }{16  \tilde{\omega}^2 e^{ 2(1-\bar{\mathfrak{u}}^2) } }   |\nabla \phi|^2 |V_{\phi}|^{r-2} |\Delta \Phi * u^{\sigma}_{\eta}|^2 dx\\
&\leq \int_{\tilde{\Omega}} |V_{\phi}|^r dx + \frac{ C_{\phi}^2 }{16  \tilde{\omega}^2} \int_{\tilde{\Omega}} (u^{\sigma}_{\eta})^2  \phi^{\frac{1}{2}} |V_{\phi}|^{r-1} |\Delta \Phi * u^{\sigma}_{\eta}|^2 dx\\
&\leq \| V_{\phi} \|_{L^r(\tilde{\Omega})}^r  + \frac{ C_{\phi}^2 }{16  \tilde{\omega}^2}  \| u^{\sigma}_{\eta} \Delta \Phi * u^{\sigma}_{\eta} \phi^{\frac{1}{4}} \|_{L^{2r}(\tilde{\Omega})}^2  \| V_{\phi} \|_{L^r(\tilde{\Omega})}^{r-1}.
\end{align*}
For $J_{43}$,
\begin{align*}
J_{43}
&\leq \frac{1}{16} \sigma (r-1) \int_{\tilde{\Omega}} |V_{\phi}|^{r-2} |\nabla V_{\phi} |^2 dx + \frac{r-1}{\sigma \tilde{\omega}^2 e^{2(1-\bar{\mathfrak{u}})}}  \int_{\tilde{\Omega}} (u^{\sigma}_{\eta})^2 |\nabla \bar{\mathfrak{u}} |^2 \phi^2 |V_{\phi}|^{r-2} | \Delta \Phi * u^{\sigma}_{\eta} |^2  dx \\
&\leq \frac{1}{16}I_1' +  \frac{r-1}{\sigma \tilde{\omega}^2} \| u^{\sigma}_{\eta} \Delta \Phi * u^{\sigma}_{\eta} \phi^{\frac{1}{2}} \|_{L^{2r}(\tilde{\Omega})}^2  \| V_{\phi} \|_{L^r(\tilde{\Omega})}^{r-1}.
\end{align*}
Therefore, we obtain that
\begin{align}
\label{est_J_3_and_J_4}
J_3 + J_4 \leq \frac{1}{16} I_1' + \frac{1}{8} I_2' +rC \| u^{\sigma}_{\eta} \Delta \Phi * u^{\sigma}_{\eta} \phi^{\frac{1}{4}} \|_{L^{2r}(\tilde{\Omega})}^2  \| V_{\phi} \|_{L^r(\tilde{\Omega})}^{r-1}.
\end{align}
Moreover, by interpolation inequality, it holds
\begin{align*}
\| u^{\sigma}_{\eta} \Delta \Phi * u^{\sigma}_{\eta} \phi^{\frac{1}{4}} \|_{L^{2r}(\tilde{\Omega})} &\leq \|  \Delta \Phi * u^{\sigma}_{\eta}  \|_{L^{4r}(\tilde{\Omega})} \| u^{\sigma}_{\eta}  \phi^{\frac{1}{4}} \|_{L^{4r}(\tilde{\Omega})}\\
&\leq \| u^{\sigma}_{\eta}   \|_{L^{4r}(\tilde{\Omega})}  \| u^{\sigma}_{\eta}   \|_{L^{4r}(\tilde{\Omega})} \| \phi^{\frac{1}{4}}  \|_{{L^\infty}(\tilde{\Omega})}
\leq C.
\end{align*}
Combining \eqref{est_II_1} - \eqref{est_J_3_and_J_4},
we get
\begin{align*}
\frac{1}{r} \frac{d}{dt} \| V_{\phi} \|_{L^r(\tilde{\Omega})}^r \leq - \frac{\sigma(r-1)}{r^2}    \int_{\tilde{\Omega}} |\nabla (V_{\phi})^{\frac{r}{2}}(t, x) |^2 dx + r^2 \mathcal{M} \| V_{\phi} \|_{L^r(\tilde{\Omega})}^r   +r^2 \mathcal{M},
\end{align*}
where $\mathcal{M}$ is a positive constant which depends only on $d$, $\sigma$, $C_{\phi}$, $m$,
$\|  u_0\|_{ L^1(\R^d)\cap L^\infty(\R^d) }$ and $|\tilde{\Omega}| $.
This implies that
\begin{align}
\label{inequality_V_phi}
\frac{d}{dt} \int_{\tilde{\Omega}} |V_{\phi}|^r dx  \leq  - \frac{\sigma(r-1)}{r}    \int_{\tilde{\Omega}} |\nabla (V_{\phi})^{\frac{r}{2}}(t, x) |^2 dx + r^3 \mathcal{M} \int_{\tilde{\Omega}} |V_{\phi}|^r dx +r^3 \mathcal{M}.
\end{align}
\subsubsection{ Moser's iteration.}
It remains to prove that $\| V_{\phi} \|_{L^{\infty}(0,T'; L^{\infty}(\tilde{\Omega}) ) }$ is bounded.
Let $q_k:= 2^k+4d+4$ for $k\in \N$ and $k\geq 6$.  We infer
\begin{align*}
\frac{d}{dt} \int_{\tilde{\Omega}} |V_{\phi}|^{q_k} dx  \leq  - \frac{\sigma(q_k-1)}{q_k}    \int_{\tilde{\Omega}} |\nabla (V_{\phi})^{\frac{q_k}{2}} |^2 dx + q_k^3 \mathcal{M} \int_{\tilde{\Omega}} |V_{\phi}|^{q_k} dx +q_k^3 \mathcal{M}.
\end{align*}
Furthermore,
\begin{align*}
q_k^3 \mathcal{M}\int_{\tilde{\Omega}} |V_{\phi}|^{q_k} dx
&\leq q_k^3 \mathcal{M} S_d^{-\theta_1} \| V_{\phi} \|_{L^{q_{k-1}} (\tilde{\Omega})}^{( 1 - \theta_1 )q_k } \| \nabla (V_{\phi})^{\frac{q_k}{2}} \|_{L^2 (\tilde{\Omega})}^{2\theta_1}\\
&\leq \frac{\sigma (q_k-1)}{2q_k} \int_{\tilde{\Omega}} | \nabla (V_{\phi})^{\frac{q_k}{2}}|^2 dx
+ (1-\theta_1) \left(  \frac{\sigma (q_k-1)}{2q_k} \cdot \frac{1}{\theta_1}  \right)^{-\frac{\theta_1}{1-\theta_1}}
(S_d^{-\theta_1}q_k^3 \mathcal{M})^\frac{1}{1-\theta_1}  \| V_{\phi} \|_{L^{q_{k-1}} (\tilde{\Omega})}^{q_k},
\end{align*}
where $ \displaystyle \theta_1= \frac{\frac{1}{q_{k-1}} - \frac{1}{q_k}}{ \frac{1}{q_{k-1}} - \frac{d-2}{dq_k}}$.
With this observation at hand, we have
\begin{align*}
\frac{d}{dt} \int_{\tilde{\Omega}} |V_{\phi}|^{q_k} dx  \leq&  -\frac{\sigma (q_k-1)}{2q_k} \int_{\tilde{\Omega}} | \nabla (V_{\phi})^{\frac{q_k}{2}}|^2 dx  + q_k^3 \mathcal{M}  \\
&+ \left(  \frac{\sigma (q_k-1)}{ 2 q_k} \cdot \frac{1}{\theta_1}  \right)^{-\frac{\theta_1}{1-\theta_1}} (1-\theta_1) \left( S_d^{-\theta_1} q_k^3 \mathcal{M} \right)^{\frac{1}{1-\theta_1}}  \| V_{\phi} \|_{L^{q_{k-1}} (\tilde{\Omega})}^{q_k}.
\end{align*}
It is easy to see that
\begin{align*}
\frac{d}{dt} \int_{\tilde{\Omega}} |V_{\phi}|^{q_k} dx  \leq& - \int_{\tilde{\Omega}} |V_{\phi}|^{q_k} dx +q_k^3 \mathcal{M} \\
&+ (1-\theta_1) \left(  \frac{\sigma (q_k-1)}{2q_k} \cdot \frac{1}{\theta_1}  \right)^{-\frac{\theta_1}{1-\theta_1}} S_d^{\frac{-\theta_1}{1-\theta_1}} \left( 1+ (q_k^3\mathcal{M})^{\frac{1}{1-\theta_1}} \right) \| V_{\phi} \|_{L^{q_{k-1}} (\tilde{\Omega})}^{q_k}\\
\leq& - \int_{\tilde{\Omega}} |V_{\phi}|^{q_k} dx +q_k^3 \mathcal{M} + C q_k^{\frac{3}{1-\theta_1}} \| V_{\phi} \|_{L^{q_{k-1}}(\tilde{\Omega})}^{q_k},
\end{align*}
where $C$ appeared in this subsection denotes a positive constant independent of $q_k$.
By simple calculations, we derive that $\frac{1}{1-\theta_1}\leq d$.
Let $y_k(t):= \| V_{\phi}(t, \cdot) \|_{L^{q_{k}}(\tilde{\Omega})}^{q_k} $. Since $\frac{q_k}{q_{k-1}} \leq 2$, we obtain that
\begin{align*}
y_k'(t) \leq  -y_k(t)   + C q_k^{3d}  (y_{k-1} (t))^{\frac{q_k}{q_{k-1}}} +C q_k^{3d}  \leq -y_k(t)   + C q_k^{3d} \max \{ 1, y_{k-1}^2(t) \}  .
\end{align*}
This implies that
\begin{align*}
e^t y_k(t) & \leq  y_k(0)+ C (4d)^{3d}16^{dk}  \max \{ 1, \sup_{t\geq 0} y_{k-1}^2(t) \}  (e^t-1).
\end{align*}
We know that
\begin{align*}
y_k(0)&=  \|\mathfrak{D}(0)\phi \|_{L^{q_{k}}(\tilde{\Omega})}^{q_k} \leq K_0^{q_k}= (K_0^{\frac{q_k}{2^k}})^{2^k}:= K^{2^k},
\end{align*}
where $K_0:= \max \{ 1, \|\mathfrak{D}(0)\phi \|_{L^1(\tilde{\Omega})}, \|\mathfrak{D}(0)\phi \|_{L^{\infty}(\tilde{\Omega})} \}$.
Defining $a_k:= C (4d)^{3d}16^{dk}$, it holds
\begin{align*}
y_k(t)
&\leq   y_k(0) +  a_k  \max \{ 1, \sup_{t\geq 0} y_{k-1}^2(t) \}  \\
&\leq a_k (K^{2^k} + \sup_{t\geq 0}(y_{k-1}^2(t) ) \\
&\leq 2 a_k \max\{ K^{2^k} , \sup_{t\geq 0}(y_{k-1}^2(t)  )\}.
\end{align*}
Therefore, we infer
\begin{align*}
y_k(t)
&\leq \prod_{j=0}^{k-1} (2a_{k-j})^{2^j} \max \{ K^{2^k}, \sup_{t\geq 0}(y_0^{2^k}(t))\}\\
&\leq \prod_{j=0}^{k-1} \left(2 C (4d)^{3d}16^{dk} \right)^{2^j}  \prod_{j=0}^{k-1}
\left( 16^{-dj} \right)^{2^j}  \max \{ K^{2^k}, \sup_{t\geq 0}(y_0^{2^k}(t))\} \\
&\leq \left( 2 C (4d)^{3d} \right)^{2^k-1}  \left( 16^d \right)^{2^{k+1} -k -2 }  \max \{ K^{2^k}, \sup_{t\geq 0}(y_0^{2^k}(t))\}.
\end{align*}
Hence,
\begin{align*}
\|V_{\phi}(t,\cdot) \|_{L^{q_{k}}(\tilde{\Omega})} \leq  \left( 2 C (4d)^{3d} \right)^{\frac{2^k-1}{q_k}}  \left( 16^d \right)^{ \frac{2^{k+1}-k-2}{q_k}} \max \{ K^{\frac{2^k}{q_k}}, \sup_{t\geq 0}(y_0^{\frac{2^k}{q_k}}(t))\}.
\end{align*}
Since $q_k= 2^k+4d+4$ for $k \in \N$ and $k\geq 6$, it follows that $\frac{2^k-1}{q_k}\leq 1$, $\frac{2^{k+1}-k-2}{q_k}\leq 2$ and $\frac{2^k}{q_k}\leq 1$.
As a consequence,
\begin{align*}
\|V_{\phi}(t,\cdot) \|_{L^{q_{k}}(\tilde{\Omega})} \leq  2 C (4d)^{3d}  16^{2d} \max \{ K, \sup_{t\geq 0}(y_0(t))\}.
\end{align*}
Therefore, we get
\begin{align}
\label{Bound_for_V_phi}
\|V_{\phi} \|_{L^{\infty}(0,T'; L^{\infty}(\tilde{\Omega}))} \leq  2 C (4d)^{3d}  16^{2d}  \max \{ K, \sup_{t\geq 0}(y_0(t))\} =: C_{\infty},
\end{align}
where $C_{\infty}$ is independent of $\eta$. Recall that $u^{\sigma}_{\eta} = f(\bar{\mathfrak{u}})$ and $\mathfrak{D} = |\nabla \bar{\mathfrak{u}}|^2 $. We conclude that
\begin{align}
\label{bound_grad_u_eta}
\sup_{0<t<T'} \| \nabla u^{\sigma}_{\eta}  \|_{L^{\infty}(\tilde{\Omega}_k)} &= \sup_{0<t<T'} \| f'(\bar{\mathfrak{u}}) \nabla \bar{\mathfrak{u}}   \|_{L^{\infty}(\tilde{\Omega}_k)}\\
&\leq   \sup_{0<t<T'} \| f'(\bar{\mathfrak{u}})    \|_{L^{\infty}(\tilde{\Omega}_k)} \left( \sup_{0<t<T'} \| \mathfrak{D} \|_{L^{\infty}(\tilde{\Omega}_k)}  \right)^{\frac{1}{2}} \notag \\
&\leq 4 e \tilde{\omega} (C_{\infty})^{\frac{1}{2}}. \notag
\end{align}
By repeating a similar argument in $\tilde{\Omega}_k$ for $k \in \{ 1, \ldots , k_0 \} $, we obtain the bound for
$\sup_{0<t<T'} \| \nabla u^{\sigma}_{\eta} \|_{L^{\infty}(\R^d)}$.

This finishes the proof of Proposition \ref{grau}.

\subsection{Global Existence of the approximated solutions}\label{global_existence_approximate}
This subsection is devoted to establishing global well-posedness of solution to the system \eqref{u_sigma_approximation}.
\begin{proposition}\label{global}
For any $T>0$, let the assumptions in Theorem \ref{theorem_u_sigma} hold, then the problem \eqref{u_sigma_approximation} has a unique solution $u^{\sigma}_{\eta}\in L^\infty(0,T;H^s(\R^d))$.
\end{proposition}
\begin{proof}
Define
\begin{align*}
T_{\max} := \sup \Big\{ T' \in (0,T):  \text{ the problem} \ \eqref{u_sigma_approximation} \ \text{possesses solution} \ u \in L^{\infty}(0,T';H^s(\R^d)),
\ s \in \Big(\frac{d}{2} +3, \infty \Big) \cap \N \Big\}.
\end{align*}
In order to obtain $T_{\max}=T$, we only need to prove that
\begin{align*}
\| u^{\sigma}_{\eta} \|_{L^{\infty}(0,T'; H^s(\R^d))} \leq C, \;\;\; \forall T' < T_{\max},
\end{align*}
where $C$ appeared in this subsection is a positive constant which depends on $\sigma$, $d$, $m$, $s$,
$\|u_0^\sigma\|_{H^s(\R^d)}$ and $\eta$.

In the following estimate, we will use the estimates obtained in Propositions \ref{prop_local_approximate} and \ref{grau}, i.e.
$$
\|u^\sigma_\eta\|_{L^\infty(0,T;W^{1,\infty}(\R^d))}\leq C (\sigma).
$$

We apply $D^{\alpha}$ $(|\alpha|\leq s)$ to the system \eqref{u_sigma_approximation}, multiply resulting equation by $D^{\alpha} u^{\sigma}_{\eta}$ and integrate it over $\R^d$ in order to obtain that
\begin{align*}
&\frac{1}{2} \frac{d}{dt}  \int_{\R^d} | D^{\alpha}u^{\sigma}_{\eta} |^2 dx + \sigma  \int_{\R^d} |\nabla D^{\alpha}u^{\sigma}_{\eta} |^2 dx \\
=&   \int_{\R^d}   D^{\alpha} (u^{\sigma}_{\eta} \nabla \Phi * u^{\sigma}_{\eta} ) \cdot \nabla D^{\alpha} u^{\sigma}_{\eta} \, dx
 -  \int_{\R^d}  D^{\alpha} \nabla ((u^{\sigma}_{\eta}+ \eta)^{m} )  \cdot \nabla D^{\alpha} u^{\sigma}_{\eta} \, dx\\
=:&I_1+I_2.
\end{align*}
The term $I_1$ can be estimated as follows
\begin{align*}
I_1
\leq&  C \| \nabla D^{\alpha} u^{\sigma}_{\eta} \|_{L^2(\R^d)}   \| D^{\alpha}( \nabla \Phi * u^{\sigma}_{\eta}) \|_{L^2(\R^d)} + C \| \nabla D^{\alpha} u^{\sigma}_{\eta} \|_{L^2(\R^d)}  \| D^{\alpha} u^{\sigma}_{\eta}  \|_{L^2(\R^d)} \\
\leq& \frac{\sigma}{4} \int_{\R^d} |\nabla D^{\alpha} u^{\sigma}_{\eta}|^2 dx + C \int_{\R^d} |D^{\alpha} u^{\sigma}_{\eta} |^2 dx.
\end{align*}
Now we consider the term $I_2$.
\begin{align*}
I_2
=&  - m \int_{\R^d}   (u^{\sigma}_{\eta}+ \eta)^{m-1} | \nabla D^{\alpha} u^{\sigma}_{\eta}|^2 dx - m \int_{\R^d} \left(  D^{\alpha} \left( (u^{\sigma}_{\eta}+ \eta)^{m-1} \nabla u^{\sigma}_{\eta} \right) - (u^{\sigma}_{\eta}+ \eta)^{m-1} \nabla D^{\alpha} u^{\sigma}_{\eta} \right) \cdot \nabla D^{\alpha} u^{\sigma}_{\eta} \, dx\\
\leq& - m \int_{\R^d}   (u^{\sigma}_{\eta}+ \eta)^{m-1} | \nabla D^{\alpha} u^{\sigma}_{\eta}|^2 dx +  \frac{\sigma}{4} \int_{\R^d} |\nabla D^{\alpha} u^{\sigma}_{\eta}|^2 dx + C \int_{\R^d} \left|  D^{\alpha} \left( (u^{\sigma}_{\eta}+ \eta)^{m-1} \nabla u^{\sigma}_{\eta} \right) - (u^{\sigma}_{\eta}+ \eta)^{m-1} \nabla D^{\alpha} u^{\sigma}_{\eta} \right|^2 dx\\
\leq& - m \int_{\R^d}   (u^{\sigma}_{\eta}+ \eta)^{m-1} | \nabla D^{\alpha} u^{\sigma}_{\eta}|^2 dx +  \frac{\sigma}{4} \int_{\R^d} |\nabla D^{\alpha} u^{\sigma}_{\eta}|^2 dx   \\
& + C \| \nabla (u^{\sigma}_{\eta}+ \eta)^{m-1}\|_{L^{\infty}(\R^d)}^2 \| D^{\alpha} u^{\sigma}_{\eta} \|_{L^2(\R^d)}^2 + C \| \nabla u^{\sigma}_{\eta}  \|_{L^{\infty}(\R^d)}^2  \|  D^{\alpha} (u^{\sigma}_{\eta}+ \eta)^{m-1} \|_{L^2(\R^d)}^2.
\end{align*}
From these estimates, we derive that
\begin{align*}
&\frac{1}{2} \frac{d}{dt} \| u^{\sigma}_{\eta} \|_{H^s(\R^d)}^2 + \frac{\sigma}{2} \sum_{|\alpha|=0}^s \int_{\R^d} | \nabla D^{\alpha}  u^{\sigma}_{\eta} |^2 dx + m  \sum_{|\alpha|=0}^s \int_{\R^d} ( u^{\sigma}_{\eta}+ \eta )^{m-1} | \nabla D^{\alpha}  u^{\sigma}_{\eta} |^2 dx \\
\leq& C \| u^{\sigma}_{\eta} \|_{H^s(\R^d)}^2 + C \| u^{\sigma}_{\eta} + \eta \|_{L^{\infty}(\R^d)}^{2(m-2)} \| \nabla u^{\sigma}_{\eta} \|_{L^{\infty}(\R^d)}^2  \| u^{\sigma}_{\eta} \|_{H^s(\R^d)}^2
+ C \sum_{|\alpha|=0}^s \| D^{\alpha} ( u^{\sigma}_{\eta} + \eta )^{m-1}\|_{L^2(\R^d)}^2\\
\leq& C \| u^{\sigma}_{\eta} \|_{H^s(\R^d)}^2  +C \sum_{|\alpha|=0}^s\| D^{\alpha} (u^{\sigma}_{\eta} + \eta)^{m-[m]}\|_{L^2(\R^d)}^2\leq  C \| u^{\sigma}_{\eta} \|_{H^s(\R^d)}^2,
\end{align*}
where $C$ depends on $\sigma$ and $\eta$.
By Gr\"onwall's inequality, we deduce that
\begin{align*}
\sup_{t\in (0,T')} \| u^{\sigma}_{\eta} \|_{H^s(\R^d)}^2 \leq \left( C+\| u_0^\sigma \|_{H^s(\R^d)}^2 \right) e^{C T_{\max}} \ \text{for any} \ T'< T_{\max}.
\end{align*}
\end{proof}

\subsection{Weak solutions to problem \texorpdfstring{ \eqref{generalized_equation_u_sigma}  }{Limit of u sigma eta as eta goes to zero} }
Our goal of this subsection is to take the limit $\eta \rightarrow 0$ in the weak formulation of \eqref{u_sigma_approximation}.
Let us start with the following lemma.
\begin{lemma}
\label{lemma_estimates_u_sigma_eta}
Let the assumptions in Theorem \ref{theorem_u_sigma} hold
and $u^{\sigma}_{\eta}$ be the solution of \eqref{u_sigma_approximation}. Then the following estimates hold
\begin{align*}
&(i)\;\| u^{\sigma}_{\eta} \|_{L^{\infty}(0,T; W^{1, \infty}(\R^d))   } + \| u^{\sigma}_{\eta} \|_{L^{\infty}(0,T; L^q(\R^d))   } \leq C,\;\;\;
\forall q \in [1, \infty],\\
&(ii)\;\|\Phi*u^{\sigma}_{\eta} \|_{L^{\infty}(0,T; W^{2, q}(\R^d))   } \leq C,\;\;\;\forall q \in (1, \infty),\\
&(iii)\;\| u^{\sigma}_{\eta} \|_{L^2(0,T; H^1(\R^d))   } \leq C,
\end{align*}
where $C$ is a positive constant independent of $\eta$.
\end{lemma}
\begin{proof}
The estimate $(i)$ is a direct consequence of the Moser iteration for $\| u^{\sigma}_{\eta} \|_{L^{\infty}(0,T; L^{ \infty}(\R^d))   } $ in Proposition \ref{prop_local_approximate} and the Bernstein type estimate for $\| \nabla u^{\sigma}_{\eta} \|_{L^{\infty}(0,T; L^{\infty}(\R^d))   } $ in Proposition \ref{grau}.
Using basic results from the theory of elliptic partial differential equation and $(i)$,  we obtain $(ii)$ directly.
Multiplying equation \eqref{u_sigma_approximation} by $u^{\sigma}_{\eta}$ and integrating over $\R^d$, we deduce that
\begin{align*}
&\frac{1}{2} \frac{d}{dt} \int_{\R^d} |u^{\sigma}_{\eta}|^2 dx +  \sigma  \int_{\R^d}  |\nabla u^{\sigma}_{\eta} |^2 dx +m \int_{\R^d}  (u^{\sigma}_{\eta}+\eta)^{m-1} |\nabla  u^{\sigma}_{\eta}|^2   dx \\
 \leq& \frac{\sigma}{2} \int_{\R^d} |\nabla u^{\sigma}_{\eta}|^2 dx + C \| u^{\sigma}_{\eta} \|_{L^{\infty}(\R^d)}^2 \| \nabla \Phi * u^{\sigma}_{\eta} \|_{L^2(\R^d)}^2.
\end{align*}
And combining $(i)$ and $(ii)$, we derive $(iii)$.
\end{proof}
With this lemma at hand, we are ready to prove the following result.
\begin{lemma}\label{leta}
Let the assumptions in Theorem \ref{theorem_u_sigma} hold,
then there exist a subsequence of $(u^{\sigma}_{\eta})_\eta$  and a function $u^{\sigma}$ such that as $\eta\rightarrow 0$,
\begin{align}
&u^{\sigma}_{\eta} \overset{*}\rightharpoonup u^{\sigma}\quad {\rm{in}}\quad L^{l}((0,T)\times\R^d)\;\;\;l\in(1,\infty],\label{i}\\
&u^{\sigma}_{\eta} \rightarrow u^{\sigma}\quad {\rm{in}}\quad L^2(0,T; L^2( \R^d) ) ,\label{ii}\\
&\nabla ( u^{\sigma}_{\eta} +\eta)^m \rightharpoonup  \nabla  ( u^{\sigma})^m\quad  {\rm{in}}\quad L^2(0,T; L^2(\R^d)  ),\label{iii}\\
&\Phi * u^{\sigma}_{\eta} \overset{*}{\rightharpoonup} \Phi * u^{\sigma}\quad {\rm{in}}\quad L^{\infty}(0,T; W^{2, q}(\R^d))
\quad  q \in (1, \infty ).\label{iv}
\end{align}
\end{lemma}

\begin{proof}
\eqref{i} and \eqref{iv} are direct consequences of  Lemma \ref{lemma_estimates_u_sigma_eta} $(i)$ and $(ii)$ respectively.
From \eqref{u_sigma_approximation} and Lemma \ref{lemma_estimates_u_sigma_eta}, we have
\begin{align}
\label{H_minus_1_estimate_u_sigma_eta}
\| \partial_t  u^{\sigma}_{\eta}  \|_{L^2(0,T; H^{-1}(\R^d))} \leq C.
\end{align}
Estimates from Lemma \ref{lemma_estimates_u_sigma_eta} $(iii)$ and \eqref{H_minus_1_estimate_u_sigma_eta} allow us to use Aubin-Lions Lemma to obtain \eqref{ii}. Since it is not possible to apply Aubin-Lions Lemma directly, one needs to consider $(u_\eta^\sigma)_\eta$ not on the whole space but on the sequence of growing $d$-dimensional balls and then use diagonal argument.

From Lemma \ref{lemma_estimates_u_sigma_eta} $(i)$ and $(iii)$, we deduce that
\begin{align*}
\| \nabla  ( u^{\sigma}_{\eta} + \eta)^m  \|_{L^2( 0,T; L^2(\R^d) ) } &\leq m \|  ( u^{\sigma}_{\eta} + \eta)^{m-1}  \|_{L^{\infty}(0,T; L^{\infty} (\R^d) ) }\| \nabla   u^{\sigma}_{\eta}  \|_{L^2(  0,T; L^2(\R^d) ) } \leq C.
\end{align*}
By this estimate together with \eqref{ii}, we derive that for any $\psi \in C_0^{\infty}((0,T) \times \R^d )$
\begin{align*}
\int_0^T \int_{\R^d} \nabla ( u^{\sigma}_{\eta} + \eta)^m \psi \,dx \,  dt &= -m \int_0^T \int_{\R^d} ( u^{\sigma}_{\eta} + \eta)^m \nabla \psi \, dx \, dt\\
&\rightarrow -m \int_0^T \int_{\R^d}  (u^{\sigma} )^m  \nabla \psi \, dx \, dt \\
&= m\int_0^T \int_{\R^d} \nabla (u^{\sigma})^m \psi \,dx \,  dt,\;\;\;\; {\rm{as}}\;\;\eta\rightarrow 0,
\end{align*}
which implies \eqref{iii}.
\end{proof}
Lemma \ref{leta} makes it possible to take limit $\eta \rightarrow 0$ in the weak formulation of \eqref{u_sigma_approximation} to construct
the weak solution $u^\sigma$ to the problem \eqref{generalized_equation_u_sigma}. Furthermore, by Moser's iteration, we can obtain that
there exists a positive constant $\hat{C}$ which depends on $d$, $m$ and $\| u_0 \|_{L^1(\R^d) \cap L^{\infty}(\R^d) }$ such that
\begin{align*}
\| u^{\sigma}  \|_{L^{\infty}(0,T; L^{\infty}(\R^d))} \leq \hat{C}.
\end{align*}
Notice that the constant $\hat C$ does not depend on $\sigma$.

\subsection{Regularity and uniqueness of the Solution $u^\sigma$ }\label{reg}
Our aim of this subsection is to prove regularity of the solution $u^\sigma$, which implies immediately the uniqueness by classical theory.
\begin{proposition}
Let the assumptions in Theorem \ref{theorem_u_sigma} hold
and $u^{\sigma}$ be a weak solution to the problem \eqref{generalized_equation_u_sigma}, then $u^{\sigma} \in W_q^{3,1}((0,T) \times  \R^d)$ for any $q\in(1,\infty)$.
\end{proposition}
\begin{proof}
So far we have proved that $u^{\sigma} \in L^{\infty} ( 0,T; W^{1, \infty} (\R^d)  ) \cap L^{\infty}(0,T; L^1(\R^d ))$.
Using the same arguments as in subsection \ref{global_existence_approximate}, we get
\begin{align*}
u^{\sigma} \in L^{\infty} ( 0,T; H^1 (\R^d)  ) \cap L^2(0,T; H^2(\R^d )).
\end{align*}
Applying  theory of the elliptic partial differential equation, we deduce that
\begin{align*}
\Phi * u^{\sigma} \in L^{\infty} (0,T; W^{2,q}(\R^d)) \ \text{for all}  \ q \in (1, \infty).
\end{align*}
Notice that
\begin{align}
\label{reg_u_sigma_2}
\partial_t u^{\sigma} - (\sigma + m (u^{\sigma})^{m-1})\Delta u^{\sigma} - m(m-1)(u^{\sigma})^{m-2} \nabla u^{\sigma} \cdot \nabla u^{\sigma} + \nabla \Phi * u^{\sigma} \cdot \nabla u^{\sigma} - u^{\sigma} u^{\sigma} =0.
\end{align}
Equation \eqref{reg_u_sigma_2} and the fact  that $u^{\sigma} \in L^{\infty}(0,T; W^{1,\infty}(\R^d))$ allow us to use
Theorem 9.2.2 of \citep{wu2006elliptic} to obtain that
\begin{align*}
u^{\sigma} \in W_q^{2,1} ((0,T) \times \R^d)\ \text{for any} \ q\in(1,\infty).
\end{align*}
Let $\partial_{x_j} u^\sigma=:v^\sigma$, then it holds
\begin{align*}
&\partial_t v^{\sigma} - (\sigma + m (u^{\sigma})^{m-1})\Delta v^{\sigma} - m(m-1)(u^{\sigma})^{m-2} \Delta u^{\sigma} \partial_{x_j} u^{\sigma}\\
& - m(m-1)(m-2)(u^{\sigma})^{m-3} |\nabla u^{\sigma} |^2\partial_{x_j} u^{\sigma} - 2m(m-1)(u^{\sigma})^{m-2} \nabla u^{\sigma} \cdot \nabla \partial_{x_j} u^{\sigma} \\
&+ \nabla \Phi * u^{\sigma} \cdot \nabla \partial_{x_j} u^{\sigma} + \nabla u^{\sigma} \cdot \nabla  \Phi *\partial_{x_j} u^{\sigma} - 2  u^{\sigma} \partial_{x_j} u^{\sigma} =0.
\end{align*}
By means of Theorem 9.2.2 of \citep{wu2006elliptic}, we deduce that
\begin{align*}
\partial_{x_j} u^{\sigma}=v^\sigma \in W_q^{2,1} ((0,T) \times \R^d)\ \text{for any} \ q\in(1,\infty).
\end{align*}
Therefore,
\begin{align*}
\nabla  u^{\sigma} \in W_q^{2,1} ((0,T) \times \R^d)\ \text{for any} \ q\in(1,\infty).
\end{align*}
It is obvious that
\begin{align*}
u^{\sigma} \in W_q^{3,1}((0,T) \times  \R^d) \ \text{for any} \  q\in(1,\infty).
\end{align*}
With this regularity, it is standard to obtain that the solution of \eqref{generalized_equation_u_sigma} is unique.

\end{proof}

\section{Well-posedness of non-local problem \eqref{generalized_equation_u_epsilon_sigma}}\label{section_u_epsilon_sigma}
The non-local terms in problem \eqref{generalized_equation_u_epsilon_sigma} appear not only in the aggregation term but also in diffusion, therefore there is no general procedure how to obtain the well-posedness for such non-local problems  without assuming  small initial data. Hence, in this section we use a perturbation method  to prove Theorem \ref{theorem_u_epsilon_sigma} for small $\varepsilon_k$ and $\varepsilon_p$ and so avoid the above restriction for the initial data. We study the dynamics of $u^{\varepsilon,\sigma}-u^\sigma$ and present the $H^s$ energy estimate for this term.
The letter $C$ in this section denotes a positive constant which depends on $\sigma$, $T$ and $s$.

\subsection{Local well-posedness of perturbation equations}
First of all, we prove the following lemma which will be used later.
\begin{lemma}
	\label{lemma_non_loc}
	Let $\varepsilon>0$, $q\in [1,\infty)$, $f\in L^q(\R^d)$ and $g\in W^{1, \infty}(\R^d)$. Then the following inequality holds
	\begin{align*}
	\| V^{\varepsilon}*(fg)-(V^{\varepsilon}*f)g \|_{L^q(\R^d)}^q \leq C \varepsilon^q \| \nabla g \|_{L^{\infty}(\R^d)}^q \| f \|_{L^q(\R^d)}^q.
	\end{align*}
\end{lemma}
\begin{proof}
Simple estimates by using H\"older's inequality show that
\begin{align*}
&\| V^{\varepsilon}*(fg)-(V^{\varepsilon}*f)g \|_{L^q(\R^d)}^q\\
=&  \int_{\R^d} \Big| \int_{\R^d} V^{\varepsilon}(x-y)^{\frac{q-1}{q}} V^{\varepsilon}(x-y)^{\frac{1}{q}}
\big( g(y) - g(x) \big) f(y) \, dy  \Big|^q \, dx  \\
\leq&  \int_{\R^d} \int_{\R^d} V^{\varepsilon}(x-y)  |f(y)|^q  | g(x) - g(y) |^q  \, dy \, dx  \\
=&  \int_{\R^d} \int_{\R^d} V^{\varepsilon}(z) |z|^q |f(y)|^q  \frac{| g(y+z) - g(y) |^q}{|z|^q}  \, dy \, dz \\
\leq & \|\nabla g \|_{L^{\infty}(\R^d)}^q  \|f\|_{L^q(\R^d)}^q \frac{1}{\varepsilon^d}  \int_{\R^d} V(\frac{z}{\varepsilon}) |z|^q \, dz  \\
\leq& C \varepsilon^q   \|\nabla g \|_{L^{\infty}(\R^d)}^q  \|f\|_{L^q(\R^d)}^q,
\end{align*}
where we have use the fact that $V$ is compactly supported.
\end{proof}

Under the assumptions of Theorem \ref{theorem_u_epsilon_sigma}, the solution $u^\sigma$ satisfies
$u^\sigma\in L^\infty(0,T;H^{s+2}(\R^d))$, $s> \frac{d}{2} +2$.
Define a set
\begin{align*}
Y:= \{ \mathbf{v}  \in \ &L^{\infty}(0,T; H^s(\R^d))\cap L^2(0,T; H^{s+1}(\R^d)) :\\
 &\| \mathbf{v} \|_{L^{\infty}(0,T; H^s(\R^d))}^2 +\| \nabla \mathbf{v}  \|_{L^2(0,T;H^s(\R^d))}^2 \leq \varepsilon_k +\varepsilon_p \ \text{and} \ \mathbf{v}(0,x)=0\}.
\end{align*}
For any $\mathbf{v} \in Y$,  we consider the following linear partial differential equation
\begin{align}
\label{lin_problem_non_loc}
\partial_t \mathbf{u} - \sigma \Delta \mathbf{u}  =& \nabla \cdot \Big( (\mathbf{u} +u^{\sigma})\nabla (p_{\lambda}-p)(V^{\varepsilon_p}*(\mathbf{v} +u^{\sigma})) + \mathbf{u} \nabla p ( V^{\varepsilon_p}*(\mathbf{v} +u^{\sigma}) ) \\
& + u^{\sigma} \nabla \int_0^1 p' ( V^{\varepsilon_p}*u^{\sigma} +s V^{\varepsilon_p}*\mathbf{v} ) V^{\varepsilon_p}*\mathbf{u} \, ds  +u^{\sigma}\big(\nabla p ( V^{\varepsilon_p}*u^{\sigma}) - \nabla p(u^{\sigma})\big) \Big) \notag \\
&- \nabla \cdot \Big(  \mathbf{u} \nabla \Phi * V^{\varepsilon_k}*\mathbf{v} + \mathbf{u} \nabla \Phi *V^{\varepsilon_k}*u^{\sigma} +u^{\sigma}\nabla \Phi*V^{\varepsilon_k}*(\mathbf{u} +u^{\sigma}) - u^{\sigma}\nabla \Phi * u^{\sigma}  \Big), \notag\\
\mathbf{u}(0,x)=0.\;\;& \notag
\end{align}
The linear problem \eqref{lin_problem_non_loc} possesses  a unique global solution $\mathbf{u} \in L^{\infty}(0,T; H^s(\R^d))\cap L^2(0,T; H^{s+1}(\R^d))$.
Let $\alpha$ be a multi-index of order $|\alpha|\leq s$. We apply $D^{\alpha}$ to \eqref{lin_problem_non_loc}, multiply the resulting equation by $D^{\alpha} \mathbf{u} $ and integrate over $\R^d$ to obtain that
\begin{align*}
\frac{1}{2} \frac{d}{dt} \int_{\R^d} |D^{\alpha} \mathbf{u} |^2 dx + \sigma \int_{\R^d} |\nabla D^{\alpha} \mathbf{u} |^2 dx =& - \int_{\R^d} D^{\alpha} \big( (\mathbf{u} +u^{\sigma})\nabla (p_{\lambda}-p)(V^{\varepsilon_p}*(\mathbf{v} +u^{\sigma})) \big) \cdot \nabla D^{\alpha} \mathbf{u} \, dx \\
& - \int_{\R^d} D^{\alpha} \big( \mathbf{u} \nabla p ( V^{\varepsilon_p}*(\mathbf{v} +u^{\sigma}) ) \big) \cdot \nabla D^{\alpha} \mathbf{u} \, dx \\
& - \int_{\R^d} D^{\alpha} \big( u^{\sigma} \nabla \int_0^1 p' ( V^{\varepsilon_p}*u^{\sigma} +z V^{\varepsilon_p}*\mathbf{v} ) V^{\varepsilon_p}*\mathbf{u} \, dz  \big) \cdot \nabla D^{\alpha} \mathbf{u} \, dx \\
& - \int_{\R^d} D^{\alpha} \big( u^{\sigma} (\nabla p ( V^{\varepsilon_p}*u^{\sigma}) - \nabla p(u^{\sigma})) \big) \cdot \nabla D^{\alpha} \mathbf{u} \, dx \\
&+ \int_{\R^d} D^{\alpha} \big(   \mathbf{u} \nabla \Phi * V^{\varepsilon_k}*\mathbf{v}   \big) \cdot \nabla D^{\alpha} \mathbf{u} \, dx \\
&+ \int_{\R^d} D^{\alpha} \big(  \mathbf{u} \nabla \Phi *V^{\varepsilon_k}*u^{\sigma}  \big) \cdot \nabla D^{\alpha} \mathbf{u} \, dx \\
&+ \int_{\R^d} D^{\alpha} \big(  u^{\sigma}\nabla \Phi*V^{\varepsilon_k}*(\mathbf{u} +u^{\sigma}) - u^{\sigma}\nabla \Phi * u^{\sigma} )  \big) \cdot \nabla D^{\alpha} \mathbf{u} \, dx \\
=:& I_1 + I_2 + I_3 + I_4 + I_5 + I_6 +I_7.
\end{align*}
Since $\|\mathbf{v} \|_{L^{\infty}(0,T; L^{\infty}(\R^d))}$ and $\|u^{\sigma} \|_{L^{\infty}(0,T; L^{\infty}(\R^d))}$ are bounded, we can take $\lambda$ small enough such that
\begin{align*}
\|\mathbf{v} \|_{L^{\infty}(0,T; L^{\infty}(\R^d))} \leq \frac{1}{2\lambda}, \qquad
\|u^{\sigma} \|_{L^{\infty}(0,T; L^{\infty}(\R^d))} \leq \frac{1}{2\lambda}.
\end{align*}
It follows that $I_1=0$.
For $I_2$,
\begin{align*}
I_2\leq& \frac{\sigma}{12} \| \nabla D^{\alpha} \mathbf{u} \|_{L^2(\R^d)}^2 + C \|  \mathbf{u} \|_{L^{\infty}(\R^d)}^2 \| D^{\alpha}  \nabla p(V^{\varepsilon_p}*(\mathbf{v} + u^{\sigma}))  \|_{L^2(\R^d)}^2\\
& + C \|   D^{\alpha} \mathbf{u}  \|_{L^2(\R^d)}^2 \| \nabla p(V^{\varepsilon_p}*(\mathbf{v} + u^{\sigma}))   \|_{L^{\infty}(\R^d)}^2\\
\leq& \frac{\sigma}{12} \| \nabla D^{\alpha} \mathbf{u} \|_{L^2(\R^d)}^2  + C \| \mathbf{u} \|_{H^s(\R^d)}^2 \| D^{\alpha}  \nabla p(V^{\varepsilon_p}*(\mathbf{v} + u^{\sigma}))   \|_{L^2(\R^d)}^2 \\
& + C  \| \mathbf{u} \|_{H^s(\R^d)}^2 \|  p(V^{\varepsilon_p}*(\mathbf{v} + u^{\sigma}))   \|_{H^s(\R^d)}^2\\
\leq& \frac{\sigma}{12}  \| \nabla D^{\alpha} \mathbf{u} \|_{L^2(\R^d)}^2
 + C \| \mathbf{u} \|_{H^s(\R^d)}^2 \| \nabla \mathbf{v} \|_{H^s(\R^d)}^2
+ C \| \mathbf{u} \|_{H^s(\R^d)}^2.
\end{align*}
For $I_3$,
\begin{align*}
I_3
=& - \int_{\R^d} \int_0^1 u^{\sigma} \nabla D^{\alpha}\left( p' ( V^{\varepsilon_p}*u^{\sigma} +z V^{\varepsilon_p}*\mathbf{v} ) V^{\varepsilon_p}*\mathbf{u}  \right) \cdot \nabla D^{\alpha} \mathbf{u}  \, dz \, dx \\
&- \int_{\R^d} \int_0^1 \big( D^{\alpha} \big( u^{\sigma} \nabla (  p' ( V^{\varepsilon_p}*u^{\sigma} +z V^{\varepsilon_p}*\mathbf{v} ) V^{\varepsilon_p}*\mathbf{u} ) \big)
- u^{\sigma} \nabla D^{\alpha} \big(  p' ( V^{\varepsilon_p}*u^{\sigma} +z V^{\varepsilon_p}*\mathbf{v} ) V^{\varepsilon_p}*\mathbf{u}  \big)  \big)  \cdot \nabla D^{\alpha} \mathbf{u} \, dz \, dx \\
=:&A+B.
\end{align*}
First we deal with $B$,
\begin{align*}
B \leq& \frac{\sigma}{36} \| \nabla D^{\alpha} \mathbf{u}  \|_{L^2(\R^d)}^2\\
&+ C \sup_{0<z<1}  \| D^{\alpha} \big( u^{\sigma} \nabla(  p' ( V^{\varepsilon_p}*u^{\sigma} +z V^{\varepsilon_p}*\mathbf{v}  )  V^{\varepsilon_p}*\mathbf{u} )  \big)   - u^{\sigma} D^{\alpha} \nabla \big( p'(V^{\varepsilon_p}*u^{\sigma} + zV^{\varepsilon_p}*\mathbf{v} ) \big)  \|_{L^2(\R^d)}^2\\
\leq& \frac{\sigma}{36} \|  \nabla D^{\alpha} \mathbf{u}   \|_{L^2(\R^d)}^2 + C \| \nabla  u^{\sigma} \|_{L^{\infty}(\R^d)}^2 \sup_{0<z<1} \| D^{\alpha-1} \nabla \big(  p' ( V^{\varepsilon_p}*u^{\sigma} +z V^{\varepsilon_p}*\mathbf{v}  )  V^{\varepsilon_p}*\mathbf{u} \big)  \|_{L^2(\R^d)}^2 \\
&+ C  \|   D^{\alpha} u^\sigma   \|_{L^2(\R^d)}^2  \sup_{0<z<1} \| \nabla \big(  p' ( V^{\varepsilon_p}*u^{\sigma} +z V^{\varepsilon_p}*\mathbf{v}  )  V^{\varepsilon_p}*\mathbf{u} \big)  \|_{L^{\infty}(\R^d)}^2 \\
\leq& \frac{\sigma}{36} \| \nabla D^{\alpha} \mathbf{u}  \|_{L^2(\R^d)}^2 + C \| \mathbf{u} \|_{H^s(\R^d)}^2.
\end{align*}
The term $A$ can be rewritten as
\begin{align*}
A =& - \int_{\R^d} \int_0^1 u^{\sigma}  p' ( V^{\varepsilon_p}*u^{\sigma} +z V^{\varepsilon_p}*\mathbf{v} ) V^{\varepsilon_p}*(\nabla D^{\alpha} \mathbf{u})   \cdot \nabla D^{\alpha} \mathbf{u}  \, dz \, dx \\
&- \int_{\R^d} \int_0^1 u^{\sigma} \big( \nabla D^{\alpha} \big( p' ( V^{\varepsilon_p}*u^{\sigma} +z V^{\varepsilon_p}*\mathbf{v} ) V^{\varepsilon_p}*\mathbf{u}   \big)
- p'( V^{\varepsilon_p}*u^{\sigma} +z V^{\varepsilon_p}*\mathbf{v} ) \nabla D^{\alpha}  V^{\varepsilon_p}* \mathbf{u}  \big) \cdot \nabla D^{\alpha} \mathbf{u}  \, dz \, dx \\
=&: G+D.
\end{align*}
It is easy to see that
\begin{align*}
D \leq&  \frac{\sigma}{36} \| \nabla D^{\alpha} \mathbf{u}  \|_{L^2(\R^d)}^2 \\
&+ C \sup_{0<z<1} \Big\| u^{\sigma} \Big( \nabla D^{\alpha} \big(  p' ( V^{\varepsilon_p}*u^{\sigma} +z V^{\varepsilon_p}*\mathbf{v} ) V^{\varepsilon_p}*\mathbf{u} \big) - p' ( V^{\varepsilon_p}*u^{\sigma} +z V^{\varepsilon_p}*\mathbf{v} ) \nabla D^{\alpha} V^{\varepsilon_p}*\mathbf{u}  \Big) \Big\|_{L^2(\R^d)}^2 \\
\leq&  \frac{\sigma}{36} \| \nabla D^{\alpha} \mathbf{u}  \|_{L^2(\R^d)}^2 + C \| \nabla \big( p' ( V^{\varepsilon_p}*u^{\sigma} +V^{\varepsilon_p}*\mathbf{v} )  \big) \|_{L^{\infty}(\R^d)}^2 \| D^{\alpha} V^{\varepsilon_p}*\mathbf{u}  \|_{L^2(\R^d)}^2 \\
&+ C \| \nabla D^{\alpha} \big( p' ( V^{\varepsilon_p}*u^{\sigma} +V^{\varepsilon_p}*\mathbf{v} )  \big) \|_{L^2(\R^d)}^2 \| V^{\varepsilon_p}*\mathbf{u}  \|_{L^{\infty}(\R^d)}^2\\
\leq&  \frac{\sigma}{36} \| \nabla D^{\alpha} \mathbf{u}  \|_{L^2(\R^d)}^2 + C \| \mathbf{u}  \|_{H^s(\R^d)}^2
+C \| \mathbf{u} \|_{H^s(\R^d)}^2 \| \nabla \mathbf{v} \|_{H^s(\R^d)}^2.
\end{align*}
We deal with the term $G$ as follows
\begin{align*}
G=& - \int_{\R^d} \int_0^1 W^{\varepsilon_p} * \big( u^{\sigma} p'(V^{\varepsilon_p}*u^{\sigma} +z V^{\varepsilon_p} * \mathbf{v})\nabla D^{\alpha}\mathbf{u} \big) \cdot W^{\varepsilon_p} * \nabla D^{\alpha} \mathbf{u} \, dz \, dx \\
=&  - \int_{\R^d} \int_0^1 u^{\sigma} p'(V^{\varepsilon_p}*u^{\sigma} +z V^{\varepsilon_p} * \mathbf{v}) | W^{\varepsilon_p} * \nabla D^{\alpha} \mathbf{u}|^2 \, dz \, dx \\
&- \int_{\R^d} \int_0^1 \Big( W^{\varepsilon_p}* \big( u^{\sigma} p'(V^{\varepsilon_p}*u^{\sigma} +z V^{\varepsilon_p} * \mathbf{v}) \nabla D^{\alpha} \mathbf{u}  \big) - u^{\sigma}  W^{\varepsilon_p}* \big( p'(V^{\varepsilon_p}*u^{\sigma} +z V^{\varepsilon_p} * \mathbf{v}) \nabla D^{\alpha} \mathbf{u} \big) \Big) \cdot \nabla D^{\alpha} \mathbf{u}  \, dz \, dx \\
& - \int_{\R^d} \int_0^1 u^{\sigma} \Big( W^{\varepsilon_p}* \big(  p'(V^{\varepsilon_p}*u^{\sigma} +z V^{\varepsilon_p} * \mathbf{v}) \nabla D^{\alpha} \mathbf{u}  \big) -  p'(V^{\varepsilon_p}*u^{\sigma} +z V^{\varepsilon_p} * \mathbf{v}) W^{\varepsilon_p}* \nabla D^{\alpha} \mathbf{u} \Big) \cdot  \nabla D^{\alpha} \mathbf{u} \, dz \, dx \\
\leq& 0+E+F.
\end{align*}
With the help of Lemma \ref{lemma_non_loc}, we have
\begin{align*}
E &\leq \sup_{0<z<1} \big\|  W^{\varepsilon_p}* \big( u^{\sigma} p'(V^{\varepsilon_p}*u^{\sigma} +z V^{\varepsilon_p} * \mathbf{v}) \nabla D^{\alpha} \mathbf{u}  \big) - u^{\sigma}  W^{\varepsilon_p}* \big( p'(V^{\varepsilon_p}*u^{\sigma} +z V^{\varepsilon_p} * \mathbf{v}) \nabla D^{\alpha} \mathbf{u} \big)   \big\|_{L^2(\R^d)} \big\| \nabla D^{\alpha} \mathbf{u}  \big\|_{L^2(\R^d)} \\
&\leq C \varepsilon_p \| \nabla u^{\sigma}  \|_{L^{\infty}(\R^d)} \sup_{0<z<1} \| p'(V^{\varepsilon_p}*u^{\sigma} +z V^{\varepsilon_p} * \mathbf{v}) \nabla D^{\alpha} \mathbf{u}   \|_{L^2(\R^d)} \| \nabla D^{\alpha} \mathbf{u}  \|_{L^2(\R^d)}\\
&\leq C \varepsilon_p  \sup_{0<z<1}  \| p'(V^{\varepsilon_p}*u^{\sigma} +z V^{\varepsilon_p} * \mathbf{v})    \|_{L^{\infty}(\R^d)}  \| \nabla D^{\alpha} \mathbf{u}  \|_{L^2(\R^d)}^2\\
&\leq  C \varepsilon_p  \| \nabla D^{\alpha} \mathbf{u}  \|_{L^2(\R^d)}^2.
\end{align*}
Similarly,
\begin{align*}
F &\leq  \sup_{0<z<1} \big\|  W^{\varepsilon_p}* \big(  p'(V^{\varepsilon_p}*u^{\sigma} +z V^{\varepsilon_p} * \mathbf{v}) \nabla D^{\alpha} \mathbf{u}  \big)
  -     p'(V^{\varepsilon_p}*u^{\sigma} +z V^{\varepsilon_p} * \mathbf{v}) W^{\varepsilon_p}* \nabla D^{\alpha} \mathbf{u}    \big\|_{L^2(\R^d)} \big\|  u^{\sigma}  \big\|_{L^{\infty}(\R^d)}  \big\| \nabla D^{\alpha} \mathbf{u}  \big\|_{L^2(\R^d)} \\
&\leq C \varepsilon_p  \sup_{0<z<1}  \|\nabla p'(V^{\varepsilon_p}*u^{\sigma} +z V^{\varepsilon_p} * \mathbf{v})    \|_{L^{\infty}(\R^d)}  \| \nabla D^{\alpha} \mathbf{u}  \|_{L^2(\R^d)}^2\\
&\leq  C \varepsilon_p  \| \nabla D^{\alpha} \mathbf{u}  \|_{L^2(\R^d)}^2 .
\end{align*}
Combining these estimates, we obtain that
\begin{align*}
I_3 \leq \frac{\sigma}{18} \| \nabla D^{\alpha} \mathbf{u}  \|_{L^2(\R^d)}^2 +  C \| \mathbf{u}   \|_{H^s(\R^d)}^2
+C \| \mathbf{u}  \|_{H^s(\R^d)}^2 \| \nabla \mathbf{v}  \|_{H^s(\R^d)}^2
+ C \varepsilon_p \|\nabla D^\alpha \mathbf{u} \|_{L^2(\R^d)}^2.
\end{align*}
We take $\varepsilon_p$ small enough such that $C \varepsilon_p \leq \frac{\sigma}{36}$ to get
\begin{align*}
I_3 \leq \frac{\sigma}{12} \| \nabla D^{\alpha} \mathbf{u}  \|_{L^2(\R^d)}^2
+  C \| \mathbf{u}   \|_{H^s(\R^d)}^2
+C \| \mathbf{u}  \|_{H^s(\R^d)}^2 \| \nabla \mathbf{v}  \|_{H^s(\R^d)}^2 .
\end{align*}
Next, we handle $I_4$ using Lemma \ref{lemma_non_loc}.
\begin{align*}
I_4
\leq& \frac{\sigma}{12} \|  \nabla D^{\alpha} \mathbf{u}  \|_{L^2(\R^d)}^2 + C \|D^{\alpha} u^{\sigma} \|_{L^2(\R^d)}^2 \| \nabla p ( V^{\varepsilon_p}*u^{\sigma}) - \nabla p(u^{\sigma})  \|_{L^{\infty}(\R^d)}^2 \\
&+ C \| u^{\sigma} \|_{L^{\infty}(\R^d)}^2 \|   D^{\alpha} \nabla \big( p ( V^{\varepsilon_p}*u^{\sigma}) -  p(u^{\sigma})  \big)  \|_{L^2(\R^d)}^2\\
\leq&   \frac{\sigma}{12} \|  \nabla D^{\alpha} \mathbf{u}  \|_{L^2(\R^d)}^2  + C \Big\| \nabla \int_0^1 p'(zV^{\varepsilon_p}*u^{\sigma} +(1-z)u^{\sigma})(V^{\varepsilon_p} *u^{\sigma} - u^{\sigma}  ) \, dz  \Big\|_{H^s (\R^d)}^2 \\
\leq&  \frac{\sigma}{12} \|  \nabla D^{\alpha} \mathbf{u}  \|_{L^2(\R^d)}^2 +C \Big\| \int_0^1 p'(z V^{\varepsilon_p}*u^{\sigma} +(1-z)u^{\sigma})\nabla(V^{\varepsilon_p} *u^{\sigma} - u^{\sigma}  ) \, dz  \Big\|_{H^s (\R^d)}^2 \\
&+ C \Big\| \int_0^1 p''(z V^{\varepsilon_p}*u^{\sigma} +(1-z)u^{\sigma}) \nabla \big(zV^{\varepsilon_p}*u^{\sigma} +(1-z)u^{\sigma} \big) (V^{\varepsilon_p} *u^{\sigma} - u^{\sigma}  ) \, dz  \Big\|_{H^s (\R^d)}^2 \\
\leq& \frac{\sigma}{12} \|  \nabla D^{\alpha} \mathbf{u}  \|_{L^2(\R^d)}^2 + C \varepsilon_p^2.
\end{align*}
For $I_5$ we get the following estimate
\begin{align*}
I_5
&\leq \frac{\sigma}{12} \|  \nabla D^{\alpha} \mathbf{u}  \|_{L^2(\R^d)}^2 + C \| D^{\alpha} \mathbf{u}  \|_{L^2(\R^d)}^2 \| \nabla \Phi * V^{\varepsilon_k}*\mathbf{v}  \|_{L^{\infty}(\R^d)}^2 + \| \mathbf{u}   \|_{L^{\infty}(\R^d)}^2  \| D^{\alpha} \big( \nabla \Phi * V^{\varepsilon_k}*\mathbf{v}   \big)  \|_{L^2(\R^d)}^2  \\
&\leq \frac{\sigma}{12} \|  \nabla D^{\alpha} \mathbf{u}  \|_{L^2(\R^d)}^2 + C \| \mathbf{u} \|_{H^s(\R^d)}^2.
\end{align*}
Similarly, we obtain that
\begin{align*}
I_6 \leq \frac{\sigma}{12} \|  \nabla D^{\alpha} \mathbf{u}  \|_{L^2(\R^d)}^2 + C \| \mathbf{u} \|_{H^s(\R^d)}^2.
\end{align*}
We estimate $I_7$ as follows
\begin{align*}
I_7
\leq& \frac{\sigma}{12} \|  \nabla D^{\alpha} \mathbf{u}  \|_{L^2(\R^d)}^2 +  C  \|  D^{\alpha} \big(  u^{\sigma} ( \nabla \Phi*V^{\varepsilon_k}*(\mathbf{u} +u^{\sigma}) - \nabla \Phi * V^{\varepsilon_k} *  u^{\sigma} )  \big)     \|_{L^2(\R^d)}^2 \\
& + C\|  D^{\alpha} \big(  u^{\sigma} ( \nabla \Phi * V^{\varepsilon_k} *  u^{\sigma}  - \nabla \Phi * u^{\sigma} )  \big)     \|_{L^2(\R^d)}^2 \\
\leq& \frac{\sigma}{12} \|  \nabla D^{\alpha} \mathbf{u}  \|_{L^2(\R^d)}^2 + C \| D^{\alpha} u^{\sigma}  \|_{L^2(\R^d)}^2 \| \nabla \Phi*V^{\varepsilon_k}*\mathbf{u}  \|_{L^{\infty}(\R^d)}^2\\
&  + C \| u^{\sigma}  \|_{L^{\infty}(\R^d)}^2 \| D^{\alpha} ( \nabla \Phi*V^{\varepsilon_k}*\mathbf{u}  ) \|_{L^2(\R^d)}^2  + C \| D^{\alpha} u^{\sigma}  \|_{L^2(\R^d)}^2 \| V^{\varepsilon_k}*\nabla \Phi*u^{\sigma} - \nabla \Phi  *  u^{\sigma}   \|_{L^{\infty}(\R^d)}^2 \\
&+ C \| u^{\sigma}  \|_{L^{\infty}(\R^d)}^2 \|  D^{\alpha} \big(    V^{\varepsilon_k} * \nabla \Phi*  u^{\sigma}  - \nabla \Phi * u^{\sigma}   \big)     \|_{L^2(\R^d)}^2  \\
\leq&  \frac{\sigma}{12} \|  \nabla D^{\alpha} \mathbf{u}  \|_{L^2(\R^d)}^2 + C \| \mathbf{u} \|_{H^s(\R^d)}^2 + C \varepsilon_k^2.
\end{align*}
By estimates for $I_1$-$I_7$, we deduce that
\begin{align*}
\frac{1}{2} \frac{d}{dt} \| \mathbf{u} \|_{H^s(\R^d)}^2 + \frac{\sigma}{2} \| \nabla \mathbf{u} \|_{H^s(\R^d)}^2 \leq \| \mathbf{u} \|_{H^s(\R^d)}^2 + C \| \mathbf{u} \|_{H^s(\R^d)}^2 \| \nabla \mathbf{v} \|_{H^s(\R^d)}^2 + C (\varepsilon_k + \varepsilon_p)^2.
\end{align*}
Using Gr\"{o}nwall's inequality, we obtain that
\begin{align*}
\| \mathbf{u}(t) \|_{H^s(\R^d)}^2 \leq C  (\varepsilon_k + \varepsilon_p)^2 e^{C \int_0^T ( 1+ \| \nabla \mathbf{v} \|_{H^{s}(\R^d)}^2)  \, d\tilde{s} } \leq C (\varepsilon_k + \varepsilon_p)^2,\;\; \forall t \in [0,T].
\end{align*}
This implies that $u \in Y$. Hence, we can define an operator
\begin{align*}
\mathcal{T}: Y &\rightarrow Y\\
 \mathbf{v} &\mapsto  \mathbf{u},
\end{align*}
where $ \mathbf{u}$ solves \eqref{lin_problem_non_loc}. In order to apply Banach fixed-point theorem, it remains to prove that $\mathcal{T}$ is a contraction.

Let $ \mathbf{u}_1,  \mathbf{u}_2 \in Y$ be solutions of \eqref{lin_problem_non_loc} for $\mathbf{v}_1,\mathbf{v}_2\in Y$ respectively. Then $\mathbf{u}_1 -\mathbf{u}_2$ solves the following equation
\begin{align}
\label{non_loc_difference}
&\partial_t (\mathbf{u}_1 -\mathbf{u}_2) - \sigma \Delta (\mathbf{u}_1 -\mathbf{u}_2)\\
  &= \nabla \cdot \Big( (\mathbf{u}_1 +u^{\sigma})\nabla (p_{\lambda}-p)(V^{\varepsilon_p}*(\mathbf{v}_1 +u^{\sigma}))  -  (\mathbf{u}_2 +u^{\sigma})\nabla (p_{\lambda}-p)(V^{\varepsilon_p}*(\mathbf{v}_2 +u^{\sigma})) \notag \\
 & \;\;\; + \mathbf{u}_1 \nabla p ( V^{\varepsilon_p}*(\mathbf{v}_1 +u^{\sigma}) ) -  \mathbf{u}_2 \nabla p ( V^{\varepsilon_p}*(\mathbf{v}_2 +u^{\sigma}) ) \notag  \\
&\;\;\; + u^{\sigma} \nabla \int_0^1 p' ( V^{\varepsilon_p}*u^{\sigma} +z V^{\varepsilon_p}*\mathbf{v}_1 ) V^{\varepsilon_p}*\mathbf{u}_1 \, dz -  u^{\sigma} \nabla \int_0^1 p' ( V^{\varepsilon_p}*u^{\sigma} +z V^{\varepsilon_p}*\mathbf{v}_2 ) V^{\varepsilon_p}*\mathbf{u}_2 \, dz  \Big)  \notag \\
&\;\;\;- \nabla \cdot \Big(  \mathbf{u}_1 \nabla \Phi * V^{\varepsilon_k}*\mathbf{v}_1 - \mathbf{u}_2 \nabla \Phi * V^{\varepsilon_k}*\mathbf{u}_2  + ( \mathbf{u}_1 -\mathbf{u}_2 ) \nabla \Phi *V^{\varepsilon_k}*u^{\sigma} +u^{\sigma}\nabla \Phi*V^{\varepsilon_k}*( \mathbf{u}_1 -\mathbf{u}_2 )  \Big), \notag\\
&(\mathbf{u}_1-\mathbf{u}_2)(0,x)=0. \notag
\end{align}
\raggedbottom
Multiplying \eqref{non_loc_difference} by $\mathbf{u}_1 -\mathbf{u}_2$ and integrating it over $\R^d$, we deduce that
\begin{align*}
&\frac{1}{2} \frac{d}{dt} \int_{\R^d} |\mathbf{u}_1 -\mathbf{u}_2|^2 \, dx  + \sigma  \int_{\R^d} | \nabla  (\mathbf{u}_1 -\mathbf{u}_2) |^2 \, dx  \\
  =&\int_{\R^d} -  (\mathbf{u}_1 +u^{\sigma})\nabla (p_{\lambda}-p)(V^{\varepsilon_p}*(\mathbf{v}_1 +u^{\sigma})) \cdot \nabla (\mathbf{u}_1 -\mathbf{u}_2)   + (\mathbf{u}_2 +u^{\sigma})\nabla (p_{\lambda}-p)(V^{\varepsilon_p}*(\mathbf{v}_2 +u^{\sigma}))  \cdot \nabla (\mathbf{u}_1 -\mathbf{u}_2)  \, dx  \notag \\
 &- \int_{\R^d}   (\mathbf{u}_1 - \mathbf{u}_2 ) \nabla p ( V^{\varepsilon_p}*(\mathbf{v}_1 +u^{\sigma}) )   \cdot \nabla (\mathbf{u}_1 -\mathbf{u}_2)  \, dx    \notag  \\
   & - \int_{\R^d} \mathbf{u}_2 \big(  \nabla p ( V^{\varepsilon_p}*(\mathbf{v}_1 +u^{\sigma}) ) -   \nabla p ( V^{\varepsilon_p}*(\mathbf{v}_2 +u^{\sigma}) ) \big) \cdot \nabla (\mathbf{u}_1 -\mathbf{u}_2)     \, dx    \notag \\
   & - \int_{\R^d}  u^{\sigma} \nabla \int_0^1 p' ( V^{\varepsilon_p}*u^{\sigma} +z V^{\varepsilon_p}*\mathbf{v}_1 ) V^{\varepsilon_p}*( \mathbf{u}_1 - \mathbf{u}_2 ) \, dz   \cdot \nabla (\mathbf{u}_1 -\mathbf{u}_2)     \, dx    \notag \\
   &  - \int_{\R^d}   u^{\sigma} \nabla \int_0^1 \big(  p' ( V^{\varepsilon_p}*u^{\sigma} +z V^{\varepsilon_p}*\mathbf{v}_1 ) - p' ( V^{\varepsilon_p}*u^{\sigma} +z V^{\varepsilon_p}*\mathbf{v}_2 )  \big) V^{\varepsilon_p}* \mathbf{u}_2  \, dz   \cdot \nabla (\mathbf{u}_1 -\mathbf{u}_2)     \, dx  \notag \\
& + \int_{\R^d}   ( \mathbf{u}_1 - \mathbf{u}_2 ) \nabla \Phi * V^{\varepsilon_k}*\mathbf{v}_1 \nabla (\mathbf{u}_1 -\mathbf{u}_2)  \, dx
 + \int_{\R^d} \mathbf{u}_2 \nabla \Phi * V^{\varepsilon_k}*( \mathbf{v}_1 - \mathbf{v}_2 )  \cdot \nabla (\mathbf{u}_1 -\mathbf{u}_2)  \, dx \notag \\
& + \int_{\R^d}  ( \mathbf{u}_1 -\mathbf{u}_2 ) \nabla \Phi *V^{\varepsilon_k}*u^{\sigma} \cdot \nabla (\mathbf{u}_1 -\mathbf{u}_2)      \, dx
+ \int_{\R^d} u^{\sigma}\nabla \Phi*V^{\varepsilon_k}*( \mathbf{u}_1 -\mathbf{u}_2 ) \cdot \nabla (\mathbf{u}_1 -\mathbf{u}_2)       \, dx  \notag\\
=&: J_1 + J_2 + J_3 + J_4 + J_5 + J_6 + J_7 + J_8 + J_9.
\end{align*}

First, we can take $\lambda$ small enough such that
\begin{align*}
J_1 =0.
\end{align*}
Next, we deal with $J_2$.
\begin{align*}
J_2\leq \frac{\sigma}{16} \int_{\R^d} | \nabla (\mathbf{u}_1 - \mathbf{u}_2 ) |^2 \, dx  + \hat{C}  \int_{\R^d} | \mathbf{u}_1 - \mathbf{u}_2 |^2 \, dx,
\end{align*}
where $\hat{C}$ in this subsection denotes a positive constant which depends on $\sigma$, $s$ and $ \varepsilon_p$.
Similarly we obtain that $J_6$, $J_7$, $J_8$ and $J_9$ are bounded by
\begin{align*}
J_6+J_7+J_8+J_9\leq\frac{\sigma}{16} \int_{\R^d} | \nabla (\mathbf{u}_1 - \mathbf{u}_2 ) |^2 \, dx  + \hat{C}  \int_{\R^d} | \mathbf{u}_1 - \mathbf{u}_2 |^2 \, dx.
\end{align*}
For $J_3$,
\begin{align*}
J_3
\leq& \frac{\sigma}{16} \int_{\R^d} | \nabla (\mathbf{u}_1 - \mathbf{u}_2 ) |^2 \, dx + \hat{C} \| \mathbf{u}_2 \|_{L^{\infty}(\R^d)}^2  \|   \nabla p ( V^{\varepsilon_p}*(\mathbf{v}_2 +u^{\sigma}) ) -   \nabla p ( V^{\varepsilon_p}*(\mathbf{v}_1 +u^{\sigma}) )  \|_{L^2(\R^d)}^2\\
\leq& \frac{\sigma}{16} \int_{\R^d} | \nabla (\mathbf{u}_1 - \mathbf{u}_2 ) |^2 \, dx \\
&+ \hat{C}  \int_{\R^d}  \Big|  \int_0^1   p'' \big(z V^{\varepsilon_p}*(\mathbf{v}_2
+ u^{\sigma}) +(1-z) V^{\varepsilon_p}*(\mathbf{v}_1 + u^{\sigma})  \big)   \\
& \ \ \ \ \ \ \ \ \ \cdot  \nabla V^{\varepsilon_p}* \big( z \mathbf{v}_2 + (1-z)\mathbf{v}_1 + u^{\sigma}  \big) V^{\varepsilon_p}*(\mathbf{v}_2 - \mathbf{v}_1)    \Big|^2      \, dz   \, dx  \\
&+ \hat{C}  \int_{\R^d}   \int_0^1  \Big|   p' \big(z V^{\varepsilon_p}*(\mathbf{v}_2 + u^{\sigma}) +(1-z) V^{\varepsilon_p}*(\mathbf{v}_1 + u^{\sigma})  \big)   \nabla V^{\varepsilon_p}* ( \mathbf{v}_2 - \mathbf{v}_1  ) \Big|^2      \, dz    \, dx  \\
=&: \frac{\sigma}{16} \int_{\R^d} | \nabla (\mathbf{u}_1 - \mathbf{u}_2 ) |^2 \, dx  +J_{31} + J_{32}.
\end{align*}
By Young's convolution inequality, we derive that
\begin{align*}
J_{31} &\leq \hat{C} \sup_{0< z < 1} \| \nabla V^{\varepsilon_p}* \big( z \mathbf{v}_2 + (1-z)\mathbf{v}_1 + u^{\sigma}  \big) V^{\varepsilon_p}*(\mathbf{v}_2 - \mathbf{v}_1)     \|_{L^2(\R^d)}^2 \\
&\leq \hat{C} \sup_{0< z < 1} \| \nabla V^{\varepsilon_p}* \big( z \mathbf{v}_2 + (1-z)\mathbf{v}_1 + u^{\sigma}  \big)  \|_{L^{\infty}(\R^d)}^2 \| \mathbf{v}_2 - \mathbf{v}_1  \|_{L^2(\R^d)}^2 \\
&\leq \hat{C} \| \mathbf{v}_2 - \mathbf{v}_1  \|_{L^2(\R^d)}^2.
\end{align*}
 For $J_{32}$ we deduce that
 \begin{align*}
 J_{32} & \leq \hat{C}  \| \nabla  V^{\varepsilon_p} * ( \mathbf{v}_1 - \mathbf{v}_2 ) \|_{L^2(\R^d)}^2 \leq \hat{C} \| \nabla  V^{\varepsilon_p}  \|_{L^1(\R^d)}^2  \| \mathbf{v}_2 - \mathbf{v}_1  \|_{L^2(\R^d)}^2 \leq \hat{C} \| \mathbf{v}_2 - \mathbf{v}_1  \|_{L^2(\R^d)}^2.
\end{align*}
This implies that
\begin{align*}
J_3&\leq \frac{\sigma}{16} \int_{\R^d} | \nabla (\mathbf{u}_1 - \mathbf{u}_2 ) |^2 \, dx  + \hat{C} \| \mathbf{v}_2 - \mathbf{v}_1  \|_{L^2(\R^d)}^2.
\end{align*}
We deal with  $J_4$ as follows
\begin{align*}
J_4
\leq& \frac{\sigma}{16} \int_{\R^d} | \nabla (\mathbf{u}_1 - \mathbf{u}_2 ) |^2 \, dx \\
& + \hat{C} \Big\|   \int_0^1   p' ( V^{\varepsilon_p}*u^{\sigma} +z V^{\varepsilon_p}*\mathbf{v}_1 ) \nabla V^{\varepsilon_p}*( \mathbf{u}_1 - \mathbf{u}_2 )  \, dz\Big \|_{L^2(\R^d)}^2 \\
& + \hat{C} \Big\|   \int_0^1   p'' ( V^{\varepsilon_p}*u^{\sigma} +z V^{\varepsilon_p}*\mathbf{v}_1 ) V^{\varepsilon_p}* \nabla ( u^{\sigma} + z \mathbf{v}_1  )  V^{\varepsilon_p}*( \mathbf{u}_1 - \mathbf{u}_2 )  \, dz \Big\|_{L^2(\R^d)}^2 \\
\leq&  \frac{\sigma}{16} \int_{\R^d} | \nabla (\mathbf{u}_1 - \mathbf{u}_2 ) |^2 \, dx + \hat{C} \| \nabla  V^{\varepsilon_p} * ( \mathbf{u}_1 - \mathbf{u}_2 ) \|_{L^2(\R^d)}^2 + \hat{C} \|  V^{\varepsilon_p} * ( \mathbf{u}_1 - \mathbf{u}_2 ) \|_{L^2(\R^d)}^2 \\
\leq&  \frac{\sigma}{16} \int_{\R^d} | \nabla (\mathbf{u}_1 - \mathbf{u}_2 ) |^2 \, dx + \hat{C}   \| \mathbf{u}_1 - \mathbf{u}_2   \|_{L^2(\R^d)}^2.
\end{align*}
For $J_5$ we obtain that
\begin{align*}
J_5
\leq& \frac{\sigma}{16} \int_{\R^d} | \nabla (\mathbf{u}_1 - \mathbf{u}_2 ) |^2 \, dx \\
& + \hat{C} \| u^{\sigma} \|_{L^{\infty}(\R^d)}^2 \Big\| \int_0^1  \nabla \Big(  \big(  p' ( V^{\varepsilon_p}*u^{\sigma} +z V^{\varepsilon_p}*\mathbf{v}_1 ) - p' ( V^{\varepsilon_p}*u^{\sigma} +z V^{\varepsilon_p}*\mathbf{v}_2 )  \big) V^{\varepsilon_p}* \mathbf{u}_2 \Big)  \, dz \Big\|_{L^2(\R^d)}^2 \\
\leq& \frac{\sigma}{16} \int_{\R^d} | \nabla (\mathbf{u}_1 - \mathbf{u}_2 ) |^2 \, dx \\
& + \hat{C} \Big\| \int_0^1  \nabla \int_0^1     p'' \big( V^{\varepsilon_p}*u^{\sigma} +z \tau  V^{\varepsilon_p}*\mathbf{v}_1 + z(1-\tau)V^{\varepsilon_p}*\mathbf{v}_2 \big)  zV^{\varepsilon_p}*( \mathbf{v}_1 - \mathbf{v}_2  )  \, d \tau \, dz \Big\|_{L^2(\R^d)}^2 \\
\leq&  \frac{\sigma}{16} \int_{\R^d} | \nabla (\mathbf{u}_1 - \mathbf{u}_2 ) |^2 \, dx \\
& + \hat{C} \Big\| \int_0^1  \int_0^1     p''' \big( V^{\varepsilon_p}*u^{\sigma} +z \tau  V^{\varepsilon_p}*\mathbf{v}_1 + z(1-\tau)V^{\varepsilon_p}*\mathbf{v}_2 \big)  \\
& \ \ \ \ \ \ \ \ \ \ \ \ \ \ \cdot \nabla V^{\varepsilon_p}* \big( u^{\sigma} +z \tau  \mathbf{v}_1 + z(1-\tau)\mathbf{v}_2 \big) zV^{\varepsilon_p}*( \mathbf{v}_1 - \mathbf{v}_2  )  \, d \tau \, dz \Big\|_{L^2(\R^d)}^2 \\
&+ \hat{C} \Big\| \int_0^1  \int_0^1     p'' \big( V^{\varepsilon_p}*u^{\sigma} +z \tau  V^{\varepsilon_p}*\mathbf{v}_1 + z(1-\tau)V^{\varepsilon_p}*\mathbf{v}_2 \big)  z\nabla V^{\varepsilon_p}*( \mathbf{v}_1 - \mathbf{v}_2  )  \, d \tau \, dz \Big\|_{L^2(\R^d)}^2 \\
\leq& \frac{\sigma}{16} \int_{\R^d} | \nabla (\mathbf{u}_1 - \mathbf{u}_2 ) |^2 \, dx \\
&+ \hat{C} \sup_{0<z, \tau <1} \|  V^{\varepsilon_p}*\nabla  \big( u^{\sigma} +z \tau  \mathbf{v}_1 + z(1-\tau)\mathbf{v}_2   \big)  \|_{L^{\infty}(\R^d)}^2 \| V^{\varepsilon_p} * ( \mathbf{v}_1 - \mathbf{v}_2 ) \|_{L^2(\R^d)}^2 \\
&+ \hat{C}\|\nabla V^{\varepsilon_p}\|_{L^1(\R^d)}^2 \|  \mathbf{v}_1 - \mathbf{v}_2  \|_{L^2(\R^d)}^2\\
\leq& \frac{\sigma}{16} \int_{\R^d} | \nabla (\mathbf{u}_1 - \mathbf{u}_2 ) |^2 \, dx + \hat{C} \|  \mathbf{v}_1 - \mathbf{v}_2  \|_{L^2(\R^d)}^2.
\end{align*}
Combining the estimates for $J_1$-$J_9$, it follows that
\begin{align*}
\frac{d}{dt} \| \mathbf{u}_1 - \mathbf{u}_2  \|_{L^2(\R^d)}^2 + \frac{\sigma}{2} \| \nabla (\mathbf{u}_1 - \mathbf{u}_2)\|_{L^2(\R^d)}^2 \leq \hat{C} \| \mathbf{u}_1 - \mathbf{u}_2  \|_{L^2(\R^d)}^2  + \hat{C} \| \mathbf{v}_1 - \mathbf{v}_2  \|_{L^2(\R^d)}^2.
\end{align*}
By Gr\"{o}nwall's inequality, we have
\begin{align*}
\| (\mathbf{u}_1 - \mathbf{u}_2)(t) \|_{L^2(\R^d)}^2 &\leq e^{\int_0^t  \hat{C}  \, d{\tilde s} } \hat{C} \int_0^t \| \mathbf{v}_1 - \mathbf{v}_2 \|_{L^2(\R^d)}^2   \, d{\tilde s} \leq e^{ \hat{C} T_*} \hat{C} T_* \| \mathbf{v}_1 - \mathbf{v}_2 \|_{ L^{\infty} ( 0, T_* ; L^2(\R^d) ) }^2,\;\;\;\forall t\in[0,T_*].
\end{align*}
We can take $T_*(\varepsilon_p)$ small enough such that $e^{ \hat{C} T_*} \hat{C} T_*<1$.
Therefore, the operator $\mathcal{T}$ possesses a unique fixed point ${\mathbf{u}}$ such that $\mathcal{T}({\mathbf{u}})={\mathbf{u}}$. Thus, we find a unique solution ${\mathbf{u}} \in Y$ on $[0,T_*]$ to the following system:
\begin{align}
\label{non_loc_fixed_point}
\partial_t {\mathbf{u}} - \sigma \Delta {\mathbf{u}}  =& \nabla \cdot \Big( ({\mathbf{u}} +u^{\sigma})\nabla (p_{\lambda}-p)(V^{\varepsilon_p}*({\mathbf{u}} +u^{\sigma})) + {\mathbf{u}} \nabla p ( V^{\varepsilon_p}*({\mathbf{u}} +u^{\sigma}) ) \\
& + u^{\sigma} \nabla \int_0^1 p' ( V^{\varepsilon_p}*u^{\sigma} +z V^{\varepsilon_p}*{\mathbf{u}} ) V^{\varepsilon_p}*{\mathbf{u}} \, dz
+u^{\sigma}\big(\nabla p ( V^{\varepsilon_p}*u^{\sigma}) - \nabla p(u^{\sigma})\big) \Big) \notag \\
&- \nabla \cdot \Big(  {\mathbf{u}} \nabla \Phi * V^{\varepsilon_k}*( {\mathbf{u}} + u^{\sigma} )  +u^{\sigma}\nabla \Phi*V^{\varepsilon_k}*{\mathbf{u}}  + u^{\sigma} ( \nabla \Phi*V^{\varepsilon_k}*u^{\sigma}   -\nabla \Phi * u^{\sigma} ) \Big), \notag\\
{\mathbf{u}}(0,x)=0.\;\;\;& \notag
\end{align}

\subsection{Global well-posedness of perturbation equations.}
We extend the local solution obtained in the last subsection globally, by doing the following estimates.
For any $t\in (0,T_*)$, we have
\begin{align*}
&\frac{1}{2} \frac{d}{dt} \int_{\R^d} |D^{\alpha} {\mathbf{u}} |^2 \, dx + \sigma \int_{\R^d} |\nabla D^{\alpha} {\mathbf{u}} |^2 \, dx \\
=& - \int_{\R^d} D^{\alpha} \big(  ({\mathbf{u}} +u^{\sigma})\nabla (p_{\lambda}-p)(V^{\varepsilon_p}*({\mathbf{u}} +u^{\sigma})) \big) \cdot \nabla D^{\alpha} {\mathbf{u}}  \, dx\\
& - \int_{\R^d}   D^{\alpha} \big( {\mathbf{u}} \nabla p ( V^{\varepsilon_p}*({\mathbf{u}} +u^{\sigma}) ) \big) \cdot \nabla D^{\alpha} {\mathbf{u}} \, dx   \\
& - \int_{\R^d}   D^{\alpha} \big(  u^{\sigma} \nabla \int_0^1 p' ( V^{\varepsilon_p}*u^{\sigma} +z V^{\varepsilon_p}*{\mathbf{u}} ) V^{\varepsilon_p}*{\mathbf{u}} \, dz  \big) \cdot \nabla D^{\alpha} {\mathbf{u}} \, dx\\
&   - \int_{\R^d}   D^{\alpha} \big( u^{\sigma}(\nabla p ( V^{\varepsilon_p}*u^{\sigma}) - \nabla p(u^{\sigma})) \big)   \cdot \nabla D^{\alpha} {\mathbf{u}} \, dx\\
& +  \int_{\R^d}   D^{\alpha} \big( {\mathbf{u}} \nabla \Phi * V^{\varepsilon_k}*( {\mathbf{u}} + u^{\sigma} ) \big)
\cdot \nabla D^{\alpha} {\mathbf{u}} \, dx \\
& + \int_{\R^d}   D^{\alpha} \big( u^{\sigma}\nabla \Phi*V^{\varepsilon_k}*{\mathbf{u}} \big)   \cdot \nabla D^{\alpha} {\mathbf{u}} \, dx  \\
& + \int_{\R^d}   D^{\alpha} \big(  u^{\sigma} ( \nabla \Phi*V^{\varepsilon_k}*u^{\sigma}   -\nabla \Phi * u^{\sigma} )  \big)   \cdot \nabla D^{\alpha} {\mathbf{u}} \, dx \\
=&: K_1 + K_2 + K_3 + K_4 + K_5 + K_6 + K_7.
\end{align*}
Since $\| V^{\varepsilon_p} * ({\mathbf{u}} +u^{\sigma})\|_{L^{\infty}((0,T_*) \times \R^d) }$ is bounded, we can choose $\lambda$ small enough in order to ensure that $K_1=0$.
For $K_2$ we get
\begin{align*}
K_2
\leq& \frac{\sigma}{13} \int_{\R^d} |\nabla D^{\alpha} {\mathbf{u}} |^2 \, dx + C \| {\mathbf{u}}  \|_{L^{\infty}(\R^d)}^2 \| D^{\alpha}  \nabla p ( V^{\varepsilon_p}*({\mathbf{u}} +u^{\sigma}) )   \|_{L^2(\R^d)}^2 \\
&+ C \| D^{\alpha} {\mathbf{u}} \|_{L^2(\R^d)}^2 \|  \nabla p ( V^{\varepsilon_p}*({\mathbf{u}} +u^{\sigma}) )   \|_{L^{\infty}(\R^d)}^2 \\
\leq&  \frac{\sigma}{13} \int_{\R^d} |\nabla D^{\alpha} {\mathbf{u}} |^2 \, dx + C \| {\mathbf{u}}  \|_{L^{\infty}(\R^d)}^2 ( 1 +   \| \nabla D^{\alpha} {\mathbf{u}} \|_{L^2(\R^d)}^2 ) + C \|  D^{\alpha} {\mathbf{u}} \|_{L^2(\R^d)}^2 \\
\leq&  \frac{\sigma}{13} \int_{\R^d} |\nabla D^{\alpha} {\mathbf{u}} |^2 \, dx
+ C \| {\mathbf{u}}  \|_{H^s(\R^d)}^2  + C (\varepsilon_k + \varepsilon_p) \|  \nabla D^{\alpha} {\mathbf{u}} \|_{L^2(\R^d)}^2,
\end{align*}
where we used that $ \| {\mathbf{u}} \|_{ L^{\infty} ( 0, T_* ; H^s(\R^d) ) }^2 \leq \varepsilon_k + \varepsilon_p $.
Then we take $ \varepsilon_k + \varepsilon_p $
 small enough such that
 \begin{align*}
 K_2 \leq \frac{\sigma}{12} \|  \nabla {\mathbf{u}} \|_{H^s(\R^d)}^2 + C \|  {\mathbf{u}} \|_{H^s(\R^d)}^2.
 \end{align*}
Next, we analyze $K_3$.
\begin{align*}
K_3 =&  - \int_{\R^d}  \int_0^1  u^{\sigma} \nabla D^{\alpha} \big(  p' ( V^{\varepsilon_p}*u^{\sigma} +z V^{\varepsilon_p}*{\mathbf{u}} ) V^{\varepsilon_p}*{\mathbf{u}} \big)   \cdot \nabla D^{\alpha} {\mathbf{u}} \, dz \, dx \\
&- \int_{\R^d}  \int_0^1 \big( D^{\alpha} \big( u^{\sigma} \nabla  \big(  p' ( V^{\varepsilon_p}*u^{\sigma} +z V^{\varepsilon_p}*{\mathbf{u}} ) V^{\varepsilon_p}*{\mathbf{u}} \big)   \big) - u^{\sigma}\nabla D^{\alpha}\big( p' ( V^{\varepsilon_p}*u^{\sigma} +z V^{\varepsilon_p}*{\mathbf{u}} ) V^{\varepsilon_p}*{\mathbf{u}} \big)  \big) \cdot \nabla  D^{\alpha} {\mathbf{u}}  \, dz \, dx \\
=&: K_{31} + K_{32}.
\end{align*}
For $K_{32}$,
\begin{align*}
K_{32} \leq& \|  \nabla D^{\alpha} {\mathbf{u}} \|_{L^2(\R^d)} \sup_{0< z < 1} \| \nabla u^{\sigma} \|_{L^{\infty}(\R^d)} \| D^{\alpha-1} \nabla \big( p' ( V^{\varepsilon_p}*u^{\sigma} +z V^{\varepsilon_p}*{\mathbf{u}} ) V^{\varepsilon_p}*{\mathbf{u}}   \big)   \|_{L^2(\R^d)}\\
&+ \|  \nabla D^{\alpha} {\mathbf{u}} \|_{L^2(\R^d)} \sup_{0< z < 1} \|  D^{\alpha} u^{\sigma} \|_{L^2(\R^d)} \|  \nabla \big( p' ( V^{\varepsilon_p}*u^{\sigma} +z V^{\varepsilon_p}*{\mathbf{u}} ) V^{\varepsilon_p}*{\mathbf{u}}   \big)   \|_{L^{\infty}(\R^d)} \\
\leq& \frac{\sigma}{36}  \|  \nabla D^{\alpha} {\mathbf{u}} \|_{L^2(\R^d)}^2 + C  \|   {\mathbf{u}} \|_{H^s(\R^d)}^2.
\end{align*}
Now we handle $K_{31}$.
\begin{align*}
K_{31} =& - \int_{\R^d}  \int_0^1  u^{\sigma}  p' ( V^{\varepsilon_p}*u^{\sigma} +z V^{\varepsilon_p}*{\mathbf{u}} ) V^{\varepsilon_p}*\nabla D^{\alpha}{\mathbf{u}}  \cdot \nabla D^{\alpha}{\mathbf{u}}   \, dz \, dx\\
&- \int_{\R^d}  \int_0^1  u^{\sigma} \Big( \nabla D^{\alpha} \big( p' ( V^{\varepsilon_p}*u^{\sigma} +z V^{\varepsilon_p}*{\mathbf{u}} ) V^{\varepsilon_p}*  {\mathbf{u}} \big)    -  p' ( V^{\varepsilon_p}*u^{\sigma} +z V^{\varepsilon_p}*{\mathbf{u}} ) V^{\varepsilon_p} *\nabla D^{\alpha}  {\mathbf{u}}   \Big)  \cdot \nabla D^{\alpha}{\mathbf{u}}   \, dz \, dx\\
=&: K_{311}+ K_{312}.
\end{align*}
For the second term $K_{312}$,
\begin{align*}
K_{312} \leq& \| \nabla D^{\alpha}{\mathbf{u}} \|_{L^2(\R^d)} \| u^{\sigma} \|_{L^{\infty}(\R^d)}  \| D^{\alpha} V^{\varepsilon_p}* {\mathbf{u}}
\|_{L^2(\R^d)}   \sup_{0\leq z\leq 1} \|\nabla  p' ( V^{\varepsilon_p}*u^{\sigma} +z V^{\varepsilon_p}*{\mathbf{u}} ) \|_{L^{\infty}(\R^d)}      \\
&+ \| \nabla D^{\alpha}{\mathbf{u}} \|_{L^2(\R^d)} \| u^{\sigma} \|_{L^{\infty}(\R^d)} \|  V^{\varepsilon_p}* {\mathbf{u}}
\|_{L^{\infty}(\R^d)}   \sup_{0\leq z\leq 1} \|\nabla D^{\alpha} p' ( V^{\varepsilon_p}*u^{\sigma} +z V^{\varepsilon_p}*{\mathbf{u}} ) \|_{L^2(\R^d)}   \\
\leq& \| \nabla D^{\alpha}{\mathbf{u}} \|_{L^2(\R^d)} \| u^{\sigma} \|_{L^{\infty}(\R^d)}  \| D^{\alpha} V^{\varepsilon_p}* {\mathbf{u}}   \|_{L^2(\R^d)}   \sup_{0\leq z\leq 1} \|\nabla  p' ( V^{\varepsilon_p}*u^{\sigma} +z V^{\varepsilon_p}*{\mathbf{u}} ) \|_{L^{\infty}(\R^d)}      \\
&+ C\| \nabla D^{\alpha}{\mathbf{u}} \|_{L^2(\R^d)} \| u^{\sigma} \|_{L^{\infty}(\R^d)} \|  V^{\varepsilon_p}* {\mathbf{u}}
 \|_{L^{\infty}(\R^d)}   (1+ \| D^{\alpha} \nabla {\mathbf{u}} \|_{L^2(\R^d)})\\
\leq& \frac{\sigma}{37} \| \nabla {\mathbf{u}}  \|_{H^s(\R^d)}^2 + C \| {\mathbf{u}} \|_{H^s(\R^d)}^2  + C (\varepsilon_k +\varepsilon_p)\| \nabla{\mathbf{u}}  \|_{H^s(\R^d)}^2.
\end{align*}
We can take $ \varepsilon_k +\varepsilon_p $ small enough such that
\begin{align*}
K_{312} \leq \frac{\sigma}{36}  \| \nabla {\mathbf{u}}  \|_{H^s(\R^d)}^2 + C  \| {\mathbf{u}}  \|_{H^s(\R^d)}^2.
\end{align*}
For $K_{311}$,
\begin{align*}
K_{311}
=&  - \int_{\R^d}  \int_0^1  W^{\varepsilon_p}* \big( u^{\sigma}  p' ( V^{\varepsilon_p}*u^{\sigma} +z V^{\varepsilon_p}*{\mathbf{u}} ) \nabla D^{\alpha}{\mathbf{u}} \big) \cdot W^{\varepsilon_p}*\nabla D^{\alpha}{\mathbf{u}}      \, dz \, dx \\
=& - \int_{\R^d}  \int_0^1   u^{\sigma}  p' ( V^{\varepsilon_p}*u^{\sigma} +z V^{\varepsilon_p}*{\mathbf{u}} ) \big| W^{\varepsilon_p}*\nabla D^{\alpha}{\mathbf{u}} \big|^2  \, dz \, dx \\
&- \int_{\R^d}  \int_0^1   \Big( W^{\varepsilon_p}* \big( u^{\sigma}  p' ( V^{\varepsilon_p}*u^{\sigma} +z V^{\varepsilon_p}*{\mathbf{u}} ) \nabla D^{\alpha}{\mathbf{u}} \big)   - u^{\sigma}  W^{\varepsilon_p}* \big( p' ( V^{\varepsilon_p}*u^{\sigma} +z V^{\varepsilon_p}*{\mathbf{u}} ) \nabla D^{\alpha}{\mathbf{u}}  \big)   \Big)\\
&\qquad\qquad \cdot W^{\varepsilon_p}*\nabla D^{\alpha}{\mathbf{u}}  \, dz \, dx \\
&- \int_{\R^d}  \int_0^1  u^{\sigma} \big( W^{\varepsilon_p}* \big( p' ( V^{\varepsilon_p}*u^{\sigma} +z V^{\varepsilon_p}*{\mathbf{u}} ) \nabla D^{\alpha}{\mathbf{u}} \big)   -   p' ( V^{\varepsilon_p}*u^{\sigma} +z V^{\varepsilon_p}*{\mathbf{u}} )   W^{\varepsilon_p}*\nabla D^{\alpha}{\mathbf{u}}    \big) \cdot W^{\varepsilon_p}*\nabla D^{\alpha}{\mathbf{u}}  \, dz \, dx \\
\leq& 0+ K_{3111} + K_{3112}.
\end{align*}
Using Lemma \ref{lemma_non_loc}, we obtain that
\begin{align*}
K_{3111}&\leq \| \nabla D^{\alpha} {\mathbf{u}}  \|_{L^2(\R^d)} \sup_{0< z < 1} \|   W^{\varepsilon_p}* \big( u^{\sigma}  p' ( V^{\varepsilon_p}*u^{\sigma} +z V^{\varepsilon_p}*{\mathbf{u}} ) \nabla D^{\alpha}{\mathbf{u}} \big)  - u^{\sigma}  W^{\varepsilon_p}* \big( p' ( V^{\varepsilon_p}*u^{\sigma} +z V^{\varepsilon_p}*{\mathbf{u}} ) \nabla D^{\alpha}{\mathbf{u}}  \big)   \|_{L^2(\R^d)}\\
&\leq C  \| \nabla D^{\alpha} {\mathbf{u}}  \|_{L^2(\R^d)} \varepsilon_p \| \nabla u^{\sigma} \|_{L^{\infty}(\R^d)} \sup_{0< z < 1} \|  p' ( V^{\varepsilon_p}*u^{\sigma} +z V^{\varepsilon_p}*{\mathbf{u}} ) \nabla D^{\alpha}{\mathbf{u}}   \|_{L^2(\R^d)}\\
&\leq C \varepsilon_p  \| \nabla D^{\alpha} {\mathbf{u}}  \|_{L^2(\R^d)}^2,
\end{align*}
and
\begin{align*}
K_{3112}&\leq \| W^{\varepsilon_p} * \nabla D^{\alpha} {\mathbf{u}}  \|_{L^2(\R^d)} \sup_{0< z < 1} \|   W^{\varepsilon_p}* \big(  p' ( V^{\varepsilon_p}*u^{\sigma} +z V^{\varepsilon_p}*{\mathbf{u}} ) \nabla D^{\alpha}{\mathbf{u}} \big)   -  p' ( V^{\varepsilon_p}*u^{\sigma} +z V^{\varepsilon_p}*{\mathbf{u}} ) W^{\varepsilon_p} * \nabla D^{\alpha}{\mathbf{u}}    \|_{L^2(\R^d)}\\
&\leq C  \| \nabla D^{\alpha} {\mathbf{u}}  \|_{L^2(\R^d)} \varepsilon_p \|  \nabla D^{\alpha} {\mathbf{u}}  \|_{L^2(\R^d)} \sup_{0< z < 1} \| \nabla  p' ( V^{\varepsilon_p}*u^{\sigma} +z V^{\varepsilon_p}*{\mathbf{u}} )  \|_{L^{\infty}(\R^d)}\\
&\leq C \varepsilon_p  \| \nabla D^{\alpha} {\mathbf{u}}  \|_{L^2(\R^d)}^2.
\end{align*}
Taking $\varepsilon_p$ small enough such that $C \varepsilon_p \leq \frac{\sigma}{36}$, it follows that
\begin{align*}
K_3 \leq \frac{\sigma}{12} \| \nabla {\mathbf{u}}   \|_{H^s(\R^d)}^2 + C \| {\mathbf{u}}   \|_{H^s(\R^d)}^2 .
\end{align*}
For $K_4$,
\begin{align*}
K_4
\leq& \frac{\sigma}{12} \int_{\R^d} |\nabla D^{\alpha} {\mathbf{u}} |^2 \, dx
+ C \| D^{\alpha}\big(u^{\sigma} ( \nabla p ( V^{\varepsilon_p}*u^{\sigma}) - \nabla p (u^{\sigma}) ) \big)  \|_{L^2(\R^d)}^2 \\
\leq&  \frac{\sigma}{12} \int_{\R^d} |\nabla D^{\alpha} {\mathbf{u}} |^2 \, dx \\
&+ C \| D^{\alpha} u^{\sigma}\|_{L^2(\R^d)}^2 \Big\| \nabla \int_0^1 p'(zV^{\varepsilon_p} * u^{\sigma} + (1-z)u^{\sigma} )(V^{\varepsilon_p}*u^{\sigma} - u^{\sigma}) \, dz \Big\|_{L^{\infty}(\R^d)}^2 \\
&+ C \| u^{\sigma}\|_{L^{\infty}(\R^d)}^2 \Big\|D^\alpha \nabla \int_0^1 p'(zV^{\varepsilon_p} * u^{\sigma} + (1-z)u^{\sigma} )(V^{\varepsilon_p}*u^{\sigma} - u^{\sigma}) \, dz \Big\|_{L^2(\R^d)}^2\\
\leq& \frac{\sigma}{12} \int_{\R^d} |\nabla D^{\alpha}{\mathbf{u}} |^2 \, dx + C \| V^{\varepsilon_p}*u^{\sigma} - u^{\sigma}  \|_{H^s(\R^d)}^2 \\
&+ C \int_0^1 \Big( \| D^{\alpha+1}  p'(zV^{\varepsilon_p} * u^{\sigma} + (1-z)u^{\sigma}) \|_{L^2(\R^d)}^2 \|  V^{\varepsilon_p}*u^{\sigma} - u^{\sigma} \|_{L^{\infty}(\R^d)}^2  \\
&\ \ \ \ \ \ \ \ \ \ \ \ \  + \|  p'(zV^{\varepsilon_p} * u^{\sigma} + (1-z)u^{\sigma}) \|_{L^{\infty}(\R^d)}^2 \|  D^{\alpha+1} ( V^{\varepsilon_p}*u^{\sigma} - u^{\sigma} ) \|_{L^2(\R^d)}^2  \Big) \, dz \\
\leq&  \frac{\sigma}{12} \int_{\R^d} |\nabla D^{\alpha} {\mathbf{u}} |^2 \, dx + C \| V^{\varepsilon_p}*u^{\sigma} - u^{\sigma}  \|_{H^{s+1}(\R^d)}^2 \\
\leq&  \frac{\sigma}{12} \int_{\R^d} |\nabla D^{\alpha} {\mathbf{u}} |^2 \, dx + C \varepsilon_p^2.
\end{align*}
Now, we estimate $K_5$.
\begin{align*}
K_5
\leq &\frac{\sigma}{12} \int_{\R^d} |\nabla D^{\alpha} {\mathbf{u}} |^2 \, dx
+ C \|  D^{\alpha} \big( {\mathbf{u}} \nabla \Phi * V^{\varepsilon_k}* {\mathbf{u}} \big) \|_{L^2(\R^d)}^2 + C \| D^{\alpha} \big( {\mathbf{u}} \nabla \Phi * V^{\varepsilon_k}* u^{\sigma}  \big) \|_{L^2(\R^d)}^2 \\
\leq& \frac{\sigma}{12} \int_{\R^d} |\nabla D^{\alpha} {\mathbf{u}} |^2 \, dx + C \|  D^{\alpha}  {\mathbf{u}} \|_{L^2(\R^d)}^2 \|  \nabla \Phi * V^{\varepsilon_k}* {\mathbf{u}}  \|_{L^{\infty}(\R^d)}^2 + C \| {\mathbf{u}}  \|_{L^{\infty}(\R^d)}^2  \| D^{\alpha}  V^{\varepsilon_k}* \nabla \Phi *{\mathbf{u}} \|_{L^2(\R^d)}^2 \\
&+ C \|  D^{\alpha}  {\mathbf{u}} \|_{L^2(\R^d)}^2 \|  \nabla \Phi * V^{\varepsilon_k}* u^{\sigma}  \|_{L^{\infty}(\R^d)}^2 + C \| {\mathbf{u}}  \|_{L^{\infty}(\R^d)}^2  \| D^{\alpha}  V^{\varepsilon_k}* \nabla \Phi *  u^{\sigma} \|_{L^2(\R^d)}^2 \\
\leq& \frac{\sigma}{12} \int_{\R^d} |\nabla D^{\alpha} {\mathbf{u}} |^2 \, dx + C \| {\mathbf{u}} \|_{H^s(\R^d)}^2.
\end{align*}
Similarly, $K_6$ and $K_7$ can be bounded by
\begin{align*}
K_6+K_7
\leq   \frac{\sigma}{6} \int_{\R^d} |\nabla D^{\alpha} \tilde{\mathbf{u}} |^2 \, dx + C \| {\mathbf{u}} \|_{H^s(\R^d)}^2.
\end{align*}
From estimates of $K_1 - K_7$, we infer that
\begin{align*}
\frac{d}{dt} \| {\mathbf{u}} \|_{H^s(\R^d)}^2 + \frac{\sigma}{2} \| \nabla {\mathbf{u}} \|_{H^s(\R^d)}^2 \leq C \| {\mathbf{u}} \|_{H^s(\R^d)}^2
+ C( \varepsilon_k + \varepsilon_p )^2.
\end{align*}
By Gr\"{o}nwall's inequality, it follows that
\begin{align*}
\| {\mathbf{u}}(t) \|_{H^s(\R^d)}^2 \leq e^{CT} CT ( \varepsilon_k + \varepsilon_p )^2, \;\;\; {\rm{for\;\;all\;\;}}  0\leq t\leq  T_*\leq T.
\end{align*}
Therefore,  we can construct a solution $\mathbf{u}$ to the following partial differential equations
\begin{align*}
\partial_t \mathbf{u} - \sigma \Delta \mathbf{u}  =& \nabla \cdot \Big( (\mathbf{u} +u^{\sigma})\nabla p_{\lambda}(V^{\varepsilon_p}*(\mathbf{u} +u^{\sigma})) - u^{\sigma} \nabla p (u^{\sigma}) \Big)\\
&- \nabla \cdot \Big(  (\mathbf{u} +u^{\sigma}) \mathbf{u} \nabla \Phi * V^{\varepsilon_k}*( \mathbf{u} + u^{\sigma} )  - u^{\sigma} \nabla (\Phi* u^{\sigma})\Big), \\
\mathbf{u}(0,x)=0\quad&
\end{align*}
on $(0,T)$ for any $T < \infty $.
Moreover,
\begin{align*}
\|  \mathbf{u} \|_{L^{\infty}(0,T; H^s(\R^d))}^2 \leq C (\varepsilon_k + \varepsilon_p)^2, \ s>\frac{d}{2} +2.
\end{align*}
By Sobolev's embedding theorem, we get
\begin{align*}
\|u^{\sigma} - u^{\varepsilon, \sigma}\|_{L^{\infty} (0,T; W^{2, \infty} (\R^d) ) } \leq  C\| \mathbf{u} \|_{L^{\infty} (0,T; H^s (\R^d) )} \leq C(\varepsilon_k + \varepsilon_p).
\end{align*}
\hfill\qedsymbol

\section{Limit $\sigma\rightarrow 0$ and the well-posedness of  \eqref{Keller_Segel_classical} }\label{section_u}

In this section, we prove the compactness of $(u^\sigma)_\sigma$ and proceed the compactness argument on all terms in \eqref{generalized_equation_u_sigma}
to complete the proof of Theorem \ref{lsigma}. Recall the estimates in Theorem \ref{theorem_u_sigma}
\begin{align}\label{ls1}
\|u^\sigma\|_{L^q((0,T)\times \R^d)}\leq C,\;\;\;\; \forall q\in [1,\infty],
\end{align}
where $C$ appeared in this proof is a positive constant independent of $\sigma$.
We multiply the system \eqref{generalized_equation_u_sigma} by $u^\sigma$ and integrate it over $\R^d$ to get
\begin{align*}
&\frac{1}{2}\frac{d}{dt}\int_{\R^d}|u^\sigma|^2\,dx+\sigma\int_{\R^d}|\nabla u^\sigma|^2\,dx
+\frac{4m}{(m+1)^2}\int_{\R^d}|\nabla (u^\sigma)^\frac{m+1}{2}|^2\,dx\\
=&\int_{\R^d}u^\sigma\nabla \Phi*u^\sigma\cdot\nabla u^\sigma\,dx
=-\frac{1}{2}\int_{\R^d}\Delta\Phi*u^\sigma(u^\sigma)^2\,dx
=\frac{1}{2}\int_{\R^d}(u^\sigma)^3\,dx\leq C.
\end{align*}
Therefore,
\begin{align}
&\sqrt{\sigma}\|\nabla u^\sigma\|_{L^2((0,T)\times\R^d)}\leq C,\label{ls2}\\
&\|\nabla(u^\sigma)^\frac{m+1}{2}\|_{L^2((0,T)\times\R^d)}\leq C.\label{ls3}
\end{align}
From \eqref{ls1} and \eqref{ls3}, we have
\begin{align}\label{ls7}
\|\nabla(u^\sigma)^m\|_{L^2((0,T)\times\R^d)}=\frac{2m}{m+1}\|(u^\sigma)^\frac{m-1}{2}\nabla (u^\sigma)^\frac{m+1}{2}\|_{L^2((0,T)\times\R^d)}
\leq C.
\end{align}
Combining above estimates with the system \eqref{generalized_equation_u_sigma}, we infer
\begin{align}\label{ls4}
\|\partial_t u^\sigma\|_{L^2(0,T;W^{-1,2}(\R^d))}\leq C.
\end{align}
Inequalities \eqref{ls7} and \eqref{ls4} allow us to use Theorem 3 of \citep{CJJ} (the nonlinear version of Aubin-Lions lemma) in growing $d$-dimensional balls $B_R$
and use diagonal argument to get
\begin{align}\label{ls6}
u^\sigma\rightarrow u\;\;\;{\rm{in}}\;\;\;L^{2m}(0,T;L^{2m}(\R^d)).
\end{align}
It is easy to see from \eqref{ls1} that
\begin{align}\label{ls9}
u^\sigma\overset{*}\rightharpoonup u\;\;\;{\rm{in}}\;\;\;L^l((0,T)\times\R^d),\;\;l\in (1,\infty].
\end{align}
By \eqref{ls7} and \eqref{ls6}, we obtain that
\begin{align}\label{ls10}
\nabla(u^\sigma)^m\rightharpoonup \nabla u^m\;\;\;{\rm{in}}\;\;\; L^2((0,T)\times\R^d).
\end{align}
With the help of \eqref{ls6}-\eqref{ls10}, we pass to the limit $\sigma\rightarrow 0$ in the weak formulation of \eqref{generalized_equation_u_sigma} to complete the proof of Theorem \ref{lsigma}.
\hfill\qedsymbol

\medskip
\indent
{\bf Acknowledgements:}
Yue Li would like to thank Chair in Applied Analysis of the University of Mannheim for hosting her one-year scientific visit as an exchange doctoral student. Yue Li is supported by NSFC (Grant No. 12071212).


\begin{thebibliography}{99}


\bibitem{arumugam2021keller}
Gurusamy Arumugam and Jagmohan Tyagi.
\newblock Keller-Segel chemotaxis models: A review.
\newblock {\em Acta Applicandae Mathematicae}, 171(1):1-82, 2021.

\bibitem{blanchet2009critical}
Adrien Blanchet, Jos{\'e} A Carrillo, and Philippe Lauren{\c{c}}ot.
\newblock Critical mass for a patlak--keller--segel model with degenerate
  diffusion in higher dimensions.
\newblock {\em Calc. Var. Partial Differential Equations}, 35 (2): 133-168, 2009.

\bibitem{blanchetinfinite}
Adrien Blanchet, Jos{\'e} A Carrillo and Nader Masmoudi.
\newblock Infinite time aggregation for the critical
  Patlak-Keller-Segel model in $\mathbb{R}^2$,
\newblock {\em Comm. Pure Appl. Math}, 61 (10): 1449-1481, 2008.


\bibitem{blanchet2006two}
Adrien Blanchet, Jean Dolbeault, and Beno{\^\i}t Perthame.
\newblock Two-dimensional keller-segel model: Optimal critical mass and
  qualitative properties of the solutions.
\newblock {\em Electron. J. Differential Equations}, No. 44, 32 pp, 2006.

\bibitem{CGHL}
Li Chen, Veniamin Gvozdik, Alexandra Holzinger and Yue Li.
\newblock Rigorous Derivation of the Degenerate Parabolic-Elliptic Keller-Segel System from A Moderately Interacting Stochastic Particle System.\\
Part II Mean-field limit.
\newblock{\em submitted.}

\bibitem{chen2012multidimensional}
Li~Chen, Jian-Guo Liu, and Jinhuan Wang.
\newblock Multidimensional degenerate keller--segel system with critical
  diffusion exponent 2n/(n+2).
\newblock {\em SIAM J. Math. Anal.}, 44 (2): 1077-1102, 2012.



\bibitem{CJJ}
Xiuqing Chen, Ansgar J\"{u}ngel and Jian-Guo Liu.
\newblock A note on Aubin-Lions-Dubinskii lemmas.
\newblock {\em Acta Appl. Math.}, 133: 33-43, 2014.



\bibitem{ishida2013gradient}
Sachiko Ishida, Yusuke Maeda, and Tomomi Yokota.
\newblock Gradient estimate for solutions to quasilinear non-degenerate
  keller-segel systems on $\mathbb{R}^n$.
\newblock {\em Discrete Contin. Dyn. Syst. Ser. B}, 18 (10): 2537-2568, 2013.

\bibitem{keller1970initiation}
Evelyn F. Keller and Lee A. Segel.
\newblock Initiation of slime mold aggregation viewed as an instability.
\newblock {\em J. Theoret. Biol.}, 26 (3): 399-415, 1970.



\bibitem{keller1971model}
Evelyn F. Keller and Lee A. Segel.
\newblock Model for chemotaxis.
\newblock {\em J. Theor. Biol.}, 30(2):225-234, 1971.

\bibitem{luckhaus2006large}
Stephan Luckhaus and Yoshie Sugiyama.
\newblock Large time behavior of solutions in super-critical cases to
  degenerate keller-segel systems.
\newblock {\em Math. Model. Numer. Anal.}, 40 (3): 597-621, 2006.

\bibitem{patlak1953random}
Clifford~S Patlak.
\newblock Random walk with persistence and external bias.
\newblock {\em Bull. Math. Biophys.}, 15 (3): 311-338, 1953.

\bibitem{sugiyama2006global}
Yoshie Sugiyama.
\newblock Global existence in sub-critical cases and finite time blow-up in
  super-critical cases to degenerate keller-segel systems.
\newblock {\em Differential Integral Equations}, 19 (8): 841-876, 2006.

\bibitem{sugiyama2007application}
Yoshie Sugiyama.
\newblock Application of the best constant of the sobolev inequality to
  degenerate keller-segel models.
\newblock {\em Adv. Differential Equations }, 12 (2): 121-144, 2007.

\bibitem{sugiyama2007time}
Yoshie Sugiyama.
\newblock Time global existence and asymptotic behavior of solutions to
  degenerate quasi-linear parabolic systems of chemotaxis.
\newblock {\em Differential and Integral Equations}, 20 (2): 133-180, 2007.

\bibitem{sugiyama2006global_power}
Yoshie Sugiyama and Hiroko Kunii.
\newblock Global existence and decay properties for a degenerate keller--segel
  model with a power factor in drift term.
\newblock {\em J. Differential Equations}, 227 (1): 333-364, 2006.

\bibitem{wang2016parabolic}
Jinhuan Wang, Li~Chen, and Liang Hong.
\newblock Parabolic elliptic type keller-segel system on the whole space case.
\newblock {\em Discrete Contin. Dyn. Syst.}, 36 (2): 1061-1084, 2016.


\bibitem{wu2006elliptic}
Zhuoqun Wu, Jingxue Yin, and Chunpeng Wang.
\newblock {\em Elliptic parabolic equations}.
\newblock World Scientific, 2006.


\end{thebibliography}
\end{document}